\documentclass[a4paper, british]{amsart}

\usepackage{libertine}
\usepackage[libertine]{newtxmath}

\usepackage{amssymb}
\usepackage{babel}
\usepackage{enumitem}
\usepackage{hyperref}
\usepackage[utf8]{inputenc}
\usepackage{newunicodechar}
\usepackage{mathtools}
\usepackage{varioref}
\usepackage[arrow,curve,matrix]{xy}

\usepackage{colortbl}
\usepackage{graphicx}
\usepackage{tikz}

\usepackage{mathrsfs}

\definecolor{linkred}{rgb}{0.7,0.2,0.2}
\definecolor{linkblue}{rgb}{0,0.2,0.6}

\numberwithin{figure}{section}

\usepackage[hyperpageref]{backref}

\setdescription{labelindent=\parindent, leftmargin=2\parindent}
\setitemize[1]{labelindent=\parindent, leftmargin=2\parindent}
\setenumerate[1]{labelindent=0cm, leftmargin=*, widest=iiii}

\newunicodechar{α}{\ensuremath{\alpha}}
\newunicodechar{β}{\ensuremath{\beta}}
\newunicodechar{χ}{\ensuremath{\chi}}
\newunicodechar{δ}{\ensuremath{\delta}}
\newunicodechar{ε}{\ensuremath{\varepsilon}}
\newunicodechar{Δ}{\ensuremath{\Delta}}
\newunicodechar{η}{\ensuremath{\eta}}
\newunicodechar{γ}{\ensuremath{\gamma}}
\newunicodechar{Γ}{\ensuremath{\Gamma}}
\newunicodechar{ι}{\ensuremath{\iota}}
\newunicodechar{κ}{\ensuremath{\kappa}}
\newunicodechar{λ}{\ensuremath{\lambda}}
\newunicodechar{Λ}{\ensuremath{\Lambda}}
\newunicodechar{ν}{\ensuremath{\nu}}
\newunicodechar{μ}{\ensuremath{\mu}}
\newunicodechar{ω}{\ensuremath{\omega}}
\newunicodechar{Ω}{\ensuremath{\Omega}}
\newunicodechar{π}{\ensuremath{\pi}}
\newunicodechar{Π}{\ensuremath{\Pi}}
\newunicodechar{φ}{\ensuremath{\phi}}
\newunicodechar{Φ}{\ensuremath{\Phi}}
\newunicodechar{ψ}{\ensuremath{\psi}}
\newunicodechar{Ψ}{\ensuremath{\Psi}}
\newunicodechar{ρ}{\ensuremath{\rho}}
\newunicodechar{σ}{\ensuremath{\sigma}}
\newunicodechar{Σ}{\ensuremath{\Sigma}}
\newunicodechar{τ}{\ensuremath{\tau}}
\newunicodechar{θ}{\ensuremath{\theta}}
\newunicodechar{Θ}{\ensuremath{\Theta}}
\newunicodechar{ξ}{\ensuremath{\xi}}
\newunicodechar{Ξ}{\ensuremath{\Xi}}
\newunicodechar{ζ}{\ensuremath{\zeta}}

\newunicodechar{ℓ}{\ensuremath{\ell}}

\newunicodechar{𝔹}{\ensuremath{\bB}}
\newunicodechar{ℂ}{\ensuremath{\bC}}
\newunicodechar{𝔻}{\ensuremath{\bD}}
\newunicodechar{𝔼}{\ensuremath{\bE}}
\newunicodechar{𝔽}{\ensuremath{\bF}}
\newunicodechar{ℕ}{\ensuremath{\bN}}
\newunicodechar{ℙ}{\ensuremath{\bP}}
\newunicodechar{ℚ}{\ensuremath{\bQ}}
\newunicodechar{ℝ}{\ensuremath{\bR}}
\newunicodechar{𝕏}{\ensuremath{\bX}}
\newunicodechar{ℤ}{\ensuremath{\bZ}}
\newunicodechar{𝒜}{\ensuremath{\sA}}
\newunicodechar{ℬ}{\ensuremath{\sB}}
\newunicodechar{𝒞}{\ensuremath{\sC}}
\newunicodechar{𝒟}{\ensuremath{\sD}}
\newunicodechar{ℰ}{\ensuremath{\sE}}
\newunicodechar{ℱ}{\ensuremath{\sF}}
\newunicodechar{𝒢}{\ensuremath{\sG}}
\newunicodechar{ℋ}{\ensuremath{\sH}}
\newunicodechar{𝒥}{\ensuremath{\sJ}}
\newunicodechar{ℒ}{\ensuremath{\sL}}
\newunicodechar{𝒪}{\ensuremath{\sO}}
\newunicodechar{𝒬}{\ensuremath{\sQ}}
\newunicodechar{𝒯}{\ensuremath{\sT}}
\newunicodechar{𝒲}{\ensuremath{\sW}}

\newunicodechar{∂}{\ensuremath{\partial}}
\newunicodechar{∇}{\ensuremath{\nabla}}

\newunicodechar{↺}{\ensuremath{\circlearrowleft}}
\newunicodechar{∞}{\ensuremath{\infty}}
\newunicodechar{⊕}{\ensuremath{\oplus}}
\newunicodechar{⊗}{\ensuremath{\otimes}}
\newunicodechar{•}{\ensuremath{\bullet}}
\newunicodechar{Λ}{\ensuremath{\wedge}}
\newunicodechar{↪}{\ensuremath{\into}}
\newunicodechar{→}{\ensuremath{\to}}
\newunicodechar{↦}{\ensuremath{\mapsto}}
\newunicodechar{⨯}{\ensuremath{\times}}
\newunicodechar{∪}{\ensuremath{\cup}}
\newunicodechar{∩}{\ensuremath{\cap}}
\newunicodechar{⊋}{\ensuremath{\supsetneq}}
\newunicodechar{⊇}{\ensuremath{\supseteq}}
\newunicodechar{⊃}{\ensuremath{\supset}}
\newunicodechar{⊊}{\ensuremath{\subsetneq}}
\newunicodechar{⊆}{\ensuremath{\subseteq}}
\newunicodechar{⊂}{\ensuremath{\subset}}
\newunicodechar{≥}{\ensuremath{\geq}}
\newunicodechar{≠}{\ensuremath{\neq}}
\newunicodechar{≫}{\ensuremath{\gg}}
\newunicodechar{≪}{\ensuremath{\ll}}

\newunicodechar{≤}{\ensuremath{\leq}}
\newunicodechar{∈}{\ensuremath{\in}}
\newunicodechar{∉}{\ensuremath{\not \in}}
\newunicodechar{∖}{\ensuremath{\setminus}}
\newunicodechar{◦}{\ensuremath{\circ}}
\newunicodechar{°}{\ensuremath{^\circ}}
\newunicodechar{…}{\ifmmode\mathellipsis\else\textellipsis\fi}
\newunicodechar{·}{\ensuremath{\cdot}}
\newunicodechar{⋯}{\ensuremath{\cdots}}
\newunicodechar{∅}{\ensuremath{\emptyset}}
\newunicodechar{⇒}{\ensuremath{\Rightarrow}}

\newunicodechar{⁰}{\ensuremath{^0}}
\newunicodechar{¹}{\ensuremath{^1}}
\newunicodechar{²}{\ensuremath{^2}}
\newunicodechar{³}{\ensuremath{^3}}
\newunicodechar{⁴}{\ensuremath{^4}}
\newunicodechar{⁵}{\ensuremath{^5}}
\newunicodechar{⁶}{\ensuremath{^6}}
\newunicodechar{⁷}{\ensuremath{^7}}
\newunicodechar{⁸}{\ensuremath{^8}}
\newunicodechar{⁹}{\ensuremath{^9}}
\newunicodechar{ⁱ}{\ensuremath{^i}}

\newunicodechar{⌈}{\ensuremath{\lceil}}
\newunicodechar{⌉}{\ensuremath{\rceil}}
\newunicodechar{⌊}{\ensuremath{\lfloor}}
\newunicodechar{⌋}{\ensuremath{\rfloor}}

\newunicodechar{≅}{\ensuremath{\cong}}
\newunicodechar{⇔}{\ensuremath{\Leftrightarrow}}

\DeclareFontFamily{OMS}{rsfs}{\skewchar\font'60}
\DeclareFontShape{OMS}{rsfs}{m}{n}{<-5>rsfs5 <5-7>rsfs7 <7->rsfs10 }{}
\DeclareSymbolFont{rsfs}{OMS}{rsfs}{m}{n}
\DeclareSymbolFontAlphabet{\scr}{rsfs}
\DeclareSymbolFontAlphabet{\scr}{rsfs}

\DeclareFontFamily{U}{mathx}{\hyphenchar\font45}
\DeclareFontShape{U}{mathx}{m}{n}{
      <5> <6> <7> <8> <9> <10>
      <10.95> <12> <14.4> <17.28> <20.74> <24.88>
      mathx10
      }{}
\DeclareSymbolFont{mathx}{U}{mathx}{m}{n}
\DeclareFontSubstitution{U}{mathx}{m}{n}
\DeclareMathAccent{\wcheck}{0}{mathx}{"71}

\DeclareMathOperator{\Aut}{Aut}
\DeclareMathOperator{\codim}{codim}

\DeclareMathOperator{\Id}{Id}

\DeclareMathOperator{\img}{img}
\DeclareMathOperator{\Pic}{Pic}
\DeclareMathOperator{\rank}{rank}
\DeclareMathOperator{\Ramification}{Ramification}

\DeclareMathOperator{\reg}{reg}

\DeclareMathOperator{\sEnd}{\sE\negthinspace \mathit{nd}}
\DeclareMathOperator{\sing}{sing}

\DeclareMathOperator{\supp}{supp}
\DeclareMathOperator{\tor}{tor}

\newcommand{\sA}{\scr{A}}
\newcommand{\sB}{\scr{B}}
\newcommand{\sC}{\scr{C}}
\newcommand{\sD}{\scr{D}}
\newcommand{\sE}{\scr{E}}
\newcommand{\sF}{\scr{F}}
\newcommand{\sG}{\scr{G}}
\newcommand{\sH}{\scr{H}}

\newcommand{\sJ}{\scr{J}}

\newcommand{\sL}{\scr{L}}

\newcommand{\sO}{\scr{O}}

\newcommand{\sQ}{\scr{Q}}

\newcommand{\sT}{\scr{T}}

\newcommand{\sW}{\scr{W}}

\newcommand{\cA}{\mathcal A}
\newcommand{\cC}{\mathcal C}
\newcommand{\cD}{\mathcal D}

\newcommand{\cV}{\mathcal V}

\newcommand{\bB}{\mathbb{B}}
\newcommand{\bC}{\mathbb{C}}
\newcommand{\bD}{\mathbb{D}}
\newcommand{\bE}{\mathbb{E}}
\newcommand{\bF}{\mathbb{F}}

\newcommand{\bN}{\mathbb{N}}

\newcommand{\bP}{\mathbb{P}}
\newcommand{\bQ}{\mathbb{Q}}
\newcommand{\bR}{\mathbb{R}}

\newcommand{\bX}{\mathbb{X}}

\newcommand{\bZ}{\mathbb{Z}}

\theoremstyle{plain}
\newtheorem{thm}{Theorem}[section]

\newtheorem{cor}[thm]{Corollary}
\newtheorem{defn}[thm]{Definition}
\newtheorem{fact}[thm]{Fact}
\newtheorem{lem}[thm]{Lemma}

\newtheorem{prop}[thm]{Proposition}

\theoremstyle{remark}

\newtheorem{asswlog}[thm]{Assumption w.l.o.g.}
\newtheorem{claim}[thm]{Claim}
\newtheorem{c-n-d}[thm]{Claim and Definition}
\newtheorem{consequence}[thm]{Consequence}
\newtheorem{construction}[thm]{Construction}

\newtheorem{example}[thm]{Example}

\newtheorem{notation}[thm]{Notation}
\newtheorem{obs}[thm]{Observation}
\newtheorem{rem}[thm]{Remark}

\newtheorem*{rem-nonumber}{Remark}

\newtheorem{warning}[thm]{Warning}

\numberwithin{equation}{thm}

\setlist[enumerate]{label=(\thethm.\arabic*), before={\setcounter{enumi}{\value{equation}}}, after={\setcounter{equation}{\value{enumi}}}}

\newcommand{\into}{\hookrightarrow}

\newcommand{\wtilde}{\widetilde}
\newcommand{\what}{\widehat}

\newcommand\CounterStep{\addtocounter{thm}{1}\setcounter{equation}{0}}

\newcommand{\factor}[2]{\left. \raise 2pt\hbox{$#1$} \right/\hskip -2pt\raise -2pt\hbox{$#2$}}
\newcommand{\Preprint}[1]{#1}
\newcommand{\Publication}[1]{}

\newcommand{\subversionInfo}{}
\newcommand{\svnid}[1]{}
\newcommand{\approvals}[2][Approval]{}
\allowdisplaybreaks

\DeclareMathOperator{\Div}{Div}

\DeclareMathOperator{\Gal}{Gal}

\DeclareMathOperator{\PSU}{PSU}
\DeclareMathOperator{\refl}{refl}

\newcommand{\pCVHS}{{\sf p$ℂ$VHS}{}}

\theoremstyle{plain}

\newtheorem{thmDef}[thm]{Theorem and Definition}
\newtheorem*{expectation}{Expectation}

\theoremstyle{remark}

\newtheorem{reminder}[thm]{Reminder}

\author{Daniel Greb}
\address{Daniel Greb, Essener Seminar für Algebraische Geometrie und Arithmetik, Fakultät für Mathe\-ma\-tik, Universität Duisburg--Essen, 45117 Essen, Germany}
\email{\href{mailto:daniel.greb@uni-due.de}{daniel.greb@uni-due.de}}
\urladdr{\href{http://www.esaga.uni-due.de/daniel.greb/}{http://www.esaga.uni-due.de/daniel.greb}}

\author{Stefan Kebekus}
\address{Stefan Kebekus, Mathematisches Institut, Albert-Ludwigs-Universität Freiburg, Eckerstraße 1, 79104 Freiburg im Breisgau, Germany and University of Strasbourg Institute for Advanced Study (USIAS), Strasbourg, France}
\email{\href{mailto:stefan.kebekus@math.uni-freiburg.de}{stefan.kebekus@math.uni-freiburg.de}}
\urladdr{\href{http://home.mathematik.uni-freiburg.de/kebekus}{http://home.mathematik.uni-freiburg.de/kebekus}}

\author{Thomas Peternell}
\address{Thomas Peternell, Mathematisches Institut, Universität
  Bayreuth, 95440~Bayreuth, Germany}
\email{\href{mailto:thomas.peternell@uni-bayreuth.de}{thomas.peternell@uni-bayreuth.de}}
\urladdr{\href{http://www.komplexe-analysis.uni-bayreuth.de/}{http://www.komplexe-analysis.uni-bayreuth.de}}

\author{Behrouz Taji}
\address{Behrouz Taji, University of Notre Dame, Department of Mathematics, 278 Hurley, Notre Dame, IN
46556, USA}
\email{\href{mailto:btaji@nd.edu}{btaji@nd.edu}}
\urladdr{\href{http://sites.nd.edu/b-taji}{http://sites.nd.edu/b-taji}}

\thanks{Daniel Greb was partially supported by the DFG-Collaborative Research
  Center SFB/TR 45 ``Periods, Moduli and Arithmetic of Algebraic Varieties''.
  Stefan Kebekus gratefully acknowledges support through a joint fellowship of
  the Freiburg Institute of Advanced Studies (FRIAS) and the University of
  Strasbourg Institute for Advanced Study (USIAS).  A part of this paper was
  worked out while Kebekus enjoyed the hospitality of IMPA in Rio de Janeiro.
  Behrouz Taji was partially supported by the DFG-Graduiertenkolleg GK1821
  ``Cohomological Methods in Geometry'' at Freiburg.}



\title[The Miyaoka-Yau Inequality and uniformisation]{The Miyaoka-Yau inequality and uniformisation of canonical models}
\date{\today}

\makeatletter
\hypersetup{
  pdfauthor={\authors},
  pdftitle={\@title},
  pdfsubject={\@subjclass},
  pdfkeywords={\@keywords},
  pdfstartview={Fit},
  pdfpagelayout={TwoColumnRight},
  pdfpagemode={UseOutlines},
  bookmarks,
  colorlinks,
  linkcolor=linkblue,
  citecolor=linkred,
  urlcolor=linkred}
\makeatother

\begin{document}

\begin{abstract}
We establish the Miyaoka-Yau inequality in terms of orbifold Chern classes for
the tangent sheaf of any complex projective variety of general type with klt
singularities and nef canonical divisor.  In case equality is attained for a
variety with at worst terminal singularities, we prove that the associated
canonical model is the quotient of the unit ball by a discrete group action.\end{abstract}

\maketitle
\approvals[Approvals for Abstract]{
  Behrouz & yes \\
  Daniel & yes \\
  Stefan & yes \\
  Thomas & yes
}
\tableofcontents

%
%
\svnid{$Id: 01-intro.tex 840 2018-11-20 14:25:41Z kebekus $}

\section{Introduction}
\subversionInfo

\approvals{
  Behrouz & yes \\
  Daniel & yes \\
  Stefan & yes \\
  Thomas & yes
}

A classical result in complex geometry asserts that the Chern classes of any
holomorphic, slope-semistable vector bundle $ℰ$ of rank $r$ on a compact Kähler
manifold $(X,ω)$ satisfy the \emph{Bogomolov-Gieseker inequality}
$$
\int_X \bigl(2r·c_2(ℰ)-(r-1)·c_1²(ℰ)\bigr)Λ ω^{n-2}≥ 0.
$$
Thanks to his solution of the Calabi conjecture, Yau established in
\cite{MR0451180} the following stronger, \emph{Miyaoka-Yau inequality} for the
holomorphic tangent bundle of any $n$-dimensional compact Kähler manifold $X$
with ample canonical class $K_X$,
\begin{equation}\label{ineq:MY-classic}
  \int_X \bigl(2(n+1)·c_2(𝒯_X)-n·c_1²(𝒯_X)\bigr)·[K_X]^{n-2}≥ 0.\tag{$\ast$}
\end{equation}
In case of equality, the natural symmetries imposed by the Kähler-Einstein
condition lead to the uniformisation of $X$ by the unit ball.

A fundamental result of Birkar, Cascini, Hacon and McKernan, \cite{BCHM10},
states that every projective manifold of general type admits a minimal model,
which is a normal, $ℚ$-factorial, projective variety with at most terminal
singularities whose canonical divisor is big and nef.  These varieties are
however usually singular.  It was expected that the Miyaoka-Yau inequality
should also hold in this context, with applications to uniformisation in case of
equality.  This problem has attracted considerable interest;
Section~\ref{sect:known-results} gives a short account of the history.

\subsection{Main results of this paper}
\approvals{
  Behrouz & yes \\
  Daniel & yes \\
  Stefan & yes \\
  Thomas & yes
}

The main result of this paper settles the problem in full generality, even in
the broader context of varieties with Kawamata log-terminal (=klt) singularities
and nef canonical divisor.

\begin{thm}[$ℚ$-Miyaoka-Yau inequality]\label{thm:MYinequality}
  Let $X$ be an $n$-dimensional, projective, klt variety of general type whose
  canonical divisor $K_X$ is nef.  Then,
  \begin{equation}\label{eq:X2}
    \Bigl( 2(n+1)· \widehat{c}_2(𝒯_X)-n·\widehat{c}_1(𝒯_X)²
    \Bigr)·[K_X]^{n-2}≥ 0.
  \end{equation}
\end{thm}

The formulation of Theorem~\ref{thm:MYinequality} uses the fact that varieties
with klt singularities have quotient singularities in codimension two.  This
allows us to define $ℚ$-Chern classes (or ``orbifold Chern classes'')
$\widehat{c}_1(𝒯_X)$ and $\widehat{c}_2(𝒯_X)$ for the tangent sheaf
$𝒯_X = (Ω¹_{X})^*$ of $X$, which is reflexive and a $ℚ$-vector bundle on the
open subset where $X$ has quotient singularities.  We refer to
Section~\ref{sect:QQ2} for definitions and for a detailed discussion.  If $X$ is
smooth in codimension two, which is the case when $X$ has terminal
singularities, these agree with the usual Chern classes $c_{•}(𝒯_X)$.  We call a
projective variety of general type \emph{minimal} if it has at worst terminal
singularities and if its canonical divisor is nef, cf.~\cite[2.13]{KM98} and
Definition~\ref{def:minimal} below.

\begin{thm}[Characterisation of singular ball quotients, I]\label{thm:BQ}
  Let $X$ be an $n$-dimensional minimal variety of general type.  If equality
  holds in \eqref{eq:X2}, then the canonical model $X_{can}$ is smooth in
  codimension two, there exists a ball quotient $Y$ and a finite, Galois,
  quasi-étale morphism $f: Y → X_{can}$.  In particular, $X_{can}$ has only
  quotient singularities.
\end{thm}

We refer to Section~\ref{ssec:defnVar} for a discussion of ball quotients and
canonical models.

We expect that Theorem~\ref{thm:BQ} holds without the additional assumption that
$X$ be terminal.  In fact, we prove a result slightly stronger than
Theorem~\ref{thm:BQ}, which applies to varieties with klt singularities that are
smooth in codimension two, cf.~Theorem~\ref{thm:more_general_uniformisation} as
well as Theorem and Definition \ref{thm:ball-quotient-II} below.  As already
said above, Theorem~\ref{thm:BQ} applies to all minimal models of smooth
varieties of general type, which is the case most relevant for applications in
the Minimal Model Program.  In Section~\ref{subsect:furtherdirections}, we
discuss further potential generalisations of the Miyaoka-Yau inequality and the
uniformisation result.

Extending Theorem~\ref{thm:BQ}, we show that the canonical models of
Theorem~\ref{thm:BQ} admit a ``singular uniformisation'' by the unit ball $𝔹^n$.
More precisely, they can be realised as quotients of $𝔹^n$ by actions of
discrete subgroups in $\PSU(1, n)$ that are not necessarily fixed-point free.
In particular, the geometry of these spaces can be studied using the theory of
automorphic forms, cf.\ \cite[Part~II]{Kollar95s}.

\begin{thmDef}[Characterisation of singular ball quotients, II]\label{thm:ball-quotient-II}
  Let $X$ be a normal, irreducible, compact, complex space of dimension $n$.
  Then, the following statements are equivalent.
  \begin{enumerate}
  \item\label{il:z1} The space $X$ is of the form $𝔹^n/\what{Γ}$ for a discrete,
    cocompact subgroup $\what{Γ} < \Aut_{𝒪}(𝔹^n)$ whose action on $𝔹^n$ is
    fixed-point free in codimension two.

  \item\label{il:z2} The space $X$ is of the form $Y/G$, where $Y$ is a ball
    quotient (cf.\ Definition~\ref{defn:BQ}), and $G$ is a finite group of
    automorphisms of $Y$ whose action is fixed-point free in codimension two.

  \item\label{il:z3} The space $X$ is projective, klt, and smooth in codimension
    two; the canonical divisor $K_X$ is ample, and we have equality in the
    $ℚ$-Miyaoka-Yau Inequality~\eqref{eq:X2}.
  \end{enumerate}
  A compact complex space is called \emph{singular ball quotient} if it
  satisfies these equivalent conditions.
\end{thmDef}

\begin{cor}[Hyperbolicity of smooth loci of canonical models]\label{cor:smoothparthyperbolic}
  In the setting of Theorem~\ref{thm:BQ}, the canonical model $X_{can}$ is a
  singular ball quotient.  In particular, the smooth locus of $X_{can}$ is
  Kobayashi-hyperbolic.
\end{cor}

\approvals{
  Behrouz & yes \\
  Daniel & yes\\
  Stefan & yes\\
  Thomas & yes
}

In fact, a more precise hyperbolicity statement holds, see
Section~\ref{subsect:hyperbolicitycomments}.\Publication{\ In addition,
  classical results concerning deformation rigidity \cite{MR0111058}, Mostow
  rigidity \cite[Thm.~6]{MR0451180}, stability under Galois conjugation
  \cite[Cor.~9.5]{MR944577}, and the fact that ball quotients can be defined
  over number fields \cite{MR0223368} have analogues for singular ball
  quotients.  These aspects will be addressed in future work.}

\subsection{Outline of the proof}
\label{ssec:outline}
\approvals{
  Behrouz & yes \\
  Daniel & yes \\
  Stefan & yes \\
  Thomas & yes}

Various earlier papers used differential-geometric techniques, such as orbifold
Kähler-Einstein metrics, to obtain the Miyaoka-Yau inequality.  Inspired by the
work of Simpson \cite{MR944577} we take a different approach, partially
generalising Simpson's results on the Kobayashi-Hitchin correspondence for Higgs
sheaves.  For suitable manifolds $X$, Simpson equips $ℰ := Ω¹_X ⊕ 𝒪_X$ with a
natural structure of a Higgs bundle, proves its stability and derives a
Bogomolov-Gieseker inequality for $ℰ$.  The Miyaoka-Yau inequality for $𝒯_X$ is
an immediate consequence.  In case of equality, he constructs a variation of
Hodge structures whose period map gives the desired uniformisation by the ball.

On a technical level, one main contribution of our paper is to establish a good
definition of Higgs sheaves on singular spaces, and an associated notion of
stability.  These definitions may seem a little awkward at first, but for
varieties with the singularities of the minimal model program they have just
enough universal properties to make Simpson's approach work---the list of
properties includes restrictions theorems of Mehta-Ramanathan type, weakly
functorial pull-back, and invariance of stability under resolution.  As for a
converse, earlier work on differential forms, \cite{GKKP11, MR3084424}, suggests
that spaces with klt singularities are the largest class of varieties where
functorial pull-back properties can possibly hold for any reasonable definition.

In our singular situation, the correct analogue of the sheaf $ℰ$ used by Simpson
is $(Ω¹_X)^{**} ⊕ 𝒪_X$.  The starting point of our analysis is the fact that
this Higgs sheaf is stable with respect to $K_X$ in case $X$ is klt and $K_X$ is
big and nef.  This is a consequence of a recent result of Guenancia
\cite{Guenancia}, which in turn generalises a by now classical result of Enoki
\cite{Eno87} to the klt setup.  Using restriction theorems of Mehta-Ramanathan
type, Theorem~\ref{thm:MYinequality} follows as a consequence of a
Bogomolov-Gieseker-type inequality for stable Higgs sheaves on surfaces with
quotient singularities, Theorem~\ref{thm:BogIneq}.

To prove Theorem~\ref{thm:BQ}, let $Y → X$ we consider a quasi-étale cover,
where the étale fundamental groups $\what{π}_1(Y)$ and $\what{π}_1(Y_{\reg})$
agree; the existence of such covers was established in \cite[Thm.~1.5]{GKP13}.
We aim to prove that $Y$ is smooth.  The proof is based on the second main
technical contribution of this paper, a partial generalisation of Simpson's
Nonabelian Hodge Correspondence to the singular setting, see
Section~\ref{subsect:VHS}.  Using the relation of special representations of
fundamental groups to Higgs \emph{bundles} and variations of Hodge structures,
the choice of $Y$ allows us to prove that $(Ω¹_Y)^{**} ⊕ 𝒪_Y$ is in fact locally
free.  The confirmation of the Zariski-Lipman conjecture for spaces with klt
singularities, \cite[Thm.~6.1]{GKKP11}, then shows that $Y$ is smooth.  Using
the original uniformisation theorem proven by Yau, we conclude that $Y$ is a
ball quotient.

\subsection{Comparison with the torus-quotient case}
\approvals{
  Behrouz & yes \\
  Daniel & yes \\
  Stefan & yes \\
  Thomas & yes
}

A related uniformisation problem for klt varieties with vanishing first and
second Chern class has been solved by the authors partly in joint work with
Steven Lu, see \cite{GKP13} and \cite{LT14}.  We would like to emphasise that
although there are some similarities between the strategies of the proof of
Theorem~\ref{thm:BQ} and those that were devised to solve the uniformisation
problem in the case of vanishing Chern classes, the two approaches are
significantly different.  First of all, to prove Theorem~\ref{thm:BQ}, we need a
suitable notion of Higgs sheaves over singular spaces verifying some important
functoriality properties (see Subsections~\ref{sec:pull-back},
\ref{sec:rpull-back} and~\ref{ssect:restrict}).  But, the difference between the
two approaches in~\cite{GKP13} and the current paper is not confined to the
technicalities that arise from the setup of a theory of Higgs sheaves over
singular spaces; a substantially refined strategy is required for a successful
application of this new machinery to establish the uniformisation result,
Theorem~\ref{thm:BQ}.  We refer the reader to Remark~\ref{rem:CompareFlat} for a
detailed comparison of the two strategies.

\subsection{Structure of the paper}
\approvals{
  Behrouz & yes \\
  Daniel & yes \\
  Stefan & yes \\
  Thomas & yes
}

Section~\ref{sec:known} establishes notation and reviews a few facts that will
be used later.  Building on work of Mumford,
Sections~\ref{sect:Q-varieties}--\ref{sect:QQ2} establish basic properties
pertaining to $ℚ$-varieties and $ℚ$-sheaves, and uses these to construct
$ℚ$-Chern classes on klt spaces.

Sections~\ref{sect:oper}--\ref{sect:Higgs} introduce the main objects of our
study: sheaves with operators and (singular) Higgs sheaves on klt spaces.  The
extension theorem for reflexive differential forms and the existence of
pull-back functors, \cite{GKKP11, MR3084424}, allow us to establish weak
functoriality properties for Higgs sheaves, including variants of pulling-back
for certain morphisms, as well as into and out of $ℚ$-varieties.  This allows us
to compare stability of Higgs sheaves on different birational models.  It also
helps to establish a restriction theorem of Mehta-Ramanathan type,
Theorem~\ref{thm:restriction}, which allows us to reduce many of our problems to
the surface case.  In Section~\ref{subsect:VHS}, we extend Simpson's
correspondence between rigid representations of the fundamental group of a
smooth projective variety and polarised complex variations of Hodge structures
to our singular setup, thereby establishing the foundational steps of a
Nonabelian Hodge Theory on klt spaces.

With these methods at hand, we establish a $ℚ$-analogue of the
Bogomolov-Gieseker inequality in Section~\ref{sect:bogo}.  Section~\ref{sect:MY}
applies this, as well as a recent stability result of
Guenancia~\cite{Guenancia}, to establish the $ℚ$-Miyaoka-Yau inequality,
Theorem~\ref{thm:MYinequality}.  The second main result, Theorem~\ref{thm:BQ},
is shown in Section~\ref{sect:uniformisation}.

The concluding Section~\ref{sect:ball-quotient} discusses quotients of the ball
by cocompact subgroups of its automorphism group, in order to prove the
characterisation of singular ball quotients given in
Theorem~\ref{thm:ball-quotient-II}, as well as the hyperbolicity result of
Corollary~\ref{cor:smoothparthyperbolic}.  We conclude with an example of Keum,
showing that many of our results are essentially sharp.

\subsection{Earlier work}
\label{sect:known-results}
\approvals{
  Behrouz & yes \\
  Daniel & yes \\
  Stefan & yes \\
  Thomas & yes
}

Generalisations of the Miyaoka-Yau inequality and uniformisation in case of
equality have attracted considerable interest in the last few decades.

Inequality~\eqref{ineq:MY-classic} and the uniformisation result were extended
to the context of compact Kähler varieties with only quotient singularities by
Cheng-Yau~\cite{MR0833802} using orbifold Kähler-Einstein metrics.  Tsuji
established Inequality~\eqref{ineq:MY-classic} for \emph{smooth minimal models}
of general type in \cite{MR0976585}.  Enoki's result on the semistability of
tangent sheaf of minimal models, \cite{Eno87}, was used by Sugiyama
\cite{MR1145268} to establish the Bogomolov-Gieseker inequality for the tangent
sheaf of any resolution of a given minimal model of general type with only
canonical singularities, the polarisation given by the pullback of the canonical
bundle on the minimal model.  By using a strategy very similar to ours, that is
via results of Simpson~\cite{MR944577}, Langer in~\cite[Thm.~5.2]{MR1954067}
established the Miyaoka-Yau inequality in this context.  He recently also gave
the first purely algebraic proof of the Bogomolov Inequality for semistable
Higgs sheaves (on smooth projective varieties over fields of arbitrary
characteristic), see \cite{MR3314517}.

A strong uniformisation result, together with the Miyaoka-Yau inequality, was
established by Kobayashi~\cite{MR0799669} in the case of open orbifold surfaces.

After the work of Tsuji, the past few years have witnessed significant
developments in the theory of singular Kähler-Einstein metrics and Kähler-Ricci
flow.  These are evident, for example, in the works of Tian-Zhang
\cite{Tian-Zhang}, Eyssidieux-Guedj-Zeriahi \cite{MR2505296}, and Zhang
\cite{Zhang06}.  In particular, Inequality \eqref{ineq:MY-classic} together with
a uniformisation result for \emph{smooth minimal models} of general type have
been successfully established by Zhang~\cite{MR2497488}.

Finally, we mention that the related uniformisation problem for klt varieties
with vanishing first and second Chern class has been solved by the authors
partly in joint work with Steven Lu, see \cite{GKP13} and \cite{LT14}.

\subsection{Acknowledgements}
\approvals{
  Behrouz & yes \\
  Daniel & yes \\
  Stefan & yes \\
  Thomas & yes
}

The authors would like to thank Dave Anderson, Paolo Cascini, Fabrizio Catanese,
Philippe Eyssidieux, Jochen Heinloth, Zsolt Patakfalvi, Erwan Rousseau, Emanuel
Scheidegger, Shigeharu Takayama and Angelo Vistoli for helpful discussions.
Adrian Langer answered our questions concerning his paper \cite{MR3314517} and
saved us from making a grave mistake at least once.  He also explained his
approach to the restriction theorem\Preprint{, cf.\ our remarks at the beginning
  of Section~\ref{ssec:langer1}}.  Behrouz Taji warmly thanks Steven Lu for fruitful discussions.
  
  All authors want
to thank the Institut Élie Cartan (Nancy) for the invitation to attend the
Journées Complexes Lorraines 2015, during which crucial discussions concerning
the content of this article took place, and the referee for carefully reading the paper as well as for valuable remarks.

%
%
\svnid{$Id: 02-knownResults.tex 838 2018-04-25 11:39:59Z kebekus $}

\section{Notation and standard facts}
\label{sec:known}
\subversionInfo

\subsection{Global conventions}
\label{ssec:conventions}
\approvals{
  Behrouz & yes \\
  Daniel & yes \\
  Stefan & yes \\
  Thomas & yes
}

Throughout this paper, all schemes, varieties and morphisms will be defined over
the complex number field.  We follow the notation and conventions of
Hartshorne's book \cite{Ha77}.  In particular, varieties are always assumed to
be irreducible.  For all notation around Mori theory, such as klt spaces and klt
pairs, we refer the reader to \cite{KM98}.

\subsection{Varieties}
\label{ssec:defnVar}
\approvals{
  Behrouz & yes \\
  Daniel & yes \\
  Stefan & yes \\
  Thomas & yes
}

In the course of the proofs, we need to switch between the Zariski-- and the
Euclidean topology at times.  We will consistently use the following notation.

\begin{notation}[Complex space associated with a variety]
  Given a variety or projective scheme $X$, denote by $X^{an}$ the associated
  complex space, equipped with the Euclidean topology.  If $f : X → Y$ is any
  morphism of varieties or schemes, denote the induced map of complex spaces by
  $f^{an} : X^{an} → Y^{an}$.  If $ℱ$ is any coherent sheaf of $𝒪_X$-modules,
  denote the associated coherent analytic sheaf of $𝒪_{X^{an}}$-modules by
  $ℱ^{an}$.
\end{notation}

The notion of ``$ℚ$-Chern class'', which is used in the formulation of our main
result, is usually defined for varieties with quotient singularities.  However,
the word ``quotient singularity'' is not consistently used in the literature and
is often left undefined.  We use the following terminology.

\begin{defn}[Quotient varieties and quotient singularities]\label{defn:quotient}
  Let $X$ be a normal, quasi-projective variety.  We say that $X$ is a
  \emph{quotient variety} if there exists a finite group $G$, a smooth
  $G$-variety $\what{X}$\footnote{in other words, the group $G$ acts on the
    variety $\what{X}$ by automorphisms.} such that $X ≅ \what{X}/G$ and such
  that the quotient map is étale over $X_{\reg}$.  We say that $X$ has
  \emph{quotient singularities}, if there exists a covering of $X^{an}$ by
  analytically-open sets $(U_α)_{α ∈ A}$, and for each $α ∈ A$ a quotient
  variety $Y_α$, and an analytically open set $V_α ⊆ Y_α^{an}$ that is
  biholomorphic to $U_α$.
\end{defn}

Our main result pertains to canonical models of varieties of general type.  We
briefly recall the relevant definitions and facts.

\begin{defn}[Minimal varieties]\label{def:minimal}
  A normal, projective variety $X$ is called \emph{minimal} if $X$ has at worst
  terminal singularities and if $K_X$ is nef.
\end{defn}

\begin{reminder}[Basepoint-Free Theorem and Canonical models]\label{remi:cm}
  If $X$ is a projective, klt variety of general type whose canonical divisor
  $K_X$ is nef, the Basepoint-Free Theorem asserts that $K_X$ is semiample,
  \cite[Thm.~3.3]{KM98}.  A sufficiently high multiple of $K_X$ thus defines a
  birational morphism $φ: X → Z$ to a normal projective variety with at worst
  klt singularities whose canonical divisor $K_{Z}$ is ample, cf.\
  \cite[Lem.~2.30]{KM98}.  There exists a $ℚ$-linear equivalence
  $K_X \sim_ℚ φ^* K_{Z}$.  If $X$ is a minimal variety of general type, $Z$ has
  at worst canonical singularities, we set $Z = X_{can}$, and call it the
  \emph{canonical model} of $X$.
\end{reminder}

\begin{defn}[Ball quotient]\label{defn:BQ}
  A smooth projective variety $X$ of dimension $n$ is a \emph{ball quotient} if
  the universal cover of $X^{an}$ is biholomorphic to the unit ball
  $𝔹^n = \{ (z_1, \dots, z_n) ∈ ℂ^n \mid |z_1|² + \dots + |z_n|² < 1\}$.
  Equivalently, there exists a discrete subgroup $Γ < \Aut_{𝒪}(𝔹^n)$ of the
  holomorphic automorphism group of $𝔹^n$ such that the action of $Γ$ on $𝔹^n$
  is cocompact and fixed-point free, and such that $X$ is isomorphic to $𝔹^n/Γ$.
\end{defn}

The following will be used for notational convenience.

\begin{notation}[Big and small subsets]
  Let $X$ be a normal, quasi-projective variety.  A closed subset $Z ⊂ X$ is
  called \emph{small} if $\codim_X Z ≥ 2$.  An open subset $U ⊆ X$ is called
  \emph{big} if $X \setminus U$ is small.
\end{notation}

Fundamental groups are basic objects in our arguments.  We will use the
following notation.

\begin{defn}[Fundamental group and étale fundamental group]
  If $X$ is a complex, quasi-projective variety, we set
  $π_1\bigl(X\bigr) := π_1\bigl(X^{an}\bigr)$, and call it the \emph{fundamental
    group of $X$}.  Moreover, the étale fundamental group of $X$ will be denoted
  by $\what{π}_1\bigl(X\bigr)$.
\end{defn}

\begin{rem}
  Recall that $\what{π}_1(X)$ is isomorphic to the profinite completion of
  $π_1(X)$; e.g.~see \cite[§5 and references given there]{Milne80}.
\end{rem}

\subsection{Morphisms}
\approvals{
  Behrouz & yes \\
  Daniel & yes \\
  Stefan & yes \\
  Thomas & yes
}

Galois morphisms appear prominently in the literature, but their precise
definition is not consistent.  We will use the following definition, which does
not ask Galois morphisms to be étale.

\begin{defn}[Covers and covering maps, Galois morphisms]\label{def:cover}
  A \emph{cover} or \emph{covering map} is a finite, surjective morphism
  $γ : X → Y$ of normal, quasi-projective varieties.  The covering map $γ$ is
  called \emph{Galois} if there exists a finite group $G ⊂ \Aut(X)$ such that
  $γ$ is isomorphic to the quotient map.
\end{defn}

\begin{notation}
  In the setting of Definition~\ref{def:cover}, we will frequently write
  $$
  \xymatrix{ %
    X \ar[rrr]^{γ}_{\text{Galois with group } G} &&& Y & \text{or} & X \ar[rr]^{γ}_{·/G} && Y
  }
  $$
  to indicate that $γ$ is isomorphic to the quotient map.  We will also write
  $G = \Gal(X/Y)$.
\end{notation}

\begin{defn}[Quasi-étale morphisms]\label{defn:quasietale}
  A morphism $f : X → Y$ between normal varieties is called \emph{quasi-étale}
  if $f$ is of relative dimension zero and étale in codimension one.  In other
  words, $f$ is quasi-étale if $\dim X = \dim Y$ and if there exists a closed,
  subset $Z ⊆ X$ of codimension $\codim_X Z ≥ 2$ such that
  $f|_{X \setminus Z} : X \setminus Z → Y$ is étale.
\end{defn}

\subsection{Sheaves}
\approvals{
  Behrouz & yes \\
  Daniel & yes \\
  Stefan & yes \\
  Thomas & yes
}

Reflexive sheaves are in many ways easier to handle than arbitrary coherent
sheaves, and we will therefore frequently take reflexive hulls.  The following
notation will be used.

\begin{notation}[Reflexive hull]
  Given a quasi-projective variety $X$ and a coherent sheaf $ℰ$ on $X$, write
  $$
  Ω^{[p]}_X := \bigl(Ω^p_X \bigr)^{**}, \quad ℰ^{[m]} := \bigl(ℰ^{⊗ m}
  \bigr)^{**} \quad\text{and}\quad \det ℰ := \bigl( Λ^{\rank ℰ} ℰ \bigr)^{**}.
  $$
  Given any morphism $f : Y → X$, write $f^{[*]} ℰ := (f^* ℰ)^{**}$, etc.
\end{notation}

One key notion in our argument is that of a \emph{flat sheaf}.

\begin{defn}[\protect{Flat sheaf, \cite[Def.~1.15]{GKP13}}]\label{defn:flat}
  If $X$ is any quasi-projective variety and $ℱ$ is any locally free, analytic
  sheaf on the underlying complex space $X^{an}$, we call $ℱ$ \emph{flat} if it
  is defined by a representation of the topological fundamental group
  $π_1\bigl(X^{an} \bigr)$.  A locally free, \emph{algebraic} sheaf $ℰ$ on $X$
  is called flat if the associated analytic sheaf $ℰ^{an}$ is flat.
\end{defn}

We use \cite[Chap.~3]{Fulton98} as our main reference for Chern classes on
singular spaces.  The Bogomolov discriminant plays a central role.

\begin{notation}[Bogomolov discriminant]\label{not:bogomolov}
  Let $X$ be a projective variety and $ℰ$ be a locally free sheaf on $X$, of
  rank $r > 0$.  One defines the \emph{Bogomolov discriminant} of $ℰ$ as
  $Δ(ℰ) := 2r·c_2(ℰ) - (r-1)·c_1(ℰ)²$.
\end{notation}

\subsection{\boldmath$G$-sheaves}
\approvals{
  Behrouz & yes \\
  Daniel & yes \\
  Stefan & yes \\
  Thomas & yes
}

In the discussion of $ℚ$-varieties one needs to consider varieties $X$ that are
equipped with a faithful action of a finite group $G$.  Almost all sheaves that
are relevant in our discussion come with a natural structure of a
\emph{$G$-sheaf}, also called \emph{$G$-linearised sheaf} in the literature,
\cite{MR1304906}.  A detailed discussion of $G$-sheaves, including full proofs
of all relevant facts used here, is found in \cite[§~1.3]{MR1304906},
\cite[§~3.2]{Viehweg95}, and \cite[Appendix~A]{GKKP11}.

\begin{notation}[$G$-invariant push-forward]
  Let $X$ be a quasi-projective variety, equipped with a faithful action of a
  finite group $G$, and with associated quotient map $π : X → X/G$.  If $ℰ$ is
  any $G$-sheaf on $X$, write $π_*(ℰ)^G ⊆ π_*(ℰ)$ to denote the $G$-invariant
  part of the push-forward.
\end{notation}

If $X$ has a $G$-action and $ℰ$ is a $G$-sheaf, it is generally not true that
any $G$-subsheaf $ℱ ⊆ ℰ$ comes from the quotient.  The following
\Publication{standard } proposition gives a criterion when this \emph{is} true.
\Preprint{We include a full proof for lack of reference.}\Publication{ The
  preprint version, available from the arXiv, contains a full proof.}

\begin{prop}[$G$-sheaves coming from the quotient]\label{prop:1:1}
  Let $γ: Y → X$ be a Galois morphism with group $G$.  Let $ℬ_X$ be a reflexive
  sheaf on $X$ and $ℬ := γ^{[*]} ℬ_X$.  Observe that $ℬ$ naturally carries the
  structure of a $G$-sheaf.  If $𝒜 ⊆ ℬ$ is any saturated $G$-subsheaf, then
  there exists a reflexive, saturated subsheaf $𝒜_X ⊆ ℬ_X$ such that
  $𝒜 = γ^{[*]} 𝒜_X$.\Publication{\qed}
\end{prop}
\Preprint{%
\begin{proof}
  Consider the quotient $𝒞 := \factor{ℬ}{𝒜}$, which is a torsion free $G$-sheaf
  by assumption.  Its push-forward $γ_* 𝒞$ and the $G$-invariant part of the
  push-forward, $γ_* (𝒞)^G$, are likewise torsion free; the same holds for $𝒜$
  and $ℬ$.  Recalling from \cite[Lemma~A.3]{GKKP11} that $γ_*( · )^G$ is an
  exact functor, we obtain an exact sequence of torsion free sheaves on $X$,
  \begin{equation}\label{eq:ups}
    \xymatrix{ %
      0 \ar[r] & γ_* (𝒜)^G \ar[r] & γ_* (ℬ)^G \ar[r] & γ_* (𝒞)^G \ar[r] & 0.
    }
  \end{equation}
   
  Since two reflexive sheaves agree if and only if they agree on the complement
  of a small closed subset, we are free to remove small subsets from $X$ and
  $Y$.  As torsion free sheaves are locally free in codimension one, we are
  therefore free to assume that all sheaves in \eqref{eq:ups} are locally free.
  Pulling back to $Y$, we will then obtain a natural diagram as follows,
  $$
  \xymatrix{ %
    0 \ar[r] & γ^* \bigl( γ_* (𝒜)^G \bigr) \ar[r] \ar[d]_{a} & γ^* \bigl( γ_* (ℬ)^G \bigr) \ar[r] \ar[d]_{b} & γ^* \bigl( γ_* (𝒞)^G \bigr) \ar[r] \ar[d]_{c} & 0 \\
    0 \ar[r] & 𝒜 \ar[r] & ℬ \ar[r] & 𝒞 \ar[r] & 0.
  }
  $$
  The morphisms $a$, $b$ and $c$ are clearly injective over the open set of $X$
  where $γ$ is étale.  It will therefore follow from local freeness that $a$,
  $b$ and $c$ are in fact injective.  The construction of $ℬ$ immediately
  implies that $γ_* (ℬ)^G$ is isomorphic to $ℬ_X$, and $b$ is therefore
  isomorphic.  It then follows from the Snake Lemma that $a$ is in fact
  surjective.  We can thus finish the proof by setting $𝒜_X := γ_* (𝒜)^G$.
  Torsion freeness of $γ_* (𝒞)^G = \factor{ℬ_X}{𝒜_X}$ implies that $𝒜_X$ is
  saturated in $ℬ_X$.
\end{proof}
}

\subsection{Intersection, slope and stability}
\approvals{
  Behrouz & yes \\
  Daniel & yes \\
  Stefan & yes \\
  Thomas & yes
}

Given a normal, quasi-projective variety $X$, we follow the notation of
\cite{Fulton98} and denote by $A_k(X)$ the \emph{groups of $k$-dimensional
  cycles modulo rational equivalence}.  The symbol $N¹(X)_ℚ$ denotes the
$ℚ$-vector space of numerical Cartier divisor classes.  Given any divisor $D$ or
any sheaf $ℰ$ whose determinant is $ℚ$-Cartier, we write the appropriate
elements of $N¹(X)_ℚ$ as $[D]$ and $[ℰ] := [\det ℰ]$, respectively.

We recall the following standard construction of intersection numbers between
Weil-- and Cartier divisors.

\begin{construction}[Intersection of Weil and Cartier divisors]\label{cons:int}
  Let $X$ be an $n$-dimensional, normal, projective variety and $0 \ne ℰ$ be any
  coherent sheaf.  Its determinant is then a Weil divisorial sheaf on $X$, say
  $\det ℰ = 𝒪_X(D)$.  The Weil divisor $D$ defines an element $Δ ∈ A_{n-1}(X)$.
  Given $(n-1)$ line bundles $ℒ_1, …, ℒ_{n-1}$, we can then form the cap
  product and consider the number
  $$
  \deg \bigl( Δ ∩ c_1(ℒ_1) ∩ ⋯ ∩ c_1(ℒ_{n-1}) \bigr) ∈ ℤ.
  $$
  Since its value depends only on the numerical classes of the line bundles
  $ℒ_i$, the sheaf $ℰ$ induces a well-defined $ℚ$-multilinear form
  $N¹(X)_ℚ^{⨯(n-1)} → ℚ$.
\end{construction}

\begin{notation}
  Abusing notation somewhat, we denote the multilinear form of
  Construction~\ref{cons:int} by $[ℰ]$, as if the sheaf $ℰ$ had a numerical
  class.  Given elements $α_1, …, α_{n-1} ∈ N¹(X)_ℚ$, we denote the associated
  value by $[ℰ]·α_1 ⋯ α_{n-1} ∈ ℚ$.
\end{notation}

The abuse of notation is partially justified by the following remark.

\begin{rem}
  In the setting of Construction~\ref{cons:int}, if $π : \wtilde{X} → X$ is any
  resolution of singularities, then
  $[ℰ]·α_1 ⋯ α_{n-1} = [π^*ℰ]·π^*α_1 ⋯ π^*α_{n-1}$.  If $\det ℰ$ is $ℚ$-Cartier,
  then there \emph{is} a numerical class $[ℰ] ∈ N¹(X)_ℚ$, and
  Construction~\ref{cons:int} gives the expected results.
\end{rem}

\begin{defn}[Slope with respect to a nef divisor]\label{def:slope2}
  Let $X$ be a normal, projective variety and $H$ be a nef $ℚ$-Cartier divisor
  on $X$.  If $ℰ \ne 0$ is any torsion free, coherent sheaf on $X$, define the
  \emph{slope of $ℰ$ with respect to $H$} as
  $$
  μ_H(ℰ) := \frac{ [ℰ]·[H]^{\dim X-1}}{\rank ℰ}.
  $$
  Call $ℰ$ \emph{semistable with respect to $H$} if $μ_H(ℱ) ≤ μ_H(ℰ)$ for any
  subsheaf $ℱ ⊆ ℰ$ with $0 < \rank ℱ < \rank ℰ$.  Call $ℰ$ \emph{stable with
    respect to $H$} if strict inequality always holds.
\end{defn}

In the setup of Definition~\ref{def:slope2}, the class $[H]^{\dim X-1}$ is a
movable numerical curve class, cf.\ \cite[Def.~2.2]{GKP15}.  If $X$ is
$ℚ$-factorial, our definition of slope agrees with that of
\cite[Def.~2.10]{GKP15}.  The standard proofs of the following elementary facts
carry over from \cite{GKP15} essentially verbatim.

\begin{lem}[Elementary properties of slope]\label{lem:elemSlp2}
  In the setup of Definition~\ref{def:slope2}, if $π : \wtilde{X} → X$ is any
  generically finite morphism of normal, projective varieties, then the
  following holds.
  \begin{enumerate}
  \item\label{il:Ah} We have
    $μ_H(ℰ) = (\deg π)^{-1}·μ_{π^*H} \bigl( π^{[*]} ℰ \bigr)$.

  \item\label{il:Be} If $\wtilde{ℰ}$ is any coherent sheaf on $\wtilde X$ that
    differs from $π^{[*]} ℰ$ at most over a small subset of $X$, then
    $μ_H(ℰ) = (\deg π)^{-1}·μ_{π^*H} \bigl( \wtilde{ℰ} \bigr)$.  \qed
  \end{enumerate}
\end{lem}

\begin{lem}[Harder-Narasimhan filtration]\label{lem:elemSlp1}
  In the setup of Definition~\ref{def:slope2}, there exists a unique filtration,
  $0 = ℰ_0 ⊊ ℰ_1 ⊊ ⋯ ⊊ ℰ_{ℓ} = ℰ$, whose quotients
  $ℰⁱ := ℰ_i/ℰ_{i-1}$ are torsion free, semistable with respect to $H$ and where
  the sequence $μ_H(ℰⁱ)$ is strictly decreasing.
\end{lem}

\begin{notation}[Harder-Narasimhan filtration]
  The filtration of Lemma~\ref{lem:elemSlp1} is called \emph{Harder-Narasimhan
    filtration}.  Call $ℰ_1$ the \emph{maximal $H$-destabilising subsheaf of
    $ℰ$} and write $ μ^{\max}_H(ℰ) := μ_H(ℰ¹)$ and
  $μ^{\min}_H(ℰ) := μ_H(ℰ^{ℓ-1})$.
\end{notation}

\part{Foundational material}
%
%
\svnid{$Id: 03-Qvarieties.tex 840 2018-11-20 14:25:41Z kebekus $}

\section{$ℚ$-varieties and $ℚ$-Chern classes}
\label{sect:Q-varieties}
\subversionInfo

\subsection{$ℚ$-varieties}
\approvals{
  Behrouz & yes \\
  Daniel & yes \\
  Stefan & yes \\
  Thomas & yes }

The construction of the $ℚ$-Chern classes that are used to formulate our main
results relies on the notions of $ℚ$-variety, also known as $V$-manifolds in the
literature.  While $ℚ$-Chern classes on $ℚ$-surfaces have been discussed in the
literature at length, the (sometimes delicate) issues arising in higher
dimensions are often not well covered.  For the reader's convenience, this
section gathers the main definitions, results and constructions concerning
$ℚ$-varieties that are used in our paper.  \Publication{For the sake of brevity,
  many of the more standard proofs are left to the reader.  The preprint version
  of the paper, available from the arXiv, contains full proofs and explains all
  constructions in detail.}

\subsection{$ℚ$-varieties}
\label{subsect:QQvarieties1}
\approvals{Behrouz & yes \\
  Daniel & yes \\
  Stefan & yes \\
  Thomas & yes}

We recall the definition of $ℚ$-varieties, as given in the fundamental articles
of Mumford and Gillet, \cite{MR717614, MR772058}.

Note that there are other definitions in the literature\footnote{The definition
  found in Megyesi's well-known article \cite[Sect.~10]{SecondAsterisque} is
  more restrictive than Definition~\ref{def:Qvar} in that it requires the
  morphisms $X_{α} → U_{α}$ to be quasi-étale.}.  To avoid
confusion, we restrict ourselves to Mumford's paper as a main reference and
stick to the notation introduced there.

\begin{defn}[\protect{$ℚ$-variety, cf.\ \cite[Sect.~2]{MR717614} and \cite[Def.~9.1]{MR772058}}]\label{def:Qvar}
  A \emph{$ℚ$-variety} is a tuple consisting of a normal, quasi-projective
  variety $X$, a finite set $A$ and for each $α ∈ A$ a smooth, quasi-projective
  variety $X_α$ and a diagram of morphisms between quasi-projective varieties
  \begin{equation}\label{eq:u5}
    \xymatrix{ %
      X_α \ar[rrr]_{\text{Galois with group }G_α} \ar@/^.4cm/[rrrrrr]^{p_α} &&& U_α \ar[rrr]_{p'_α\text{, étale}} &&& X
    }
  \end{equation}
  such that $X = \bigcup p_{α}(X_{α})$ and such that the following compatibility
  condition holds: for each $(α,β) ∈ A ⨯ A$, denoting by $X_{αβ}$ the
  normalisation of $X_α ⨯_X X_β$, then the natural morphisms
  $$
  p_{αβ,α} : X_{αβ} → X_{α} \quad \text{and} \quad p_{αβ,β} : X_{αβ} → X_{β}
  $$
  are étale.  In particular, $X_{αβ}$ is smooth.  For brevity of notation, we
  refer to the $ℚ$-variety by $\bigl(X, \{ p_α\}_{α ∈ A} \bigr)$.  We refer to
  the diagrams~\eqref{eq:u5} as \emph{charts}.
\end{defn}

We refer to \cite[Sect.~2.b]{MR717614} for the definition of a morphism of
$ℚ$-varieties as well as for an explanation of the relation to general Deligne-Mumford stacks.

\subsection{Quasi-étale $ℚ$-varieties}

This paper is mainly concerned with $ℚ$-varieties whose charts are quasi-étale.
As we will see below, these have particularly good properties.

\begin{defn}
  A $ℚ$-variety $\bigl(X, \{ p_α\}_{α ∈ A}\bigr)$ is called \emph{quasi-étale}
  if all the morphisms $p_α$ are quasi-étale.
\end{defn}

\begin{rem}[Quasi-étale charts]\label{rem:qec2}
  Let $X$ be a normal, quasi-projective variety $X$, let $A$ be a finite set and
  for each $α ∈ A$, and assume we are given a diagram as in~\eqref{eq:u5}, such
  that $X = \bigcup p_{α}(X_{α})$.  If all morphisms $p_α$ are quasi-étale, then
  the morphisms $X_{αβ} → X_α$ are then likewise étale in codimension one and
  hence, by purity of the branch locus, étale, \cite[Prop.~2]{MR0095846} or
  \cite[Thm.~1]{MR0106930}.  The condition on the $p_{αβ,α}$ is therefore
  vacuous, and the $p_α$ equip $X$ with the structure of a $ℚ$-variety.
\end{rem}

\begin{lem}[Uniqueness of quasi-étale $ℚ$-variety structures]\label{lem:cref}
  Let $X$ be any normal, quasi-projective variety.  Then, any two quasi-étale
  $ℚ$-variety structures on $X$ have a common refinement.\Publication{\qed}
\end{lem}
\Preprint{%
  \begin{proof}
    Let $X¹_ℚ = \bigl(X, \{ p¹_α\}_{α ∈ A}\bigr)$ and
    $X²_ℚ = \bigl(X, \{ p²_β\}_{β ∈ B}\bigr)$ be two quasi-étale $ℚ$-variety
    structures on $X$.  Denoting by $X_{αβ}$ the normalisation of $X_α ⨯_X X_β$,
    consider the diagram
    $$
    \xymatrix{ %
      X_{αβ} \ar[rr] \ar[d] && X_α \ar[d]^{\text{quasi-étale}} \\
      X_β \ar[rr]_{\text{quasi-étale}} && X.
    }
    $$
    The maps $X_{αβ} → X_α$ and $X_{αβ} → X$ are quasi-étale on every component
    of $X_{αβ}$.  Since $X_α$ is smooth, it follows from purity of the branch
    locus that the map $X_{αβ} → X_α$ is étale.  In particular, we see that
    $X_{αβ}$ is smooth.  The set of diagrams obtained by restricting
    $$
    \xymatrix{ %
      X_{αβ} \ar[rrrr]_(.45){\text{Quotient by }G_α⨯G_β} \ar@/^.4cm/[rrrrrrr]^{p_{αβ}} &&&& U_α⨯_XU_β \ar[rrr]_{p'_α ⨯ p'_β\text{, étale}} &&& X
    }
    $$
    to components of $X_{αβ}$ and to the appropriate image components of
    $U_α⨯_XU_β$ yields a $ℚ$-variety structure that refines both $X¹_ℚ$ and
    $X²_ℚ$.
  \end{proof}%
}

\subsection{Global covers}
\label{ssec:globalCover}
\approvals{
  Behrouz & yes \\
  Daniel & yes \\
  Stefan & yes \\
  Thomas & yes
}

Given an $n$-dimensional $ℚ$-variety $X_ℚ := \bigl( X, \{ p_α\}_{α ∈ A} \bigr)$
as in Definition~\ref{def:Qvar}, Mumford constructs in \cite[Sect.~2]{MR717614}
a \emph{global cover of $X_ℚ$}, that is, a normal variety $\what{X}$ (not
necessarily smooth), a global Galois morphism $γ : \what{X} → X$, and for every
$α ∈ A$ a commutative diagram as follows,
$$
\xymatrix{
  \what{X}_α \ar[rrrr]^{q_α}_{\text{Galois with group }H_α \triangleleft G} \ar@{^(->}[d]_{\text{incl.\ of open}} &&&& X_α \ar[rrrr]_{\text{Galois with group }G_α = G/H_α} &&&& U_α \ar[d]^{p'_α\text{, étale}} \\
  \what{X} \ar[rrrrrrrr]_{γ\text{, Galois with group }G} &&&&&&&& X.
}
$$
We call $\what{X}$ a \emph{global cover of $X_ℚ$}.

\begin{obs}[The importance of being Cohen-Macaulay, I]\label{obs:CM1}
  If $\what{X}$ is Cohen-Macaulay, then the Galois morphisms $q_α$ will
  automatically be flat, \cite[Ex.~18.17]{E95}.  In particular, pull-back of
  coherent sheaves is an exact functor.  Recalling that a coherent sheaf $ℱ$ is
  reflexive if and only if it is locally a $2^{nd}$ syzygy sheaf,
  \cite[Prop.~1.1]{MR597077}, it follows that for any $α ∈ A$, the pull-back of
  any reflexive sheaf on $X_α$ to $\what{X}_α$ is again reflexive.
\end{obs}

\subsection{$ℚ$-sheaves}
\label{subsect:QQsheaves}
\approvals{
  Behrouz & yes \\
  Daniel & yes \\
  Stefan & yes \\
  Thomas & yes
}

The next relevant items are the definition of $ℚ$-sheaves and the construction
of $ℚ$-sheaves by reflexive pull-back.

\begin{defn}[\protect{$ℚ$-sheaf and $ℚ$-bundle, cf.\ \cite[§ 2]{MR717614}}]\label{def:QS}
  Given a $ℚ$-variety $X_ℚ := \bigl(X, \{ q_α\}_{α ∈ A} \bigr)$ as in
  Definition~\ref{def:Qvar}, a \emph{$ℚ$-sheaf} $ℱ$ on $X_ℚ$ is a tuple
  $$
  \bigl( \{ ℱ_α\}_{α ∈ A}, \{i_{αβ}\}_{(α,β) ∈ A⨯A}\bigr)
  $$
  consisting of a family of coherent sheaves $ℱ_α$ on $X_α$ plus isomorphisms
  $$
  i_{αβ} : p_{αβ,α}^*(ℱ_α) → p_{αβ,β}^*(ℱ_β)
  $$
  that are compatible on the triple overlaps.  The $ℚ$-sheaf $ℱ$ is called
  \emph{reflexive} if all the $ℱ_α$ are reflexive.  It is called
  \emph{$ℚ$-bundle} if all the $ℱ_α$ are locally free.
\end{defn}

\begin{rem}[Induced sheaves on global covers]\label{rem:isgc}
  In the setting of Definition~\ref{def:QS}, given a global cover as in
  Section~\ref{ssec:globalCover}, Mumford shows that the pull-back sheaves
  $q_α^*ℱ_α$ glue to give a coherent $G$-sheaf $\what{ℱ}$ on $\what{X}$,
  \cite[Sect.~2]{MR717614}.  If we assume in addition that $ℱ$ is reflexive,
  then the $ℱ_α$ are locally free in codimension two, \cite[Cor.~1.4]{MR597077},
  and the same holds for $\what{ℱ}$.  In particular, if $ℱ$ is reflexive and
  $\dim X = 2$, then $\what{ℱ}$ is locally free.

  With this construction, Mumford proves that to give a $ℚ$-sheaf on $X_ℚ$, it
  is equivalent to give a $G$-sheaf on $\what{X}$ whose restrictions to
  $\what{X}_α$ are isomorphic (as $H_α$-sheaves) to pull-back sheaves from
  $X_α$.
\end{rem}

\begin{construction}[Reflexive pull-back]\label{cons:rpb1}
  Given a $ℚ$-variety $X_ℚ := \bigl(X, \{ p_α\}_{α ∈ A} \bigr)$ and given any
  coherent sheaf $ℱ$ on $X$, one defines a reflexive $ℚ$-sheaf $ℱ^{[ℚ]}$ on
  $X_ℚ$ by setting $ℱ_α := p_α^{[*]} ℱ$---the existence of natural isomorphisms
  $i_{αβ}$ is guaranteed by étalité of $p_{αβ,α}$ and $p_{αβ,β}$.  The
  $ℚ$-sheaves $\bigl( Ω^{[p]}_X \bigr)^{[ℚ]}$ and $\bigl( 𝒯_X \bigr)^{[ℚ]}$ are
  $ℚ$-bundles.
\end{construction}

\begin{rem}[The importance of being Cohen-Macaulay, II]\label{rem:CM2}
  In the setting of Construction~\ref{cons:rpb1}, let $\what{X}$ be a global
  cover as in Section~\ref{ssec:globalCover}, and let $\what{ℱ}$ be the sheaf
  induced by the $ℚ$-sheaf $ℱ^{[ℚ]}$, as in Remark~\ref{rem:isgc}.  If
  $\what{X}$ is Cohen-Macaulay, it follows from Observation~\ref{obs:CM1} that
  $\what{ℱ} = γ^{[*]} ℱ = γ^{[*]}(ℱ/\tor)$.  In particular, we obtain that these
  sheaves are locally free in codimension two.  In a similar vein, observing
  that the two reflexive sheaves $\what{ℱ}^*$ and
  $γ^{[*]}(ℱ^*) = γ^{[*]} \bigl( (ℱ/\tor)^* \bigr)$ agree over the big open set
  where the torsion free sheaf $ℱ/\tor$ is locally free, we see that
  $\what{ℱ}^*$ and $γ^{[*]}(ℱ^*)$ do in fact agree.
\end{rem}

\subsection{Constructions}
\approvals{Behrouz & yes \\
  Daniel & yes \\
  Stefan & yes \\
  Thomas & yes}

We recall three folklore constructions of $ℚ$-variety structures.

\subsubsection{Varieties with quotient singularities}
\approvals{Behrouz & yes \\
  Daniel & yes \\
  Stefan & yes \\
  Thomas & yes }

Given any $ℚ$-variety $\bigl(X, \{ p_α\}_{α ∈ A} \bigr)$ as in
Definition~\ref{def:Qvar}, then $X$ clearly has quotient singularities, in the
sense of Definition~\ref{defn:quotient}.  We briefly recall the fundamental fact
that the converse is also true, cf.~\cite[p.~276]{MR717614}.

\begin{prop}[Varieties with quotient singularities admit $ℚ$-structures]\label{cons:vqs1}
  Let $X$ by any quasi-projective variety with quotient singularities.  Then,
  $X$ admits the structure of a quasi-étale $ℚ$-variety.\Publication{\qed}
\end{prop}
\Preprint{%
  \begin{proof}
    Artin's Algebraic Approximation, \cite[Cor.~2.6]{ArtinApprox}, allows us to
    find an étale covering of $X$ by (normal) quotient varieties
    $U_α ≅ X_α/G_α$, with $X_α$ smooth, that have the additional property that
    the morphisms $X_α → X$ are étale in codimension one.  We have seen in
    Remark~\ref{rem:qec2} that this defines a $ℚ$-variety structure on $X$.
  \end{proof}%
}

\subsubsection{Cutting down}
\approvals{Behrouz & yes \\
  Daniel & yes \\
  Stefan & yes \\
  Thomas & yes}

If $X$ is a quasi-projective variety that has been equipped with the structure
of a $ℚ$-variety, there is generally no natural $ℚ$-variety structure on an
arbitrary hypersurfaces or subvarieties of $X$, cf.\ \cite[Warnings on
p.~116]{SecondAsterisque}.  We remark that this is different for general
elements of basepoint free linear systems.

\begin{prop}[$ℚ$-variety structures on general hyperplanes]\label{prop:CD1}
  Let $X_ℚ := \bigl(X, \{ p_α\}_{α ∈ A}) \bigr)$ be a $ℚ$-variety, let $ℒ$ be
  a line bundle on $X$ and $V ⊆ |ℒ|$ a finite-dimensional, basepoint free
  linear system whose generals element are irreducible.  Then, there exists a
  dense, Zariski-open subset $V° ⊆ V$ such any hypersurface $H$ is irreducible
  and normal, and admits a structure $H_ℚ$ of a $ℚ$-variety, together with a
  morphism $ι_ℚ : H_ℚ → X_ℚ$ whose induced morphism $ι : H → X$ is the
  inclusion.  Under the following additional assumptions, more is true.
  \begin{enumerate}
  \item\label{il:d1} If $ℰ$ is any reflexive sheaf on $X$ and if $H ∈ V°$ is
    general, then $ℰ|_H$ is likewise reflexive, and
    $ι_ℚ^* \bigl( ℰ^{[ℚ]} \bigr) ≅ \bigl( ℰ|_H \bigr)^{[ℚ]}$.
  \item\label{il:d2} If $X_ℚ$ is quasi-étale and $H ∈ V°$ is general, then $H_ℚ$
    is quasi-étale.
  \item\label{il:d3} If $X_ℚ$ admits a global, Cohen-Macaulay cover and $H ∈ V°$
    is general, then $H_ℚ$ admits a global, Cohen-Macaulay
    cover.\Publication{\qed}
  \end{enumerate}
\end{prop}
\Preprint{%
  \begin{proof}
    There exists a dense, Zariski-open subset $V° ⊆ V$ such that the following
    holds for any hypersurface $H ∈ V°$.
    \begin{enumerate}
    \item The hypersurfaces $H$ and $(p'_α)^{-1}(H)$ are normal,
      \cite[Thm.~7]{Seidenberg50}.
    \item For any $α ∈ A$, the preimage $H_α := p_α^{-1}(H)$ is smooth and
      intersects the ramification locus of $p_α$ with the expected dimension.
    \end{enumerate}
    The charts obtained by restricting $p_α$ to the irreducible components of
    $H_α$ equip $H$ with the structure $H_ℚ$ of a $ℚ$-variety.  The obvious
    inclusion maps $H_α → X_α$ give the desired morphism $ι_ℚ : H_ℚ → X_ℚ$.  It
    remains to consider the special cases.

    \subsubsection*{Case~\ref*{il:d1}}
    
    The $ℚ$-sheaf $i_ℚ^*\bigl( ℰ^{[ℚ]} \bigr)$ is given by the collection of
    sheaves $\bigl(p_α^{[*]} ℰ) \bigr|_{H_α}$ on the $H_α$.  If $H ∈ V°$ is
    general, then both $ℰ|_H$ and the sheaves $\bigl(p_α^{[*]} ℰ\bigr)|_{H_α}$
    will be reflexive, \cite[Thm.~12.2.1]{EGA4-3}, and $H$ intersects the small
    subset of $X$ where $ℰ$ fails to be locally free in the expected dimension.
    The last condition guarantees that the sheaves
    $\bigl(p_α^{[*]} ℰ \bigr)|_{H_α}$ and $\bigl(p_α|_{H_α}\bigr)^{[*]} (ℰ|_H)$
    agree away from a small set, for all $α ∈ A$.  Since both sheaves are
    reflexive, they actually agree.  Conclude by observing that the $ℚ$-sheaf
    $\bigl(ℰ|_H\bigr)^{[ℚ]}$ is given by the collection
    $\bigl(p_α|_{H'_α}\bigr)^{[*]} (ℰ|_H)$, for irreducible components
    $H'_α ⊆ H_α$.

    \subsubsection*{Case~\ref*{il:d2}}
    
    If $X_ℚ$ is quasi-étale, then the ramification locus of $p_α$ is small in
    $X_α$, and so is the ramification locus of $p_α|_{H'_α}$ on the irreducible
    components $H'_α ⊆ H_α$, for general $H ∈ V°$ and for all $α ∈ A$.  The
    $ℚ$-variety $H_ℚ$ is thus again quasi-étale.

    \subsubsection*{Case~\ref*{il:d3}}
    
    If $X_ℚ$ admits a Cohen-Macaulay cover, say $γ : \what{X} → X$, then one may
    use \cite[Thm.~7]{Seidenberg50} to see that for general $H$, the preimage
    $γ^*(H)$ is a normal Cartier divisor in $\what{X}$ and therefore again
    Cohen-Macaulay.  Its irreducible components form global Cohen-Macaulay
    covers of $X_ℚ|_H$.
  \end{proof}%
}

\subsubsection{Quasi-étale coverings}
\approvals{
  Behrouz & yes \\
  Daniel & yes \\
  Stefan & yes \\
  Thomas & yes
}

There is generally no notion of ``pull-back'' for $ℚ$-variety structures, even
for finite morphisms.  If the morphism is quasi-étale, a pull-back structure
does exist, however.

\begin{prop}[$ℚ$-variety structures on quasi-étale coverings]\label{prop:QC1}
  Let $X_ℚ$ be a quasi-étale $ℚ$-variety, and $γ : Y → X$ a finite, quasi-étale
  cover.  Then, $Y$ admits a structure $Y_ℚ$ of a quasi-étale $ℚ$-variety,
  together with a morphism $γ_ℚ : Y_ℚ → X_ℚ$ whose induced morphism $Y → X$ is
  $γ$.  If $ℰ$ is any reflexive sheaf on $X$, then
  $γ_ℚ^* \, (ℰ^{[ℚ]}) ≅ (γ^{[*]} ℰ)^{[ℚ]}$.
\end{prop}
\begin{proof}
  Write $X_ℚ := \bigl(X, \{ p_α\}_{α ∈ A} \bigr)$ and, given any $α ∈ A$, let
  $Y_α$ be the normalisation of $X_α ⨯_X Y$.  Base change gives a finite set of
  diagrams as follows,
  \begin{equation}\label{eq:dsp}
    \begin{gathered}
      \xymatrix{ %
        Y_α \ar[rr]_{Q_α} \ar@/^5mm/[rrrr]^{q_α} \ar[d]_{γ_α} && U_α ⨯_X Y \ar[rr]_{q'_α} \ar[d] && Y \ar[d]^{γ\text{, quasi-étale}}\\
        X_α \ar[rr]^{\text{quot.\ by }G_α} \ar@/_5mm/[rrrr]_{p_α} && U_α \ar[rr]^{p'_α\text{, étale}} && X.
      }
    \end{gathered}
  \end{equation}
  It follows from stability of étalité under base change that $q'_α$ is étale
  and that $q_α$ is quasi-étale.  In particular, $U_α ⨯_X Y$ is normal.  This in
  turn implies that $Q_α$ is the quotient map for the natural $G_α$-action on
  $Y_α$.  Lastly, note that the map $γ_α$ is étale away from
  $$
  γ_α^{-1} \bigl( \Ramification p_α \bigr) \,∪\, q_α^{-1} \bigl( \Ramification γ \bigr),
  $$
  which is a small subset of $Y_α$.  Purity of the branch locus then implies
  that $γ_α$ is étale and, in particular, that $Y_α$ is smooth.  Using
  Remark~\ref{rem:qec2}, we see that the top rows of the diagrams
  \eqref{eq:dsp}, restricted to the irreducible components of $Y_α$, equip $Y$
  with a structure $Y_ℚ$ of a quasi-étale $ℚ$-variety.  The restrictions of the
  full Diagrams~\eqref{eq:dsp} to the irreducible components of $Y_α$ define a
  morphism $γ_ℚ : Y_ℚ → X_ℚ$ whose induced morphism $Y → X$ is $γ$.

  It remains to consider the $ℚ$-sheaves attached to a reflexive sheaf $ℰ$ on
  $X$.  To this end, observe that the $ℚ$-sheaf $γ_ℚ^* (ℰ^{[ℚ]})$ is given at
  the level of the $Y_α$ by the sheaves $γ_α^* \bigl( (ℰ^{[ℚ]})_α \bigr)$, which
  are reflexive because the $γ_α$ are étale.  But we have canonical
  isomorphisms,
  \begin{align*}
    γ_α^* \bigl( (ℰ^{[ℚ]})_α \bigr) & ≅ γ_α^* p_α^{[*]} ℰ && \text{definition of $ℰ^{[ℚ]}$}\\
                                    & ≅ q_α^{[*]} γ^{[*]} ℰ && \text{sheaves agree over big set where $ℰ$ is locally free,}
  \end{align*}
  which give the desired statement.
\end{proof}

\subsection{$ℚ$-Chern classes on klt spaces}
\label{sect:QQ2}
\approvals{Behrouz & yes \\
  Daniel & yes \\
  Stefan & yes \\
  Thomas & yes}

It is well understood that the base variety $X$ of any klt surface pair $(X,D)$
has quotient singularities.  The geometry of $X$ can then be studied using
generalised Chern classes, known as \emph{$ℚ$-Chern classes} or \emph{orbifold
  Chern classes}---we refer to Kawamata's proof \cite{Kaw92} of the abundance
conjecture in dimension three for an example.  In higher dimensions, the base
variety of a klt pair does not necessarily have quotient singularities.
However, once one removes a suitable subset $Z ⊆ X$ of codimension three, only
quotient singularities remain and $X \setminus Z$ can be equipped with the
structure of a $ℚ$-variety that admits a global, Cohen-Macaulay cover\Preprint{,
  cf.\ Lemma~\ref{lem:lx} below}.  In particular, following Mumford's
fundamental paper \cite{MR717614}, Chern classes can be defined.  Since
$\codim Z = 3$, this allows us to construct on any klt space useful intersection
products with first and second $ℚ$-Chern classes of any reflexive sheaf on $X$.
This applies in particular to the tangent sheaf $𝒯_X$.  \Preprint{We include a
  full construction and full proofs for lack of an adequate
  reference.}\Publication{The (standard) proofs are omitted, but spelled out in
  the arXiv version of this paper.}

\begin{thm}[$ℚ$-Chern classes on klt spaces]\label{thm:41}
  There exist a map that assigns to any projective, klt pair $(X,D)$ of
  dimension $n ≥ 2$ and any reflexive sheaf $ℰ$ on $X$ three symmetric,
  $ℚ$-multilinear forms, denoted as follows,
  $$\arraycolsep=2.0pt
  \begin{array}{rclcrcl}
    \what{c}_1(ℰ) : N¹(X)_ℚ^{⨯(n-1)} & → & ℚ, & \hphantom{xxx} & (α_1, …, α_{n-1}) & ↦ & \what{c}_1(ℰ)·α_1 ⋯ α_{n-1}\\
    \what{c}_1(ℰ)² : N¹(X)_ℚ^{⨯(n-2)} & → & ℚ, && (α_1, …, α_{n-2}) & ↦ & \what{c}_1(ℰ)²·α_1 ⋯ α_{n-2} \\
    \what{c}_2(ℰ) : N¹(X)_ℚ^{⨯(n-2)} & → & ℚ, && (α_1, …, α_{n-2}) & ↦ & \what{c}_2(ℰ)·α_1 ⋯ α_{n-2}
  \end{array}
    $$
    such that the following properties hold for all $X$, all reflexive $ℰ$ on
    $X$, and all $α_1, …, α_{n-1} ∈ N¹(X)_ℚ$.
  \begin{enumerate}
  \item\label{il:p1} If $n = 2$, then $X$ has quotient singularities, and
    $\what{c}_1(ℰ) ∈ N¹(X)^*_ℚ$ as well as $\what{c}_1(ℰ)²$, $\what{c}_2(ℰ) ∈ ℚ$
    are the classical $ℚ$-Chern classes\footnote{The ``multilinear forms''
      $\what{c}_1(ℰ)²$ and $\what{c}_2(ℰ)$ take no arguments and are therefore
      identified with rational numbers.} discussed in the literature,
    cf.~\cite[Chapt.~10]{SecondAsterisque}.  In particular, there exists a
    Galois cover $γ : \what{X} → X$ (not necessarily quasi-étale), where
    $γ^{[*]} ℰ$ is locally free for any reflexive sheaf $ℰ$ on $X$, and where
    the following equalities hold, for all $ℰ$ and for all numerical classes
    $α_1 ∈ N¹(X)_ℚ$,
    $$
    \what{c}_1(ℰ)·α_1 = (\deg γ)^{-1}·c_1 \bigl( γ^{[*]} ℰ \bigr)·γ^* α_1.
    $$
    Ditto for $\what{c}_1(ℰ)²$ and $\what{c}_2(ℰ)$.  The cover $γ$ is a global,
    Cohen-Macaulay cover for a suitable, quasi-étale $ℚ$-variety structure on
    $X$.
    
  \item\label{il:p3} If $n > 2$, if $ℒ ∈ \Pic(X)$ is a line bundle and
    $V ⊆ |ℒ|$ is a basepoint free linear system whose elements are all
    connected, then there exists a dense open subset $V° ⊆ V$ such that for all
    $H ∈ V°$, the hypersurface $H$ is irreducible, not contained in $\supp D$,
    the pair $(H,D|_H)$ is klt, the restriction $ℰ|_H$ is reflexive, and
    $$
    \what{c}_1(ℰ)·[ℒ]·α_2 ⋯ α_{n-1} = \what{c}_1(ℰ|_H)·α_2 ⋯
    α_{n-1}.
    $$
    Ditto for $\what{c}_1(ℰ)²$ and $\what{c}_2(ℰ)$.\Publication{\qed}
  \end{enumerate}
\end{thm}

\Preprint{Theorem~\ref{thm:41} will be shown in Section~\vref{ssec:pf:thm:41}.}
The construction of $\what{c}_1$ is compatible with classical definitions in
case the determinant of $ℰ$ is $ℚ$-Cartier.  If $ℰ$ is locally free, then the
forms $\what{c}_1(ℰ)$, $\what{c}_1(ℰ)²$ and $\what{c}_2(ℰ)$ equal the usual
product with the Chern classes of $ℰ$.

\begin{rem}
  Since every line bundle on $X$ is the difference of two very ample ones, it
  follows from multilinearity that the forms are uniquely determined by
  Items~\ref{il:p1}, \ref{il:p3}.  As for the converse, it might seem tempting
  take these items as a definition of the forms, in order to avoid the Mumford's
  constructions.  But then well-definedness needs to be shown, which will in
  essence lead to the same set of problems.
\end{rem}

The following definition and notation will be used in most of the applications.

\begin{defn}[$ℚ$-Bogomolov discriminant]
  Let $(X,D)$ be a projective klt pair and $ℰ$ be a reflexive sheaf on $X$ of
  rank $r > 0$.  One defines the \emph{$ℚ$-Bogomolov discriminant} of $ℰ$ as the
  multilinear form
  $$
  \what{Δ}(ℰ) := 2r·\what{c}_2(ℰ) - (r-1)·\what{c}_1(ℰ)².
  $$
\end{defn}

We end this section with a number of remarks and immediate corollaries that we
will later use.

\subsection{Calculus of $ℚ$-Chern classes}
\approvals{Behrouz & yes \\
  Daniel & yes \\
  Stefan & yes \\
  Thomas & yes}

The following results help to compute $ℚ$-Chern classes in practise.  They
follow fairly quickly from Mumford's construction and from basic properties of
$ℚ$-varieties.

\begin{lem}[Behaviour under quasi-étale covers]\label{lem:buqec}
  If $(X,D)$ is a projective klt pair of dimension $n ≥ 2$ and if $γ : Y → X$ is
  quasi-étale between projective varieties, then $(Y,γ^*D)$ is again klt, and
  the following equalities hold for all reflexive sheaves $ℰ$ and all numerical
  classes $α_1, …, α_{n-1} ∈ N¹(X)_ℚ$,
  $$
  \what{c}_1(γ^{[*]} ℰ)· (γ^*α_1) ⋯ (γ^*α_{n-1}) = (\deg γ)·\what{c}_1(ℰ)·α_1 ⋯ α_{n-1}.
  $$
  Ditto for $\what{c}_1(γ^{[*]} ℰ)²$ and
  $\what{c}_2(γ^{[*]} ℰ)$.\Publication{\qed}
\end{lem}
\Preprint{%
  \begin{proof}
    The fact that $(Y,γ^*D)$ is klt is found in \cite[Prop.~5.20]{KM98}.
    Cutting down, Item~\ref{il:p3} allows us to assume without loss of
    generality that $X$ and $Y$ are surfaces, and that $ℚ$-Chern $\what{c}_1$,
    $\what{c}^{\,2}_1$ and $\what{c}_2$ are the classical $ℚ$-Chern classes.
    
    In this setting, Proposition~\ref{prop:QC1} equips $Y$ with the structure of
    a quasi-étale $ℚ$-variety $Y_ℚ$ that admits a morphism of $ℚ$-varieties,
    $γ_ℚ : Y_ℚ → X_ℚ$ whose induced morphism of varieties is $γ$.  In addition,
    by Proposition~\ref{prop:QC1} we have a canonical isomorphism of $ℚ$-sheaves
    $γ_ℚ^* \, (ℰ^{[ℚ]}) ≅ (γ^{[*]} ℰ)^{[ℚ]}$.  The formulas then follow from
    \cite[Prop.~3.8]{MR717614}.
  \end{proof}%
}

\begin{lem}[$ℚ$-Chern classes of flat sheaves]\label{lem:flat}
  If $(X,D)$ is a projective klt pair of dimension $n ≥ 2$ and $ℰ$ a reflexive
  sheaf on $X$ whose restriction to $X_{\reg}$ is locally free and flat, then
  the forms $\what{c}_1(ℰ)$, $\what{c}_1(ℰ)²$ and $\what{c}_2(ℰ)$ are all
  zero.\Publication{\qed}
\end{lem}
\Preprint{%
  \begin{proof}
    Given $X$ and $ℰ$, recall from \cite[Thm.~1.14]{GKP13} that there exists a
    quasi-étale covering $γ : Y → X$ such that the reflexive pull-back
    $γ^{[*]} ℰ$ is locally free and flat.  This implies that all Chern classes
    of $γ^{[*]} ℰ$ vanish.  Lemma~\ref{lem:buqec} thus yields the claim.
  \end{proof}
}

\Preprint{%
  \begin{lem}[Calculus of Chern classes]\label{lem:calculus}
    If $(X,D)$ is a projective klt pair of dimension $n ≥ 2$ and if $ℰ$, $ℱ$ are
    two reflexive sheaves on $X$, then the usual formulas for Chern classes of
    duals, tensor products and direct sums hold.  More specifically, we have
    equalities of multilinear forms, for $i ∈ \{1,2\}$,
    \begin{align*}
      \what{c}_i(ℰ) & = (-1)ⁱ·\what{c}_i(ℰ^*), & \what{c}_1(ℰ ⊕ ℱ) & = \what{c}_1(ℰ) + \what{c}_1(ℱ), \\
      \what{c}_1(ℰ)² & = \what{c}_1(ℰ^*)², & \what{c}_2(ℰ) & = \what{c}_2(ℰ ⊕ 𝒪_X), \\
      \what{Δ}(ℰ) & = \what{Δ}(ℰ^*) & \what{Δ}(\sEnd ℰ) & = 2(\rank ℰ)²·\what{Δ}(ℰ)
    \end{align*}
  \end{lem}
  \begin{proof}
    Item~\ref{il:p3} of Theorem~\ref{thm:41} allows us to reduce to the case
    where $X$ is a surface.  Item~\ref{il:p1} then allows us to pass to a cover
    $γ: \what{X} → X$, in order to compute the $ℚ$-Chern classes as honest Chern
    classes of locally free sheaves.  To conclude, observe that
    \begin{align*}
      γ^{[*]} (ℰ^*) & ≅ (γ^{[*]} ℰ)^*, & γ^{[*]}(ℰ ⊕ ℱ) & ≅ \bigl( γ^{[*]} ℰ \bigr ) ⊕ \bigl( γ^{[*]}ℱ \bigr), \\
      γ^{[*]} (\sEnd ℰ) & ≅ \sEnd \bigl(γ^{[*]} ℰ \bigr)
    \end{align*}
    because the sheaves are reflexive, and pairwise agree over the big open set
    where $ℰ$ is locally free.  We refer to \cite[Sect.~3.4]{MR2665168} for the
    relation between the Bogomolov discriminant of a locally free sheaf and of
    its endomorphism bundle.
  \end{proof}%
}

\Preprint{
\subsection{Preparation for the proof of Theorem~\ref*{thm:41}}
\approvals{
  Behrouz & yes \\
  Daniel & yes \\
  Stefan & yes \\
  Thomas & yes
}

We remarked in the introduction that the base variety of any klt pair has
quotient singularities in codimension two, and can be equipped with the
structure of a $ℚ$-variety after removing a very small set.  The following lemma
makes this statement precise.

\begin{lem}\label{lem:lx}
  Let $(X,D)$ be a klt pair, and $ℰ$ be a reflexive sheaf on $X$.  Then, there
  exists a closed subset $Z ⊆ X$ with $\codim_X Z ≥ 3$ and a structure of a
  quasi-étale $ℚ$-variety $X°_ℚ$ on $X° := X \setminus Z$ such that the
  following holds.
  \begin{enumerate}
  \item The $ℚ$-variety $X°_ℚ$ admits a global, Cohen-Macaulay cover
    $\what{X}°$.

  \item The $ℚ$-sheaf $(ℰ|_{X°})^{[ℚ]}$ is locally free.
  \end{enumerate}
\end{lem}
\begin{proof}
  Recall from \cite[Prop.~9.3]{GKKP11} that there exists a closed subset
  $Z_1 ⊆ X$ with $\codim_X Z_1 ≥ 3$ such that $X \setminus Z_1$ has quotient
  singularities.  In particular, Proposition~\ref{cons:vqs1} allows us to equip
  $X' := X \setminus Z_1$ with the structure of a $ℚ$-variety, say $X'_ℚ$.  In
  Section~\ref{ssec:globalCover}, we recalled Mumford's construction of a global
  cover $γ : \what{X}' → X'$, which is normal, in particular $S_2$.  It follows
  that there exists a closed subset $Z_2 ⊆ X'$ with $\codim_{X'} Z_2 ≥ 3$, such
  $\what{X}' \setminus γ^{-1}(Z_2)$ is Cohen-Macaulay,
  \cite[Prop.~5.7.4.1]{EGA4-3}.  Remark~\ref{rem:isgc} allows us to find a third
  set $Z_3 ⊆ X'$, again of codimension three, outside which $(ℰ|_{X°})^{[ℚ]}$ is
  locally free.  Set $Z := Z_1 ∪ Z_2 ∪ Z_3$.
\end{proof}

\subsection{Proof of Theorem~\ref*{thm:41}}
\label{ssec:pf:thm:41}
\approvals{
  Behrouz & yes \\
  Daniel & yes \\
  Stefan & yes \\
  Thomas & yes
}

Assume we are given a projective klt pair $(X,D)$ and a reflexive sheaf $ℰ$
there.  Applying Lemma~\ref{lem:lx} we find a closed subset $Z ⊆ X$, with
$\codim_X Z ≥ 3$, a structure of a quasi-étale $ℚ$-variety $X°_ℚ$ on
$X° := X \setminus Z$ with global, Cohen-Macaulay cover $\what{X}°$.  Consider
the $ℚ$-vector bundle $ℰ^{[ℚ]}$ on $X°_ℚ$, constructed by reflexive pull-back.
Next, recall Mumford's construction of $ℚ$-Chern classes for $ℚ$-vector bundles,
denoted as $c_i\bigl(ℰ^{[ℚ]}\bigr) ∈ A_{n-i}(X°)$.  This class is independent of
the atlas chosen in the construction of the $ℚ$-variety structure on $X°$, cf.\
\cite[Prop.~3.8]{MR717614} and the fact that any two quasi-étale atlases for
$ℚ$-variety structures on $X°$ admit a common refinement, Lemma~\ref{lem:cref}.

Recall also that Mumford equips $A_{*}(X°)$ with a ring structure, which allows
us to consider $c_1\bigl(ℰ^{[ℚ]}\bigr)² ∈ A_{n-2}(X°)$.  The localisation
sequence for Chow groups, \cite[Prop.~1.8]{Fulton98},
$$
A_{n-i}(Z) → A_{n-i}(X) → A_{n-i}(X°) → 0,
$$
induces a canonical isomorphism $A_{n-i}(X) ≅ A_{n-i}(X°)$, because
$\dim Z ≤ n-3$, hence $A_{n-i}(Z) = 0$ for $i ≤ 2$.  In summary, we obtain
classes $c_i\bigl(ℰ^{[ℚ]}\bigr) ∈ A_{n-i}(X)$ and
$c_1\bigl(ℰ^{[ℚ]}\bigr)² ∈ A_{n-2}(X)$.  Observing that all constructions
commute with open immersions, it follows that the classes are independent of any
of the choices made in their construction.  Taking cap products with Chern
classes of line bundles on $X$, we will therefore obtain a well-defined,
symmetric, $ℤ$-linear map
$$
\begin{matrix}
  \what{c}_2(ℰ) &:& \Pic(X)^{⨯(n-2)} & → & ℚ\\
  & & (ℒ_1, …, ℒ_{n-2}) & ↦ & \deg \bigl(c_2\bigl(ℰ^{[ℚ]}\bigr) ∩
  c_1(ℒ_1) ∩ ⋯ ∩ c_1(ℒ_{n-2}) \bigr),
\end{matrix}
$$
and analogously with $\what{c}_1(𝒢)$ and $\what{c}_1(𝒢)²$.  Observing that these
forms do not depend on the actual line bundles, but only on their numerical
classes, we obtain the forms $\what{c}_1(ℰ)$, $\what{c}_1(ℰ)²$, and
$\what{c}_2(ℰ)$ of Theorem~\ref{thm:41}.  We need to argue that they satisfy
conditions~\ref{il:p1} and \ref{il:p3}.

\subsubsection*{Condition~\ref*{il:p1}}

If $\dim X = 2$, then the subset $Z$ is necessarily empty.  It follows that
$X = X°$, and the construction above \emph{equals} the classical construction of
$ℚ$-Chern classes.

Choose any global cover $γ: \what{X} → X$ of $X_ℚ$, as discussed in
Section~\ref{ssec:globalCover}.  Then, $\what{X}$ is a normal surface and
therefore automatically Cohen-Macaulay.  The assertion that all sheaves of the
form $γ^{[*]} ℰ$ are locally free then follows from Remark~\ref{rem:CM2}.
Remark~\ref{rem:CM2} also asserts that the sheaf $\what{ℰ}$ on $\what{X}$
induced by $ℰ^{[ℚ]}$ equals $γ^{[*]} ℰ$.  The formulas for $\what{c}_1(ℰ)·α$,
$\what{c}_1(ℰ)²$ and $\what{c}_2(ℰ)$ thus follow directly from Mumford's
construction of Chern classes in $A_*(X)$, and the ring structure there.

\subsubsection*{Condition~\ref*{il:p3}}

If $H$ is general, then $H$ is normal, irreducible, and $(H,D_H)$ is klt,
\cite[Lem.~5.17]{KM98}.  The restriction $ℰ|_H$ is reflexive,
\cite[Thm.~12.2.1]{EGA4-3}.  The set $Z_H := H \setminus Z$ has codimension
$\codim_H Z_H ≥ 3$, and Proposition~\ref{prop:CD1} equips $H° := H \setminus Z$
with the structure of a quasi-étale $ℚ$-variety $H°_ℚ$ that admits a morphism of
$ℚ$-varieties, $ι_ℚ : H°_ℚ → X°_ℚ$ whose associated morphism of varieties is the
inclusion map $ι : H° → X°$.  In addition, we have a canonical isomorphism of
$ℚ$-sheaves $ι_ℚ^* \bigl( (ℰ|_{X°})^{[ℚ]} \bigr) ≅ (ℰ|_{H°})^{[ℚ]}$.  The
formulas then follow from \cite[Prop.~3.8]{MR717614}.  \qed }

\subsection{$ℚ$-varieties as Deligne-Mumford stacks}
\approvals{
  Behrouz & yes \\
  Daniel & yes \\
  Stefan & yes \\
  Thomas & yes
}

In the language of stacks, Mumford's constructions of $ℚ$-varieties recalled in
this section reflect the existence of an algebraic Deligne-Mumford (DM) stack
structure for algebraic varieties with only quotient singularities,
cf.~\cite[Prop.~9.2]{MR772058}.  More precisely, given a $ℚ$-variety /
quasi-étale $ℚ$-variety $(X, \{ p_{α} \}_{α ∈ A} )$, there is an algebraic DM
stack $χ$ verifying the following properties.

\begin{enumerate}
\item The stack $χ$ is smooth and $X$ is its coarse moduli space.
\item The isotropy group at the generic point / every codimension-one point of
  $χ$ is trivial.
\end{enumerate}

In \cite{MR772058}, Gillet developed an intersection theory for DM stacks and
their coarse moduli spaces and showed that for $ℚ$-varieties the theory is the
same as Mumford's theory recalled above.  Consequently, one could have used
Gillet's results instead of Mumford's in our discussion of $ℚ$-Chern classes for
reflexive sheaves on klt varieties.  By using Mumford's approach, we followed a
tradition in higher-dimensional classification theory, e.g.~see
\cite{SecondAsterisque, SBW94}.

%
%
\svnid{$Id: 04-Operators.tex 838 2018-04-25 11:39:59Z kebekus $}

\section{Sheaves with operators}
\label{sect:oper}
\subversionInfo

\subsection{Definitions and elementary operations}
\approvals{Behrouz & yes \\
  Daniel & yes \\
  Stefan & yes \\
  Thomas & yes}

In order to define and discuss Higgs sheaves on singular spaces in
Section~\ref{sect:Higgs}, this preliminary section discusses sheaves with
operators.  Our main emphasis lies on stability properties.  Because of the
singularities we cannot assume that any of the sheaves in question is locally
free.  We need to resort to the following, rather general definition.  We also
need to discuss the case of $G$-sheaves, but restrict ourselves to the minimal
amount of material required to make our arguments work.

\begin{defn}[Sheaf with an operator]\label{def:nshfop1}
  Let $X$ be a normal, quasi-projective variety and $𝒲$ be a coherent sheaf of
  $𝒪_X$-modules.  A \emph{sheaf with a $𝒲$-valued operator} is a pair $(ℰ, θ)$
  where $ℰ$ is a coherent sheaf and $θ: ℰ → ℰ ⊗ 𝒲$ is an $𝒪_X$-linear sheaf
  morphism.
\end{defn}

\begin{defn}[$G$-sheaf with an invariant operator]\label{def:nshfop2}
  Let $X$ be a normal, quasi-projective variety, equipped with the action of a
  finite group $G$, and let $𝒲$ be a coherent $G$-sheaf of $𝒪_X$-modules.  A
  \emph{$G$-sheaf with an invariant $𝒲$-valued operator} is a sheaf with a
  $𝒲$-valued operator, $(ℰ, θ)$, where $ℰ$ is a coherent $G$-sheaf and $θ$ is a
  morphism of $G$-sheaves.
\end{defn}

\begin{warning}[Incompatible definitions in the literature]
  The literature contains no uniform definition of sheaves with operators.  Our
  definition agrees with that of \cite[p.~257]{Langer04a} but differs from
  \cite[Def.~1.1]{MR1954067}.  All definitions that we have seen agree if $ℰ$ is
  torsion free and $𝒲$ is locally free.  We will be careful to quote the
  literature, in particular \cite{MR1954067}, only in settings where these
  conditions hold.
\end{warning}

\begin{construction}[Direct sum and tensor product]\label{cons:nHGStensor}
  Let $X$ be a normal, quasi-projective variety and $(ℰ_1, θ_1)$, $(ℰ_2, θ_2)$
  two sheaves with a $𝒲$-valued operator, as in Definition~\ref{def:nshfop1}.
  Then, $(ℰ_1⊗ℰ_2, θ_1⊗\Id_{ℰ_2} + \Id_{ℰ_1}⊗θ_2)$ and $(ℰ_1⊕ℰ_2, θ_1⊕θ_2)$ are
  again sheaves with a $𝒲$-valued operator
\end{construction}

\begin{construction}[Duals and endomorphisms]\label{cons:nHGSdual}
  Let $X$ be a normal, quasi-projective variety and $(ℰ, θ)$ a sheaf with a
  $𝒲$-valued operator, as in Definition~\ref{def:nshfop1}.  Assume that $ℰ$ is
  locally free.  The operator $θ$ can then be seen as a section in the sheaf
  $(\sEnd ℰ)⊗𝒲$.  Using the canonical identification $\sEnd ℰ ≅ \sEnd(ℰ^*)$,
  we obtain an operator on the dual sheaf, $θ^* : ℰ^* → ℰ^* ⊗ 𝒲$.
\end{construction}

\begin{notation}[Elementary operations]
  We denote the sheaves of Construction~\ref{cons:nHGStensor} briefly as
  $(ℰ_1, θ_2)⊕(ℰ_1, θ_2)$ and $(ℰ_1, θ_2)⊗(ℰ_1, θ_2)$.  If $ℒ$ is any coherent
  sheaf of $𝒪_X$-modules, taking zero-morphism gives sheaf $(ℒ, 0)$ with a
  $𝒲$-valued operator.  We will briefly write $(ℰ_1, θ_2)⊗ℒ$ instead of
  $(ℰ_1, θ_2)⊗(ℒ,0)$.  In the setting of Construction~\ref{cons:nHGSdual},
  write $(ℰ, θ)^* = (ℰ^*, θ^*)$ and $\sEnd (ℰ, θ) = (ℰ, θ)^* ⊗ (ℰ, θ)$.
\end{notation}

\begin{construction}[Pull-back and restriction]
  Let $X$ be a normal, quasi-projective variety and $(ℰ, θ)$ a sheaf with a
  $𝒲$-valued operator.  If $f : Y → X$ is morphism of normal varieties, then
  $f^*θ : f^*ℰ → f^*ℰ ⊗ f^*𝒲$ equips $f^*ℰ$ with the structure of a sheaf with
  an $f^*𝒲$-valued operator, which we denote as $f^*(ℰ, θ) = (f^*ℰ, f^*θ)$.  If
  $f$ is a closed immersion, we will also write $(ℰ, θ)|_Y = (ℰ|_Y, θ|_Y)$.
\end{construction}

\subsection{Invariant subsheaves}
\approvals{
  Behrouz & yes \\
  Daniel & yes \\
  Stefan & yes \\
  Thomas & yes
}

Much of the classical literature discusses sheaves $(ℰ, θ)$ with $𝒲$-valued
operators only in settings where both $ℰ$ and $𝒲$ are locally free.  Stability
of $(ℰ, θ)$ is then measured by looking at $θ$-invariant subsheaves of $ℰ$, that
is, subsheaves $ℱ ⊆ ℰ$ where $θ(ℱ) ⊆ ℱ ⊗ 𝒲$.  If $ℰ$ and $𝒲$ are arbitrary, the
tensor product $ℱ ⊗ 𝒲$ is not necessarily a subsheaf of $ℰ ⊗ 𝒲$ and the question
whether $θ(ℱ)$ is contained in $ℱ ⊗ 𝒲$ no longer makes sense.  In order to
obtain a workable theory with good universal properties and a meaningful
restriction theorem, the following more delicate definition needs to be used.

\begin{defn}[Invariant subsheaf]\label{def:ninvarSS}
  Let $X$ be a normal, quasi-projective variety and $(ℰ, θ)$ a sheaf with a
  $𝒲$-valued operator, as in Definition~\ref{def:nshfop1}.  A coherent subsheaf
  $ℱ ⊆ ℰ$ is called \emph{$θ$-invariant} if $θ(ℱ)$ is contained in the image of
  the natural map $ℱ ⊗ 𝒲 → ℰ ⊗ 𝒲$.  Call $ℱ$ \emph{generically invariant} if the
  restriction $ℱ|_U$ is invariant with respect to $θ|_U$, where $U ⊆ X$ is the
  maximal, dense, open subset where $𝒲$ is locally free.
\end{defn}

\begin{warning}[No operator on invariant subsheaves]\label{war:noois}
  In the setting of Definition~\ref{def:ninvarSS}, if $ℱ ⊆ ℰ$ is $θ$-invariant
  and $𝒲$ is locally free, then $ℱ ⊗ 𝒲 → ℰ ⊗ 𝒲$ is injective, the restricted map
  $θ|_{ℱ}$ factors via $ℱ ⊗ 𝒲$ and therefore endows $ℱ$ with the structure of a
  sheaf with a $𝒲$-valued operator.  If $𝒲$ is not locally free, then $θ$ does
  in general not induce a natural $𝒲$-valued operator on $ℱ$.  We refrain from
  discussing Harder-Narasimhan filtrations of sheaves with operators and do not
  attempt to define morphisms, or to construct an Abelian category.
\end{warning}

\begin{rem}[Invariance and tensor product]\label{rem:ntensor}
  In the setting of Construction~\ref{cons:nHGStensor}, let $ℱ ⊆ ℰ$ be any given
  subsheaf.  If $ℒ$ is invertible, then $ℱ ⊆ ℰ$ is $θ$-invariant (resp.\
  generically $θ$-invariant) if and only if $ℱ ⊗ ℒ ⊆ ℰ ⊗ ℒ$ is.
\end{rem}

We end the present subsection with two lemmas, pointing out that invariance is
well-behaved with respect to saturation.

\begin{lem}[Saturation of invariant subsheaf if $𝒲$ is locally free]\label{lem:nsat}
  In the setting of Definition~\ref{def:ninvarSS}, assume that $ℰ$ is torsion
  free and that $𝒲$ is locally free.  If $ℱ ⊆ ℰ$ is invariant, then so is its
  saturation $ℱ^{sat} ⊆ ℰ$.
\end{lem}
\begin{proof}
  Since $𝒲$ is locally free, $ℱ^{sat} ⊗ 𝒲$ is saturated in $ℰ ⊗ 𝒲$.  The sheaf
  $θ(ℱ^{sat})$, which is almost everywhere contained in $ℱ^{sat} ⊗ 𝒲$, is
  therefore entirely contained in $ℱ^{sat} ⊗ 𝒲$, and is hence $θ$-invariant.
\end{proof}

\begin{lem}[Saturations of sheaves that are invariant on an open subset]\label{lem:x11}
  In the setting of Definition~\ref{def:ninvarSS}, if there exists a dense open
  set $V ⊆ X$ such that $𝒲|_V$ is locally free and $ℱ|_V$ is $θ$-invariant, then
  is saturation $ℱ^{sat}$ is generically $θ$-invariant.
\end{lem}
\begin{proof}
  Aiming to prove that $ℱ^{sat}$ is \emph{generically} $θ$-invariant, we may
  assume without loss of generality that $𝒲$ is locally free.  The following
  composition of morphisms,
  $$
  \xymatrix{ %
    ℱ^{sat} \ar[r]^{θ|_{ℱ^{sat}}} & ℰ ⊗ 𝒲 \ar[rr]^(.25){\text{projection}} && \factor{(ℰ ⊗ 𝒲)}{(ℱ^{sat} ⊗ 𝒲)} = (\factor{ℰ}{ℱ^{sat}} ) ⊗ 𝒲,
  }
  $$
  will then vanish identically over $V$.  Since its target is torsion free as a
  tensor product of a torsion free and a locally free sheaf, it follows that the
  composition vanishes everywhere.  This shows the claim.
\end{proof}

\subsection{Stability}
\approvals{
  Behrouz & yes \\
  Daniel & yes \\
  Stefan & yes \\
  Thomas & yes
}

The notion of stability of sheaves with operators will be crucial for all what
follows.  The definition may look rather technical and perhaps not intuitive,
but has several advantages that will make our arguments work.  For one, it
agrees with the classical definition in cases where $ℰ$ is torsion free and $𝒲$
is locally free.  Secondly, it has good universal properties.  These will later
enable us to prove a restriction theorem for Higgs sheaves on singular spaces,
and compare stability of a Higgs sheaf on a singular space with that of its
pull-back to a resolution of singularities.

\begin{defn}[Stability of sheaves with operator]\label{defn:swostab1}
  Let $X$ be a normal, projective variety and $H$ be any nef, $ℚ$-Cartier
  $ℚ$-divisor on $X$.  Let $(ℰ, θ)$ be a sheaf with an operator, as in
  Definition~\ref{def:nshfop1}, were $ℰ$ is torsion free.  We say that $(ℰ, θ)$
  is \emph{semistable with respect to $H$} if the inequality $μ_H(ℱ) ≤ μ_H(ℰ)$
  holds for all generically $θ$-invariant subsheaves $ℱ ⊆ ℰ$ with
  $0 < \rank ℱ < \rank ℰ$.  The pair $(ℰ, θ)$ is called \emph{stable with
    respect to $H$} if strict inequality holds.  Direct sums of stable sheaves
  with operator are called \emph{polystable}.
\end{defn}

\begin{defn}[$G$-Stability of $G$-sheaves with operator]\label{defn:swostab2}
  Let $X$ be a normal, projective variety equipped with the action of a finite
  group $G$, and $H$ be any nef, $ℚ$-Cartier $ℚ$-divisor on $X$.  Let $(ℰ, θ)$
  be a $G$-sheaf with an invariant operator, as in Definition~\ref{def:nshfop2},
  were $ℰ$ is torsion free.  We say that $(ℰ, θ)$ is \emph{$G$-semistable with
    respect to $H$} if the inequality $μ_H(ℱ) ≤ μ_H(ℰ)$ holds for all
  generically $θ$-invariant $G$-subsheaves $ℱ ⊆ ℰ$ with $0 < \rank ℱ < \rank ℰ$.
  The pair $(ℰ, θ)$ is called \emph{$G$-stable with respect to $H$} if strict
  inequality holds.
\end{defn}

\begin{rem}
  The conditions spelled out in Definitions~\ref{defn:swostab1} and
  \ref{defn:swostab2} are trivially satisfied if $ℰ$ does not contain
  generically invariant subsheaves of the appropriate rank.
\end{rem}

\begin{lem}[Stability and tensor product]\label{lem:stabtp}
  In the setting of Definition~\ref{defn:swostab1}, let $ℒ$ be any invertible
  sheaf.  Then, $(ℰ, θ)$ is stable (resp.\ semistable) with respect to $H$ if
  and only if $(ℰ, θ) ⊗ ℒ$ is.
\end{lem}
\begin{proof}
  Lemma~\ref{lem:stabtp} follows from Remark~\ref{rem:ntensor} and the fact that
  slope is additive, $μ_H(ℱ ⊗ ℒ) = μ_H(ℱ)+μ_H(ℒ)$ for all non-trivial
  subsheaves $ℱ ⊆ ℰ$.
\end{proof}

We next address openness properties of stability, with the goal to generalise
results for ample polarisations to the nef case.  The following proposition is
not the strongest possible, but suffices for our purposes.

\begin{prop}[Openness of stability]\label{prop:G-opennessx}
  Let $X$ be a normal, projective variety, equipped with an action of a finite
  group $G$.  Let $H$ be a nef $ℚ$-Cartier $ℚ$-divisor, $[H] \ne 0$, and $(ℰ,θ)$
  be a torsion free $G$-sheaf which an invariant operator, and assume that
  $(ℰ,θ)$ is $G$-stable with respect to $H$.  Given any nef $ℚ$-Cartier
  $ℚ$-divisor $A$, there exists a positive number $ε_0$ such that for all
  rational numbers $0 < ε < ε_0$, the $G$-sheaf with invariant operator $(ℰ, θ)$
  is $G$-stable with respect to $(H + ε·A)$.
\end{prop}
\begin{proof}
  For simplicity of notation, write $n := \dim X$ and $r := \rank ℰ$.  We may
  assume that $H$ and $A$ are integral and Cartier.  In particular, recalling
  that the intersection numbers of Weil and Cartier divisors of
  Construction~\ref{cons:int} take values in the integers, the $H$-stability of
  $(ℰ,θ)$ implies that for any $G$-subsheaf $0 \ne ℱ ⊂ ℰ$ with $\rank ℱ < r$ we
  have
  \begin{equation} \label{eq:bounded denominator}
     μ_H(ℰ) - μ_H(ℱ) ≥ r^{-1}.
  \end{equation}
  Generalising Definition~\ref{def:slope2} slightly, given any number
  $0 ≤ k < n$, write
  $$
  μ_{A^kH^{n-1-k}}(ℰ) := \frac{ [ℰ]·[A]^{k}·[H]^{n-1-k}}{\rank ℰ}.
  $$
  Fix a resolution of singularities, $π : \wtilde{X} → X$ and observe that the
  curve class $α_k := [π^*A]^{k}·[π^*H]^{n-1-k} ∈ N_1(\wtilde X)_{ℚ}$ is
  movable.  In particular, if $ℱ ⊆ ℰ$ is any coherent subsheaf, then
  $π^{[*]} ℱ ⊆ π^{[*]} ℰ$, we have an equality of slopes,
  $μ_{A^kH^{n-1-k}}(ℱ) = μ_{α_k} \bigl( π^{[*]} ℱ \bigr)$, and it follows from
  \cite[Prop.~2.21]{GKP15} that
  $$
  μ^{\max}_{A^{•}H^{•}}(ℰ) := \sup \bigl\{ μ_{A^kH^{n-1-k}}(ℱ)
  \,|\, ℱ ⊆ ℰ \text{ a coherent subsheaf and } 0≤k<n \bigr\}
  $$
  is finite.  Now, given any rational $0 ≤ ε ≪ 1$ and any $G$-subsheaf
  $0 \ne ℱ ⊂ ℰ$ with $\rank ℱ < r$, owing to \eqref{eq:bounded denominator} we
  have
  \begin{align*}
    μ_{H+ε·A}(ℱ) & = μ_{H}(ℱ) + \sum_{k=1}^{n-1} \begin{pmatrix} n-1\\ k \end{pmatrix} ε^k·μ_{A^kH^{n-1-k}}(ℱ) \\
    & ≤ μ_{H}(ℰ)-\frac{1}{r} + μ^{\max}_{A^{•}H^{•}}(ℰ) · \sum_{k=1}^{n-1} \begin{pmatrix} n-1\\ k \end{pmatrix} ε^k \\
    & < μ_{H + ε·A}(ℰ),
  \end{align*}
  for $ε$ sufficiently small, which proves the claim.
\end{proof}

%
%
\svnid{$Id: 05-Higgs.tex 840 2018-11-20 14:25:41Z kebekus $}

\section{Higgs sheaves}
\label{sect:Higgs}
\subversionInfo
\approvals{
  Behrouz & yes \\
  Daniel & yes \\
  Stefan & yes \\
  Thomas & yes
}

This section introduces Higgs sheaves on singular varieties and establishes
their basic properties.  We include a discussion of Higgs $ℚ$-sheaves on
$ℚ$-varieties, investigate functoriality of Higgs sheaves, define stability and
prove a restriction theorem of Mehta-Ramanathan type.  We conclude with a
section on Higgs bundles and variations of Hodge structures that summarises some
work of Simpson and fits it into the framework of minimal model theory.

\subsection{Fundamentals}
\approvals{
  Behrouz & yes \\
  Daniel & yes \\
  Stefan & yes \\
  Thomas & yes
}

On a singular variety, some attention has to be paid concerning the definition
of ``Higgs sheaf'' at singular points.  We will see in
Section~\ref{sec:pull-back}--\ref{ssect:restrict} that Higgs sheaves in the
sense of the following definition have just enough universal properties to make
our strategy of proof work.  In the converse direction, it seems that
Definition~\ref{def:Higgs} and our notion of stability are in essence uniquely
dictated if we ask all these universal properties to hold.

\begin{defn}[Higgs sheaf and Higgs $G$-sheaf]\label{def:Higgs}
  Let $X$ be a normal variety.  A \emph{Higgs sheaf} is a pair $(ℰ, θ)$ of a
  coherent sheaf $ℰ$ of $𝒪_X$-modules, together with an $Ω^{[1]}_X$-valued
  operator $θ : ℰ → ℰ ⊗ Ω^{[1]}_X$, called \emph{Higgs field}, such that the
  composed morphism
  $$
  \xymatrix{ %
    ℰ \ar[rr]^(.35){θ} && ℰ ⊗ Ω^{[1]}_X \ar[rr]^(.45){θ ⊗ \Id} && ℰ ⊗
    Ω^{[1]}_X ⊗ Ω^{[1]}_X \ar[rr]^(.55){\Id ⊗ [Λ]} && ℰ ⊗ Ω^{[2]}_X }
  $$
  vanishes.  Following tradition, the composed morphism will be denoted by
  $θ Λ θ$.  If $X$ is equipped with the action of a finite group $G$, a
  \emph{Higgs $G$-sheaf} on $X$ is a Higgs sheaf $(ℰ, θ)$, where $ℰ$ is a
  $G$-sheaf, and where the Higgs field $θ$ is a morphism of $G$-sheaves.
\end{defn}

\begin{defn}[Morphism of Higgs sheaves]
  In the setting of Definition~\ref{def:Higgs}, a \emph{morphism of Higgs
    sheaves} (resp.\ \emph{morphism of Higgs $G$-sheaves}), written
  $f : (ℰ_1, θ_1) → (ℰ_2, θ_2)$, is a morphism $f : ℰ_1 → ℰ_2$ of sheaves
  (resp.\ $G$-sheaves) that commutes with the Higgs fields,
  $(f ⊗ \Id_{Ω^{[1]}_X}) ◦ θ_1 = θ_2 ◦ f$.
\end{defn}

The above definitions extend to $ℚ$-Higgs sheaves on $ℚ$-varieties.  These will
be introduced in Section~\ref{ssec:Qhiggs} once the existence of the necessary
pull-back functors has been established.

\begin{example}[A natural Higgs sheaf attached to a normal variety]\label{ex:BQfield1}
  Let $X$ be a normal variety.  Set $ℰ := Ω^{[1]}_X ⊕ 𝒪_X$ and define an
  operator $θ$ as follows,
  $$
  \begin{matrix}
    θ : & \mathclap{Ω^{[1]}_X} & ⊕ & \mathclap{𝒪_X} & → & \Bigl( Ω^{[1]}_X & ⊕ & 𝒪_X \Bigr) & ⊗ & Ω^{[1]}_X \\
    & a & + & b & ↦ & ( 0 & + & 1 ) & ⊗ & a.
  \end{matrix}
  $$
  An elementary computation shows that $θ Λ θ = 0$, so that $(ℰ, θ)$ forms a
  Higgs sheaf.  If $X$ is a $G$-variety, then $ℰ$ has a natural structure of a
  $G$-sheaf, and $(ℰ, θ)$ is in fact a $G$-Higgs sheaf.  Observe that the direct
  summand $𝒪_X ⊆ ℰ$ is generically $θ$-invariant.  Non-zero subsheaves of the
  direct summand $Ω^{[1]}_X$ are not generically $θ$-invariant.
\end{example}

\begin{example}[Tensor, dual and endomorphisms]
  The direct sum and tensor operation of Construction~\ref{cons:nHGStensor}
  transforms Higgs sheaves into Higgs sheaves.  Ditto for the dual sheaf and the
  endomorphism sheaf that are constructed in \ref{cons:nHGSdual} if the Higgs
  sheaf is locally free.
\end{example}

\subsection{Explanation}
\approvals{
  Behrouz & yes \\
  Daniel & yes \\
  Stefan & yes \\
  Thomas & yes
}

The reader might wonder why Definition~\ref{def:Higgs} requires the Higgs field
to take its values in $ℰ ⊗ Ω^{[1]}_X$.  At least two other potential choices for
the target come to mind.  At first sight, it might seem most natural and
functorial to take $ℰ ⊗ Ω_X¹$ for a target.  However, in the main application to
Miyaoka-Yau inequalities and to uniformisation for varieties of general type,
the naturally induced sheaf of geometric origin is $ℰ := Ω^{[1]}_X ⊕ 𝒪_X$, as
discussed in Example~\ref{ex:BQfield1} above.  For this particular $ℰ$ to be a
Higgs sheaf, we have to allow the target of the Higgs field to be
$ℰ ⊗ Ω_X^{[1]}$.  Also, note that looking at $Ω_X¹ ⊕ 𝒪_X$ instead would render a
discussion of semistability moot, as semistability requires torsion freeness and
even the most simple klt singularities lead to torsion in $Ω_X¹$, see
\cite{MR2915479} for examples.
  
On the other hand, the reader might wonder why $θ$ takes its values in
$ℰ ⊗ Ω^{[1]}_X$ and not in its reflexive hull.  The advantages of our choice
will become apparent in the following Section~\ref{sec:pull-back}, where
pull-back functors are defined: in general, none of the constructions there will
work for reflexive hulls.

\subsection{Pull-back}
\label{sec:pull-back}
\approvals{
  Behrouz & yes \\
  Daniel & yes \\
  Stefan & yes \\
  Thomas & yes
}

To pull back Higgs sheaves is at least as difficult as to pull-back reflexive
differentials.  Functorial pull-back for reflexive differentials does, however,
not exist in general unless the target space supports a divisor that makes it
klt.

\begin{construction}[Pull-back of Higgs sheaves]\label{cons:pb1}
  Let $(X,D)$ be a klt pair and let $(ℰ, θ)$ be a Higgs sheaf on $X$.  Given any
  normal variety $Y$ and any morphism $f : Y → X$, recall from \cite[Thms.~1.3
  and 5.2]{MR3084424} that there exists a natural pull-back functor for
  reflexive differentials on klt pairs that is compatible with the usual
  pull-back of Kähler differentials and gives rise to a sheaf morphism
  $$
  d_{\refl} f : f^* Ω^{[1]}_X → Ω^{[1]}_Y.
  $$
  We claim that $θ'$, defined as the composition of the following morphisms,
  \begin{equation}\label{eq:pbhgs}
    f^* ℰ \xrightarrow{f^* θ} f^* \Bigl( ℰ ⊗ Ω^{[1]}_X \Bigr) = f^*
    ℰ ⊗ f^* Ω^{[1]}_X \xrightarrow{\Id_{f^* ℰ} ⊗
      d_{\refl} f} f^* ℰ ⊗ Ω^{[1]}_Y,
  \end{equation}
  equips $f^* ℰ$ with the structure of a Higgs sheaf.  To check that
  $θ' Λ θ' = 0$, one uses the compatibility of reflexive pull-back with wedge
  products, \cite[Prop.~5.13]{MR3084424}, to verify that the following diagram
  is commutative,
  $$
  \xymatrix{ %
    f^* ℰ \ar[rr]_(.3){(θ ⊗ \Id)◦ θ} \ar@/^5mm/[rrr]^{f^*(θ Λ θ)} \ar@{=}[d] && f^* ℰ ⊗ f^* Ω^{[1]}_X ⊗ f^* Ω^{[1]}_X \ar[r] \ar[d]^{\Id ⊗ d_{\refl} ⊗ d_{\refl} } & f^* ℰ ⊗ f^* Ω^{[2]}_X \ar[d]^{\Id ⊗ d_{\refl} } \\
    f^* ℰ \ar[rr]^(.3){(θ' ⊗ \Id)◦ θ'} \ar@/_4mm/[rrr]_{θ' Λ θ'} && f^*
    ℰ ⊗ Ω^{[1]}_Y ⊗ Ω^{[1]}_Y \ar[r] & f^* ℰ ⊗ Ω^{[2]}_Y.  }
  $$
  By minor abuse of notation, this Higgs sheaf will be denoted as $f^* (ℰ, θ)$
  or $(f^* ℰ, f^* θ)$.
\end{construction}

\begin{notation}[Restriction of Higgs sheaves]
  In the setting of Construction~\ref{cons:pb1}, if $f$ is a closed or open
  immersion, we will also write $(ℰ, θ)|_Y$ or $(ℰ|_Y, θ|_Y)$.  To keep notation
  reasonably short, we will in the remainder of the paper tacitly equip
  restrictions of Higgs sheaves with their natural Higgs fields.
\end{notation}

We mentioned above that the pull-back functor
$d_{\refl} f : f^* Ω^{[1]}_X → Ω^{[1]}_Y$ is compatible with the usual pull-back
of Kähler differentials.  If $X$ and $Y$ are smooth and $f$ is a closed
immersion, the pull-back $f^* (ℰ, θ)$ of Construction~\ref{cons:pb1} will
therefore agree with the standard pull-back (resp.~restriction) of Higgs sheaves
discussed in the literature.

\begin{lem}[Pull-back of invariant subsheaves]\label{lem:pbI}
  In the setting of Construction~\ref{cons:pb1}, if $ℱ ⊆ ℰ$ is $θ$-invariant in
  the sense of Definition~\ref{def:ninvarSS}, then
  $ℱ' := \img (f^* ℱ → f^* ℰ) ⊆ f^* ℰ$ is $θ'$-invariant.
\end{lem}
\begin{proof}
  Denote the natural inclusion map as $i: ℱ → ℰ$.  If $ℱ$ is $θ$-invariant, then
  $θ|_{ℱ}$ will factor via $\img \bigl( ℱ ⊗ Ω^{[1]}_X → ℰ ⊗ Ω^{[1]}_X \bigr)$.
  Pulling back, we obtain a commutative diagram
  $$
  \xymatrix{ %
    f^* ℱ \ar@{->>}[r]^(.20){a} \ar[d]_{f^* i} & f^* \img \Bigl( ℱ ⊗ Ω^{[1]}_X → ℰ ⊗ Ω^{[1]}_X \Bigr) \ar[r]^(.65){b} & f^* ℰ ⊗ f^* Ω^{[1]}_X \ar@{=}[d] \\
    f^* ℰ \ar[rr]_{f^* θ} && f^* ℰ ⊗ f^* Ω^{[1]}_X
  }
  $$
  and, by an elementary computation, an inclusion
  $$
  (f^* θ) (ℱ') ⊆ \img b = \img \Bigl( f^* ℱ ⊗ f^* Ω^{[1]}_X → f^* ℰ ⊗ f^* Ω^{[1]}_X \Bigr).
  $$
  The following commutative diagram,
  $$
  \xymatrix{ %
    f^* ℱ ⊗ f^* Ω^{[1]}_X \ar[rr]^{f^* i ⊗ \Id} \ar[d]_{\Id ⊗ d_{\refl} f} && f^* ℰ ⊗ f^* Ω^{[1]}_X \ar[d]^{\Id ⊗ d_{\refl} f} && f^* ℰ \ar[ll]_(.4){f^* θ} \ar@/^5mm/[lld]^(.4){\quad θ' := (\Id ⊗ d_{\refl} f) ◦ f^* θ} \\
    f^* ℱ ⊗ Ω^{[1]}_Y \ar[rr]_{f^* i ⊗ \Id} && f^* ℰ ⊗ Ω^{[1]}_Y, }
  $$
  then yields the claim.
\end{proof}

The following two lemmas are almost immediate.

\begin{lem}[Pull-back as criterion for invariance]\label{lem:pbII}
  In the setting of Construction~\ref{cons:pb1}, assume that $f$ is étale.  If
  $ℱ ⊆ ℰ$ is any subsheaf such that $f^* ℱ ⊆ f^* ℰ$ is $θ'$-invariant, then $ℱ$
  is $θ$-invariant.  \qed
\end{lem}

\begin{lem}[Functoriality with respect to morphisms between spaces]\label{lem:fun2b}
  Given klt pairs $(X,D_X)$ and $(Y,D_Y)$, a normal space $Z$, a Higgs sheaf
  $(ℰ, θ)$ on $X$ and morphisms $g: Z → Y$ and $f: Y → X$, then
  $g^* f^* (ℰ, θ) = (f ◦ g)^* (ℰ, θ)$.  \qed
\end{lem}

\subsection{Reflexive pull-back}
\label{sec:rpull-back}
\approvals{
  Behrouz & yes \\
  Daniel & yes \\
  Stefan & yes \\
  Thomas & yes
}

In the setting of Construction~\ref{cons:pb1}, assume that $(ℰ, θ)$ is a
reflexive Higgs sheaf on $X$ and $f : Y → X$ is a resolution of singularities.
The pull-back $f^* (ℰ, θ)$ is then a Higgs sheaf on $Y$, but $f^* ℰ$ is
generally not torsion free.  In particular, we cannot ask if $f^* (ℰ, θ)$ is
stable as a sheaf with an $Ω¹_Y$-valued operator.  Using smoothness of $Y$, the
following construction avoids this problem by equipping the \emph{reflexive}
pull-back $f^{[*]} ℰ$ with the structure of a Higgs sheaf.

\begin{construction}[Reflexive pull-back of Higgs sheaves]\label{cons:pb2}
  If the variety $Y$ of Construction~\ref{cons:pb1} is smooth, then
  $Ω^{[1]}_Y = Ω¹_Y$ is locally free.  Taking reflexive hulls on either end of
  \eqref{eq:pbhgs}, we obtain an operator
  $$
  f^{[*]} θ: f^{[*]} ℰ → \Bigl( f^* ℰ ⊗ Ω^{[1]}_Y \Bigr)^{**} = f^{[*]} ℰ
  ⊗ Ω¹_Y.
  $$
  The associated map $f^{[*]}θ Λ f^{[*]}θ : f^{[*]} ℰ → f^{[*]} ℰ ⊗ Ω²_Y$
  clearly agrees with $0 = θ' Λ θ'$ wherever $f^* ℰ$ is locally free.  It
  follows that $f^{[*]}θ Λ f^{[*]}θ$ vanishes generically and hence, since
  $f^{[*]} ℰ ⊗ Ω²_Y$ is torsion free, identically.  In summary, we see that
  $f^{[*]}θ$ equips the reflexive pull-back $f^{[*]} ℰ$ with the structure of a
  Higgs sheaf.  We will use the symbols $f^{[*]} (ℰ, θ)$ or
  $(f^{[*]} ℰ, f^{[*]} θ)$.
\end{construction}

\begin{lem}[Reflexive pull-back of invariant subsheaves]\label{lem:pbIII}
  In the setting of Construction~\ref{cons:pb2}, if $ℱ ⊆ ℰ$ is $θ$-invariant,
  write
  $$
  ℱ' := \img (f^* ℱ → f^* ℰ) ⊆ f^* ℰ \quad\text{and}\quad ℱ'' :=
  (ℱ')^{**} ⊆ f^{[*]} ℰ.
  $$
  Then, $ℱ''$ is $f^{[*]}θ$-invariant.
\end{lem}
\begin{proof}
  Since $Ω¹_Y$ is locally free, $ℱ'' ⊗ Ω¹_Y$ is a subsheaf of
  $f^{[*]} ℰ ⊗ Ω¹_Y$, and Lemma~\ref{lem:pbI} gives a commutative diagram
  \begin{equation}\label{eq:dfdlhk}
    \begin{gathered}
      \xymatrix{ %
        ℱ' \ar[r] \ar@{^(->}[d] & ℱ' ⊗ Ω¹_Y \ar@{^(->}[d] \\
        f^* ℰ \ar[r]_(.4){θ'} & f^* ℰ ⊗ Ω¹_Y
      }
    \end{gathered}
  \end{equation}
  Taking reflexive hulls is a left-exact functor.  Applied to \eqref{eq:dfdlhk},
  it will thus give the desired inclusion
  $f^{[*]} θ(ℱ'') ⊆ ℱ'' ⊗ Ω¹_Y ⊆ f^{[*]} ℰ ⊗ Ω¹_Y$.
\end{proof}

\begin{obs}[Weak functoriality with respect to morphisms between spaces]\label{obs:fun4}
  Assume we are given klt pairs $(X,D_X)$ and $(Y,D_Y)$, a smooth space $Z$, a
  sheaf $ℰ$ on $X$ and morphisms $g: Z → Y$ and $f: Y → X$.  Then, there exists
  a canonical morphism $c : (f ◦ g)^{[*]} ℰ → g^{[*]}f^{[*]} ℰ$.  If we assume
  additionally that $f^* ℰ$ is reflexive, then $c$ is isomorphic and given any
  Higgs-field $θ$, one verifies immediately that
  $g^{[*]} f^* (ℰ, θ) = (f ◦ g)^{[*]} (ℰ, θ)$.
\end{obs}

\begin{warning}[No full functoriality with respect to morphisms between spaces]\label{warn:nofunct}
  We have seen in Lemma~\ref{lem:fun2b} that pull-back of Higgs sheaves is fully
  functorial with respect to morphisms between spaces.  There is no full
  analogue of this for reflexive pull-back.  In fact, taking reflexive hulls
  does in general not commute with pull-back, the morphism $c$ of
  Observation~\ref{obs:fun4} will in general not be isomorphic, and
  functoriality fails already at the level of sheaves, without any additional
  Higgs structure.  \Publication{The arXiv version of this paper discusses an
    example in detail.}\Preprint{For an example, consider the following
    well-known sequence of morphisms
    $$
    \xymatrix{ %
      Z \ar[r]^g & Y \ar[r]^f & X
    }
    $$
    that is obtained as follows.  Embed $ℙ¹ ⨯ ℙ¹$ into $ℙ³$ and let $X$ be the
    cone over it, which has an isolated singular point $x ∈ X$.  Let $Z$ be the
    blow-up of $X$, which is smooth, and let $Y$ one of the obvious small
    intermediate desingularisations.  Finally, let $D ⊂ X$ be a reduced
    Weil-divisor that is not $ℚ$-Cartier, and whose strict transform $D_Y ⊂ Y$
    does not contain the $f$-exceptional set.  If $D_Z ⊂ Z$ is the strict
    transform of $D$ in $Z$ and $E ⊂ Z$ the preimage of $x$, then
    $$
    f^{[*]} 𝒥_D = 𝒥_{D_Y} \quad \text{and} \quad g^{[*]} f^{[*]} 𝒥_D =
    𝒥_{D_Z}
    $$
    while $(f◦ g)^{[*]} 𝒥_D = 𝒥_{D_Z+E}$.}
\end{warning}

\subsection{Higgs sheaves on $ℚ$-varieties}
\label{ssec:Qhiggs}
\approvals{
  Behrouz & yes \\
  Daniel & yes \\
  Stefan & yes \\
  Thomas & yes
}

The definition of $ℚ$-sheaves given in Section~\ref{subsect:QQsheaves} has an
obvious analogue for Higgs sheaves.

\begin{defn}[Higgs $ℚ$-sheaf and $ℚ$-bundle]\label{def:QHS}
  Setup and notation as in Definition~\ref{def:Qvar}.  A \emph{Higgs $ℚ$-sheaf}
  $(ℰ, θ)$ on $X_ℚ$ is a tuple
  $$
  \bigl( \{ (ℰ_α, θ_α)\}_{α ∈ A}, \{i_{αβ}\}_{(α,β) ∈ A⨯A}\bigr)
  $$
  consisting of a family of Higgs sheaves $(ℰ_α, θ_α)$ on $X_α$ plus
  isomorphisms
  $$
  i_{αβ} : p_{αβ,α}^*(ℰ_α, θ_α) → p_{αβ,β}^*(ℰ_β, θ_β)
  $$
  that are compatible on the triple overlaps.  The Higgs $ℚ$-sheaf $(ℰ, θ)$ is
  called \emph{reflexive} if all the $ℰ_α$ are reflexive.  It is called
  \emph{Higgs $ℚ$-bundle} if all the $ℰ_α$ are locally free.
\end{defn}

In complete analogy to Construction~\ref{cons:rpb1}, any Higgs sheaf on $X$
pulls back to a reflexive Higgs $ℚ$-sheaf on $X_ℚ$.

\begin{construction}[Construction of Higgs $ℚ$-sheaf by reflexive pull-back]\label{cons:rpb2}
  Given a quasi-étale $ℚ$-variety $X_ℚ := \bigl(X, \{ p_α\}_{α ∈ A} \bigr)$,
  recall from \cite[Prop.~5.20]{KM98} that $X$ is necessarily klt.  In
  particular, there exists reflexive pull-back from Higgs sheaves on $X$ to
  reflexive Higgs sheaves on the manifolds $X_α$.  We can thus define a
  reflexive Higgs $ℚ$-sheaf $(ℰ, θ)^{[ℚ]}$ on $X_ℚ$, setting
  $(ℰ_α, θ_α) := p_α^{[*]} (𝒢, θ)$---the existence of natural isomorphisms
  $i_{αβ}$ is guaranteed by étalité of $p_{αβ,α}$ and $p_{αβ,β}$.
\end{construction}

As with $ℚ$-sheaves, any Higgs $ℚ$-sheaf on a $ℚ$-variety pulls back to an
honest Higgs sheaf on any global cover.  The following are direct analogues of
the appropriate statements for $ℚ$-sheaves that are found in
Section~\ref{subsect:QQsheaves}.

\begin{fact}[Induced Higgs $G$-sheaf on global cover]\label{fact:isHgc}
  In the setting of Definition~\ref{def:QHS}, assume we are given a global cover
  $γ : \what{X} → X$ as in Section~\ref{ssec:globalCover}, which is Galois with
  group $G$.  Then, the pull-back Higgs sheaves $q_α^*(ℰ_α, θ_α)$ glue to give a
  Higgs $G$-sheaf $(\what{ℰ}, \what{θ})$ on $\what{X}$.  If the Higgs $ℚ$-sheaf
  $(ℰ,θ)$ is reflexive, then $\what{ℰ}$ is locally free in codimension two.  If
  $(ℰ,θ)$ is reflexive and $\what{X}$ is Cohen-Macaulay, then
  $(\what{ℰ}, \what{θ})$ is likewise reflexive.  \qed
\end{fact}

A Higgs $ℚ$-sheaf does not only induce an honest Higgs-sheaf on any global
cover, but also on any resolution of singularities of global covers that are
Cohen-Macaulay.  This can again be seen as a form of reflexive pull-back, this
time from the global cover (which need not be klt) to the resolution of
singularities.

\begin{lem}[Induced Higgs $G$-sheaf on resolution of global cover]\label{lem:rpbfgc}
  Given a $ℚ$-variety $X$, a reflexive Higgs $ℚ$-sheaf $(ℰ, θ)$, a global cover
  $\what{X}$ with Galois group $G$ and induced Higgs sheaf
  $\bigl( \what{ℰ}, \what{θ} \bigr)$, let $π : \wtilde{X} → \what{X}$ be a
  $G$-equivariant resolution of singularities.  Set
  $\wtilde{ℰ} := π^{[*]} \what{ℰ}$.  If $\what{X}$ is Cohen-Macaulay, then there
  exists a $G$-invariant Higgs field $\wtilde{θ}$ on $\wtilde{ℰ}$, such that the
  Higgs $G$-sheaf $\bigl( \wtilde{ℰ}, \wtilde{θ} \bigr)$ agrees with the
  reflexive $π$-pull back of $\bigl( \what{ℰ}, \what{θ} \bigr)$ over the maximal
  open set where $\what X$ is klt (and where reflexive $π$-pull pull-back is
  therefore defined).
\end{lem}
\begin{proof}
  To define a $G$-invariant Higgs field on $\wtilde{ℰ}$, we denote the charts of
  the $ℚ$-variety $X_ℚ$ by $(X, \{ p_α\}_{α ∈ A})$, and use the notation for
  global covers introduced in Section~\ref{ssec:globalCover}.  Setting
  $\wtilde{X}_α := π^{-1}(\what{X}_α)$, the following diagrams summarise our
  situation
  $$
  \xymatrix{ %
    \wtilde{X}_α \ar[rr]^{π_α := π|_{\wtilde{X}_α}} && \what{X}_α \ar[rr]^{q_α}_{·/H_α} && X_α \ar[rr]_{·/G_α} \ar@/^.4cm/[rrrr]^{p_α} && U_α \ar[rr]_{p'_α\text{, étale}} && X.
  }
  $$
  Set
  $\bigl( \wtilde{ℰ}_α, \wtilde{θ}_α \bigr) := (q_α ◦ π_α)^{[*]} (ℰ_α, θ_α)$.
  Using the assumption that $\what{X}$ is Cohen-Macaulay, recall from
  Observation~\ref{obs:CM1} that $q_α^* ℰ_α$ is reflexive.  In particular, it
  follows directly that $\wtilde{ℰ}_α = \wtilde{ℰ}|_{\wtilde{X}_α}$.  More is
  true.  Over the open set where $\what X$ is smooth and pull-back of Higgs
  sheaves is therefore defined, it follows from weak functoriality,
  Observation~\ref{obs:fun4}, that
  $$
  \bigl( \wtilde{ℰ}_α, \wtilde{θ}_α \bigr) = π_α^{[*]} q_α^* (ℰ_α, θ_α) =
  π_α^{[*]} \bigl(\what{ℰ}|_{\what{X}_α}, \what{θ}_α|_{\what{X}_α} \bigr).
  $$
  In particular, we see that the $G$-invariant Higgs fields $\wtilde{θ}_α$ agree
  over this dense open set.  Since $\wtilde{X}$ is smooth, two Higgs fields on
  the torsion free sheaf $\wtilde{ℰ}$ agree if they agree on an open set.  It
  follows that the $\wtilde{θ}_α$ glue to give a globally defined Higgs
  $G$-sheaf $\bigl( \wtilde{ℰ}, \wtilde{θ} \bigr)$ that agrees with the
  reflexive $π$-pull back of $\bigl( \what{ℰ}, \what{θ} \bigr)$ wherever that
  pull-back is defined.
\end{proof}

\begin{notation}[Reflexive pull-back from global cover]\label{not:rpbfgc}
  In the setting of Lemma~\ref{lem:rpbfgc}, we write
  $π^{[*]} \bigl( \what{ℰ}, \what{θ} \bigr) := \bigl( \wtilde{ℰ}, \wtilde{θ}
  \bigr)$, and refer to this sheaf as the reflexive pull-back.
\end{notation}

Since this new piece of terminology agrees with the old one as soon as
$\what{X}$ is klt, we do not expect this to lead to any confusion.

\subsection{Stability}
\label{ssec:stability}
\approvals{
  Behrouz & yes \\
  Daniel & yes \\
  Stefan & yes \\
  Thomas & yes
}

A Higgs sheaf is stable if it is stable as a sheaf with an $Ω^{[1]}_X$-valued
operator, cf.\ Definition~\vref{defn:swostab1}.  For later use, the following
propositions, describing the behaviour of stability under pull-backs, will be
useful.

\begin{prop}[$G$-stability under birational pull-back]\label{prop:1}
  Let $(X,D)$ be a projective klt pair, where $X$ is equipped with an action of
  a finite group $G$.  Let $H$ be any nef, $ℚ$-Cartier $ℚ$-divisor on $X$ and
  $(ℰ, θ)$ be any Higgs $G$-sheaf, where $ℰ$ is torsion free.  Given a
  birational morphism $π : \wtilde{X} → X$ of projective $G$-varieties with
  $\wtilde{X}$ smooth, then $(ℰ, θ)$ is $G$-stable (resp.\ semistable) with
  respect to $H$ if and only if $π^{[*]}(ℰ, θ)$ is $G$-stable (resp.\
  semistable) with respect to $π^*H$.
\end{prop}
\begin{proof}
  Given any number $s ∈ ℚ$, we need to show that the following two statements
  are equivalent.
  \begin{enumerate}
  \item\label{il:Tx1} There exists a $G$-subsheaf $0 \ne ℱ ⊆ ℰ$ with slope
    $μ_H(ℱ) ≥ s$ that is generically $θ$-invariant.

  \item\label{il:Tx2} There exists a $G$-subsheaf $0 \ne \wtilde{ℱ} ⊆ π^{[*]} ℰ$
    with slope $μ_{π^* H}( \wtilde{ℱ} ) ≥ s$ that is generically
    $π^{[*]}θ$-invariant.
  \end{enumerate}
  To this end, let $X° ⊆ X_{\reg}$ be the maximal open set where $π$ is
  isomorphic, and observe that $X°$ is a big, $G$-invariant subset of $X$.  We
  set $\wtilde{X}° := π^{-1}(X°)$.

  \subsubsection*{\ref{il:Tx1}$⇒$\ref{il:Tx2}}

  Given a sheaf $ℱ$ as in \ref{il:Tx1}, set $\wtilde{ℱ}' := f^{[*]} ℱ$.  This is
  a $G$-invariant subsheaf of $f^{[*]} ℰ$ whose restriction to $\wtilde{X}°$ is
  $f^{[*]}θ$-invariant.  Its saturation $\wtilde{ℱ}$ is $G$-invariant, and, by
  Lemma~\ref{lem:x11}, generically $π^{[*]}θ$-invariant.  The ranks of
  $\wtilde{ℱ}$ and $ℱ$ agree, the slope only increases in the process.

  \subsubsection*{\ref{il:Tx2}$⇒$\ref{il:Tx1}}

  Given a sheaf $\wtilde{ℱ}$ as in \ref{il:Tx2}, use the identification
  $\wtilde{X}° ≅ X°$ to view $\wtilde{ℱ}|_{\wtilde{X}°}$ as a
  $θ|_{X°}$-invariant sheaf $ℱ°$ on $X°$.  Recall from \cite[I.Thm.~9.4.7 and
  0.Sect.~5.3.2]{EGA1} that there exists a coherent subsheaf extension of $ℱ°$
  to $X$, that is, a coherent subsheaf $ℱ' ⊆ ℰ$ whose restriction to $X°$ equals
  $ℱ°$.  As before, Lemma~\ref{lem:x11} guarantees that its saturation
  $ℱ := (ℱ')^{sat}$ is generically $θ$-invariant.  The ranks of $\wtilde{ℱ}$ and
  $ℱ$ agree, the slope only increases in the process.
\end{proof}

The following is an analogue for morphisms that are generically Galois, say with
group $G$.  It differs from Proposition~\ref{prop:1} in that it compares
$G$-stability on the domain to normal stability on the target of the morphism.
\Publication{Its proof uses Proposition~\ref{prop:1:1} to descent sheaves from
  $\wtilde{X}$ to $X$, but is otherwise completely similar to that of
  Proposition~\ref{prop:1}.  The arXiv version of this paper contains the full
  argument.}

\begin{prop}[Stability under generically Galois pull-back]\label{prop:stability_upstairs_downstairs}
  Let $(X,D)$ be a projective, klt pair, let $H$ be any nef, $ℚ$-Cartier
  $ℚ$-divisor on $X$ and $(ℰ, θ)$ be any Higgs sheaf, where $ℰ$ is torsion free.
  Given a sequence of morphisms between normal, projective varieties,
  $$
  \xymatrix{ %
    \wtilde{X} \ar@/^4mm/[rrrrrr]^f \ar[rrr]_{π\text{, $G$-equivar.\ biratl.}} &&& \what{X} \ar[rrr]_{γ\text{, Galois with group }G} &&& X,
  }
  $$
  with $\wtilde{X}$ smooth, then $(ℰ, θ)$ is stable (resp.\ semistable) with
  respect to $H$ if and only if $f^{[*]}(ℰ, θ)$ is $G$-stable (resp.\
  semistable) with respect to $π^*H$.\Publication{\qed}
\end{prop}
\Preprint{ %
\begin{proof}
  Given any number $s ∈ ℚ$, we need to show that the following two statements
  are equivalent.
  \begin{enumerate}
  \item\label{il:T1} There exists a subsheaf $ℱ ⊆ ℰ$ with slope $μ_H(ℱ) ≥ s$
    that is generically $θ$-invariant.

  \item\label{il:T2} There exists a $G$-subsheaf $\wtilde{ℱ} ⊆ f^{[*]} ℰ$ with
    slope $μ_{π^* H}( \wtilde{ℱ} ) ≥ s$ that is generically
    $f^{[*]}θ$-invariant.
  \end{enumerate}
  Let $X° ⊆ X$ be the maximal open set such that $ℰ|_{X°}$ is locally free, $X°$
  is smooth, and $\what{X}° := γ^{-1}(X°)$ is smooth and isomorphic to
  $\wtilde{X}° := f^{-1}(X°)$.  Observe that $X°$ is a big subset of $X$.

  \subsubsection*{\ref{il:T1}$⇒$\ref{il:T2}}

  Given a sheaf $ℱ$ as in \ref{il:T1}, set $\wtilde{ℱ}' := f^{[*]} ℱ$.  This is
  a $G$-invariant subsheaf of $f^{[*]} ℰ$ whose restriction to $\wtilde{X}°$ is
  $f^{[*]}θ$-invariant.  Its saturation $\wtilde{ℱ}$ is clearly $G$-invariant,
  and, by Lemma~\ref{lem:x11}, generically $π^{[*]}θ$-invariant.  The ranks of
  $\wtilde{ℱ}$ and $ℱ$ agree, the slope only increases in the process.

  \subsubsection*{\ref{il:T2}$⇒$\ref{il:T1}}

  Given a $G$-subsheaf $\wtilde{ℱ}$ as in \ref{il:T2}, Lemma~\ref{lem:x11}
  allows us to assume that $\wtilde{ℱ}$ is saturated in $f^{[*]}ℰ$.
  Proposition~\ref{prop:1:1} guarantees the existence of a saturated subsheaf
  $ℱ° ⊆ ℰ|_{X°}$ such that $f^{[*]}ℱ° = \overline{ℱ}|_{\wtilde{X}°}$.  Over the
  open set $X^{◦◦} ⊆ X°$ where $f$ is étale, the sheaf $ℱ°$ is clearly
  $θ$-invariant.

  As before, recall from \cite[I.Thm.~9.4.7 and 0.Sect.~5.3.2]{EGA1} that there
  exists a coherent extension of $ℱ°$ to $X$, that is, a coherent subsheaf
  $ℱ' ⊆ ℰ$ whose restriction to $X°$ equals $ℱ°$.  Let $ℱ := (ℱ')^{sat}$ be its
  saturation in $ℰ$, which, by Lemma~\ref{lem:x11} is generically $θ$-invariant.
  The ranks of $\wtilde{ℱ}$ and $ℱ$ agree, the slope only increases in the
  process.
\end{proof}}

Consider the setting of Proposition~\ref{prop:stability_upstairs_downstairs} in
the special case where $γ$ is quasi-étale.  The pair
$\bigl( \what{X}, γ^* D \bigr)$ is then klt, and reflexive pull-back $π^{[*]}$
from $\what{X}$ to $\wtilde{X}$ exists.  The Higgs sheaves $f^{[*]}(ℰ,θ)$ and
$π^{[*]}γ^{[*]}(ℰ,θ)$, however, need not agree, cf.\ Warning~\ref{warn:nofunct}.
More generally, given a commutative diagram of morphisms between supporting
spaces of klt pairs, failure of functoriality will frequently lead to a large
number of potentially different reflexive pull-back Higgs sheaves, each
corresponding to one particular path through the diagram.  The following
proposition will often be used to compare their stability properties.

\begin{prop}[Comparison of $G$-stability]\label{prop:comparison}
  Let $X$ be a normal, projective variety, let $G$ be a finite group that acts
  on $X$, let $H$ be any nef, $ℚ$-Cartier $ℚ$-divisor on $X$ and
  $π : \wtilde{X} → X$ be a projective, birational, $G$-equivariant morphism,
  where $\wtilde{X}$ is smooth.  Let $E ⊆ \wtilde{X}$ be the $π$-exceptional set
  and assume that we are given two Higgs $G$-sheaves on $\wtilde X$, say
  $(ℰ¹, θ¹)$ and $(ℰ², θ²)$, that agree as Higgs $G$-sheaves away from $E$.
  Then, the following two statements are equivalent for any pair of numbers
  $r ∈ ℕ$, $s ∈ ℚ$.
  \begin{enumerate}
  \item\label{il:Ts1} There exists a $G$-invariant subsheaf $ℱ¹ ⊆ ℰ¹$ with
    $\rank ℱ = r$ and slope $μ_{π^*H}(ℱ¹) ≥ s$ that is $θ¹$-invariant.

  \item\label{il:Ts2} There exists a $G$-invariant subsheaf $ℱ² ⊆ ℰ²$ with
    $\rank ℱ = r$ and slope $μ_{π^*H}(ℱ²) ≥ s$ that is $θ²$-invariant.
  \end{enumerate}
  In particular, $(ℰ¹, θ¹)$ is $G$-stable (resp.\ $G$-semistable) with respect
  to $π^*H$ if and only if $(ℰ², θ²)$ is.
\end{prop}
\begin{proof}
  By symmetry, it suffices to show \ref{il:Ts1}$⇒$\ref{il:Ts2}.  Given a
  subsheaf $ℱ¹$ as in \ref{il:Ts1}, consider the open set
  $\wtilde{X}° := \wtilde{X} \setminus E$ and recall from \cite[I.Thm.~9.4.7 and
  0.Sect.~5.3.2]{EGA1} that there exists a subsheaf $𝒢 ⊆ ℰ²_{\wtilde X}$ whose
  restriction to $\wtilde{X}°$ equals $ℱ¹$.  Replacing $𝒢$ by
  $\sum_{g ∈ G} g^* 𝒢 ⊆ ℰ²_{\wtilde X}$ if needed, we may assume without loss of
  generality that $𝒢$ is $G$-invariant.  Next, recall from Item~\ref{il:Be} of
  Lemma~\ref{lem:elemSlp2} that $μ_{π^* H} (𝒢) = μ_{π^*H}(ℱ¹) ≥ s$.  Let
  $ℱ² ⊆ ℰ²$ be the saturation of $𝒢$, observe that $ℱ² ⊆ ℰ²$ is again
  $G$-invariant, and recall from Lemma~\ref{lem:x11} that $ℱ²$ is invariant with
  respect to $θ²$.
\end{proof}

\subsection{The restriction theorem for Higgs sheaves}
\label{ssect:restrict}
\approvals{Behrouz & yes \\
  Daniel & yes \\
  Stefan & yes \\
  Thomas & yes}

This subsection establishes the restriction theorem for stable Higgs sheaves,
which will be crucial for the proof of our main results.  For Higgs bundles on
manifolds with ample polarisation, the theorem appears in Simpson's work,
\cite[Lem.~3.7]{MR1179076}, referring to ``arguments of Mehta and Ramanathan''
for a restriction theorem for sheaves with operators\Preprint{\ analogous to our
  Theorem~\ref{thm:nrestrSWO}}.  \Publication{Our proof instead cites a
  restriction theorem for sheaves with operators from the work of Langer,
  \cite[Thm.~9]{MR3314517}.  He works in positive characteristic but says that,
  \emph{mutatis mutandis}, his arguments will also work in characteristic zero,
  cf.\ \cite[Page~906]{MR3314517}.  For clarity's sake, the arXiv version of
  this paper contains a statement of the precise result needed and a short,
  self-contained proof.}

\begin{thm}[Restriction theorem for stable Higgs sheaves]\label{thm:restriction}
  Let $(X,Δ)$ be a projective klt pair of dimension $n ≥ 2$, let $H ∈ \Div(X)$
  be an ample, $ℚ$-Cartier $ℚ$-divisor and let $(ℰ, θ)$ be a torsion free Higgs
  sheaf on $X$ of positive rank.  Assume that $(ℰ, θ)$ is stable with respect to
  $H$.  If $m ≫ 0$ is sufficiently large and divisible, then there exists a
  dense open set $U ⊆ |m·H|$ such that the following holds for any hyperplane
  $D ∈ U$ with associated inclusion map $ι : D → X$.
  \begin{enumerate}
  \item The hyperplane $D$ is normal, connected and not contained in $\supp Δ$.
    The pair $(D,Δ|_D)$ is klt.
  \item The sheaf $ℰ|_D$ is torsion free.  The Higgs sheaf $ι^{*}(ℰ, θ)$ is
    stable with respect to $H|_D$.
  \end{enumerate}
\end{thm}

\subsubsection*{Proof.  Step 1: Setup}
\approvals{
  Behrouz & yes \\
  Daniel & yes \\
  Stefan & yes \\
  Thomas & yes
}

Choose a strong, log resolution of singularities, say $π : \wtilde{X} → X$.  We
have seen in Proposition~\ref{prop:1} that $π^{[*]} (ℰ, θ)$ is stable with
respect to $\wtilde H := π^*H$.  Set $\wtilde{ℰ} := π^{[*]} ℰ$ and
$r := \rank ℰ$.

\begin{notation}
  Given sheaves $𝒜$ on $X$, $ℬ$ on $\wtilde X$ and $𝒞$ on a subvariety
  $\wtilde D ⊆ \wtilde X$, we write $\deg 𝒜 := \deg_H 𝒜$,
  $\deg ℬ := \deg_{\wtilde H} ℬ$, $\deg 𝒞 := \deg_{\wtilde H|_{\wtilde D}} 𝒞$
  and similarly with $μ$ and $μ^{\max}$.
\end{notation}

Twisting $ℰ$ with a sufficiently ample, invertible sheaf, Lemma~\ref{lem:stabtp}
allows us to assume that the following condition holds in addition.

\begin{asswlog}\label{ass:spoE}
  The numbers $μ(ℰ)$ and $μ^{\max}(\wtilde{ℰ})$ are positive.
\end{asswlog}

\subsubsection*{Step 2: Choice of $m$}
\CounterStep
\approvals{Behrouz & yes \\
  Daniel & yes \\
  Stefan & yes \\
  Thomas & yes}

Choosing $m ≫ 0$ sufficiently large and divisible, the following will hold.

\begin{enumerate}
\item\label{il:1} The divisor $m·H$ is integral, Cartier and very ample.

\item\label{il:2} Flenner's restriction theorem holds for $ℰ$, cf.\
  \cite[Thm.~1.2]{Flenner84}.  In particular, if $D ∈ |m·H|$ is general, then
  $μ^{\max}(ℰ|_D) = μ^{\max}(ℰ)$.
  
\item\label{il:3} The number $m$ satisfies the condition spelled out in the
  restriction theorem for sheaves with an operator, when applying the theorem to
  the Higgs sheaf $π^{[*]}(ℰ, θ)$ as a sheaf with an $Ω¹_{\wtilde{X}}$-valued
  operator,
  \Preprint{Theorem~\ref{thm:nrestrSWO}}\Publication{\cite[Thm.~9]{MR3314517}
    but see also the appendix in the arXiv version of this paper}.
  
\item\label{il:4} We have a strict inequality
  $2r·μ^{\max}(\wtilde{ℰ}) < [m·H]^n$.
\end{enumerate}

\subsubsection*{Step 3: Choice of $U$}
\CounterStep
\approvals{
  Behrouz & yes \\
  Daniel & yes \\
  Stefan & yes \\
  Thomas & yes
}

Next, observe that there exists an open subset $U ⊆ |m·H|$ such that the
following holds for all hyperplanes $D ∈ U$ and their preimages
$\wtilde D := π^{-1}(D)$.
  
\begin{enumerate}
\item\label{il:balthasar} The hyperplane $D$ is reduced, irreducible, normal and
  not contained in $\supp Δ$, Item~\ref{il:1} and Seidenberg's theorem
  \cite[Thm.~1.7.1]{BS95}.  Its preimage $\wtilde D$ is smooth, Item~\ref{il:1}
  and Bertini.  The pair $(D, Δ|_D)$ is klt, Item~\ref{il:1} and
  \cite[Lem.~5.17]{KM98}.

  In particular, there exists a reflexive pull-back functor from Higgs sheaves
  on $D$ to Higgs sheaves on $\wtilde D$.
  
\item\label{il:caspar} The restrictions $ℰ|_D$ and $\wtilde{ℰ}|_{\wtilde D}$ are
  reflexive, \cite[Thm.~12.2.1]{EGA4-3}.
    
\item\label{il:melchior} If $𝒜 ⊆ \wtilde{ℰ}|_{\wtilde D}$ is any subsheaf, then
  $μ(𝒜) ≤ μ^{\max}(\wtilde{ℰ})$, Item~\ref{il:2} and Lemma~\ref{lem:elemSlp2}.
\end{enumerate}

Choose one hyperplane $D ∈ U$ and fix that choice for the remainder of the
proof.  As before, write $\wtilde D := π^{-1}(D)$ and consider the diagram
$$
\xymatrix{ %
  \wtilde{D} \ar[rr]^{\wtilde{ι}\text{, inclusion}} \ar[d]_{π|_{\wtilde{D}}\text{, desing.}} \ar[rrd]^{δ} && \wtilde{X} \ar[d]^{π\text{, desing.}} \\
  D \ar[rr]_{ι\text{, inclusion}} && X.
}
$$
Pulling back, we can equip all spaces considered so far with naturally defined
Higgs sheaves, which we list here for the reader's convenience.
\begin{align*}
  (ℰ, θ) &&& …\text{ Higgs sheaf on $X$ that is initially given} \\
  (ℰ_D, θ_D) & := ι^* (ℰ, θ) && …\text{ Higgs sheaf on $D$, equals $ι^{[*]} (ℰ, θ)$ by \ref{il:caspar}} \\
  (\wtilde{ℰ}, \wtilde{θ}) & := π^{[*]} (ℰ, θ) && … \text{ Higgs sheaf on $\wtilde{X}$} \\
  (\wtilde{ℰ}|_{\wtilde{D}}, \wtilde{θ}|_{\wtilde{D}}) &&& …\text{ Refl.\ sheaf on $\wtilde{D}$ with $Ω¹_{\wtilde X}|_{\wtilde{D}}$-valued operator} \\
  (\wtilde{ℰ}_{\wtilde{D}}, \wtilde{θ}_{\wtilde{D}}) & := \wtilde{ι}^{\:*} (\wtilde{ℰ}, \wtilde{θ}) && …\text{ Higgs sheaf on $\wtilde{D}$, equals $\wtilde{ι}^{\:[*]}(\wtilde{ℰ}, \wtilde{θ})$ by \ref{il:caspar}} \\
  (\wtilde{ℰ_D}, \wtilde{θ_D}) & := (π|_{\wtilde{D}})^{[*]}(ℰ_D, θ_D) && …\text{ Higgs sheaf on $\wtilde{D}$, equals $δ^{[*]} (ℰ, θ)$ by \ref{il:caspar}} \\
           &&& \hphantom{…}\text{\quad and Observation~\ref{obs:fun4}}
\end{align*}
We do not claim that the two Higgs sheaves on $\wtilde{D}$, namely
$(\wtilde{ℰ}_{\wtilde{D}}, \wtilde{θ}_{\wtilde{D}})$ and
$(\wtilde{ℰ_D}, \wtilde{θ_D})$ necessarily agree, although they certainly agree
outside of the $π|_{\wtilde{D}}$-exceptional set.  We will compare these sheaves
in the last step of this proof.

\subsubsection*{Step 4: Numerical computations}
\approvals{
  Behrouz & yes \\
  Daniel & yes \\
  Stefan & yes \\
  Thomas & yes
}

We aim to show that $(ℰ_D, θ_D)$ is stable, or equivalently, that
$(\wtilde{ℰ_D}, \wtilde{θ_D})$ is stable.  For this, we will first establish
stability of $(\wtilde{ℰ}_{\wtilde{D}}, \wtilde{θ}_{\wtilde{D}})$ in Step~5 of
this proof.  The following numerical computation is instrumental.

\begin{claim}\label{claim:xya}
  If $ℱ ⊆ \wtilde{ℰ}|_{\wtilde D}$ is any saturated subsheaf with
  $μ(ℱ) ≥ μ(ℰ) = μ(\wtilde{ℰ})$ and if $𝒜$ is any subsheaf of the quotient
  $𝒬 := \bigl( \wtilde{ℰ}|_{\wtilde D} \bigr)/ℱ$, then
  $\deg 𝒜 ≤ r·μ^{\max}(\wtilde{ℰ})$.
\end{claim}
\begin{proof}[Proof of Claim~\ref{claim:xya}]
  Let $q : \wtilde{ℰ}|_{\wtilde D} → 𝒬$ be the natural projection and consider
  the exact sequence
  $$
  0 → ℱ → q^{-1} 𝒜 → 𝒜 → 0.
  $$
  We obtain that $\rank(q^{-1} 𝒜) = \rank(𝒜)+\rank(ℱ)$ and
  \begin{align*}
    \deg 𝒜 & = \deg q^{-1} 𝒜 - \deg ℱ \\
             & = \rank(q^{-1} 𝒜)·μ(q^{-1} 𝒜) - \rank(ℱ)·μ(ℱ) &&\text{Definition of $μ$} \\
             & ≤ \rank(q^{-1} 𝒜)·μ^{\max}(\wtilde{ℰ}) - \rank(ℱ)·μ(ℱ) &&\text{Item~\ref{il:melchior}} \\
             & ≤ \rank(q^{-1} 𝒜)·μ^{\max}(\wtilde{ℰ}) - \rank(ℱ)·μ(\wtilde{ℰ}) &&\text{Assumption on $ℱ$} \\
             & ≤ r·μ^{\max}(\wtilde{ℰ}) &&\text{Assumption~\ref{ass:spoE}}
  \end{align*}
  and Claim~\ref{claim:xya} follows.
\end{proof}
  
\begin{consequence}\label{cons:xya}
  In the setting of Claim~\ref{claim:xya}, if
  $ℬ ⊆ 𝒬 ⊗ 𝒪_{\wtilde D}(-\wtilde D)$ is of positive rank, then
  $\deg ℬ ≤ r·μ^{\max}(\wtilde{ℰ}) - \rank(ℬ)·[\wtilde D]^{\dim X} ≤
  r·μ^{\max}(\wtilde{ℰ}) - [\wtilde D]^{\dim X}$.  \qed
\end{consequence}

\subsubsection*{Step 5: Stability of $(\wtilde{ℰ}_{\wtilde{D}}, \wtilde{θ}_{\wtilde{D}})$}
\approvals{
  Behrouz & yes \\
  Daniel & yes \\
  Stefan & yes \\
  Thomas & yes
}

With Consequence~\ref{cons:xya} at hand, stability of
$(\wtilde{ℰ}_{\wtilde{D}}, \wtilde{θ}_{\wtilde{D}})$ can now be established
following the line of argument outlined by Simpson, \cite[p.~38]{MR1179076}.
  
\begin{claim}\label{claim:restus}
  The Higgs sheaf $(\wtilde{ℰ}_{\wtilde{D}}, \wtilde{θ}_{\wtilde{D}})$ is stable
  with respect to $\wtilde{H}|_{\wtilde D}$.
\end{claim}
\begin{proof}[Proof of Claim~\ref{claim:restus}]
  Argue by contradiction and assume that there exists a generically
  $\wtilde{θ}_{\wtilde{D}}$-invariant subsheaf $ℱ ⊆ \wtilde{ℰ}_{\wtilde{D}}$
  with $μ(ℱ) ≥ μ(ℰ)$.  Lemma~\ref{lem:nsat} allows us to assume that $ℱ$ is a
  saturated subsheaf of $\wtilde{ℰ}|_{\wtilde D}$.  For this, note that the
  slope of a sheaf increases when passing to the saturation.  Consider the
  standard conormal bundle sequence for the submanifold $\wtilde D ⊂ \wtilde X$,
  twisted by $𝒬 := (\wtilde{ℰ}_{\wtilde D})/ℱ$,
  \begin{equation}\label{eq:xxl}
    0 \longrightarrow 𝒬 ⊗ 𝒪_{\wtilde D}(-\wtilde D) \overset{α}{\longrightarrow}
    𝒬 ⊗ Ω¹_{\wtilde X}|^{\vphantom{1}}_{\wtilde D} \overset{β}{\longrightarrow}
    𝒬 ⊗ Ω¹_{\wtilde D} \longrightarrow 0
  \end{equation}
  and the composition $γ$ of the following two morphisms,
  $$
  ℱ \xrightarrow{\;\wtilde{θ}|_{\wtilde{D}}\;} \wtilde{ℰ}|_{\wtilde D} ⊗
  Ω¹_{\wtilde X}|^{\vphantom{1}}_{\wtilde D} \longrightarrow 𝒬 ⊗ Ω¹_{\wtilde
    X}|^{\vphantom{1}}_{\wtilde D}.
  $$
  Recalling from Condition~\ref{il:3} that $\wtilde{ℰ}|_{\wtilde D}$ is stable
  as a sheaf with the $Ω¹_{\wtilde X}|^{\vphantom{1}}_{\wtilde D}$-valued
  operator $\wtilde{θ}|_{\wtilde{D}}$, it follows that $ℱ$ is \emph{not}
  generically invariant under that operator.  In other words, the composed map
  $γ$ is not generically zero.  In contrast, the assumption that the sheaf $ℱ$
  \emph{is} generically a Higgs subsheaf implies that the map $β ◦ γ$ is
  necessarily zero.  Exactness of \eqref{eq:xxl} then gives a non-zero map
  $τ : ℱ → 𝒬 ⊗ 𝒪_{\wtilde D}(-\wtilde D)$.  We will now show by way of numerical
  computation that such a map cannot exist.  To this end, observe on the one
  hand that
  \begin{align*}
    \deg (\img τ) & = \deg (ℱ) - \deg (\ker τ) \\
                  & ≥ \rank ℱ · μ (\wtilde{ℰ}) - \deg (\ker τ) && \text{Choice of }ℱ \\
                  & ≥ \rank ℱ · μ (\wtilde{ℰ}) - \rank (\ker τ) · μ^{\max} (\wtilde{ℰ}) && \text{Item~\ref{il:melchior}} \\
                  & ≥ -r · μ^{\max} (\wtilde{ℰ}) && \text{Assumption~\ref{ass:spoE}.} \\
    \intertext{On the other hand,}
    \deg (\img τ) & ≤ r·μ^{\max}(\wtilde{ℰ}) - [\wtilde D]^{\dim X} && \text{Consequence~\ref{cons:xya}.}
  \end{align*}
  We obtain a contradiction to the choice of $m$ in Assumption~\ref{il:4}.  This
  finishes the proof of Claim~\ref{claim:restus}.
\end{proof}

\subsubsection*{Step 6: End of proof}
\approvals{
  Behrouz & yes \\
  Daniel & yes \\
  Stefan & yes \\
  Thomas & yes
}

We aim to show that the Higgs sheaf $(ℰ_D, θ_D) = ι^{[*]}(ℰ, θ)$ is stable with
respect to $H|_D$.  Applying Proposition~\ref{prop:1} to the resolution morphism
$π|_{\wtilde{D}}$, this is equivalent to showing that
$(\wtilde{ℰ_D}, \wtilde{θ_D})$ is stable with respect to
$\wtilde{H}|_{\wtilde D}$.  But since $(\wtilde{ℰ_D}, \wtilde{θ_D})$ and
$(\wtilde{ℰ}_{\wtilde{D}}, \wtilde{θ}_{\wtilde{D}})$, agree outside of the
$π|_{\wtilde{D}}$-exceptional set, Proposition~\ref{prop:comparison} says that
one is stable if and only if the other is.  Stability of
$(\wtilde{ℰ}_{\wtilde{D}}, \wtilde{θ}_{\wtilde{D}})$ was, however, established
in Claim~\ref{claim:restus}.  \qed

\subsection{Higgs bundles and variations of Hodge structures}
\label{subsect:VHS}
\approvals{
  Behrouz & yes \\
  Daniel & yes \\
  Stefan & yes \\
  Thomas & yes
}

In a series of fundamental works, including \cite{MR944577, MR1179076}, Simpson
relates locally free Higgs sheaves on projective manifolds to representations of
the fundamental group, and to variations of Hodge structures.  We will use these
results later to prove our uniformisation result, Theorem~\ref{thm:BQ}.  For the
reader's convenience, we briefly recall the most relevant definitions and
explain how they fit into the framework of minimal model theory.

\begin{defn}[Polarised, complex variation of Hodge structures]\label{def:pCVHS}
  Let $X$ be a complex manifold, and $w ∈ ℕ$ a natural number.  A
  \emph{polarised, complex variation of Hodge structures of weight $w$}, or
  \pCVHS\ in short, is a $\cC^∞$-vector bundle $\cV$ with a direct sum
  decomposition $\cV = ⊕_{r+s=w} \cV^{r,s}$, a flat connection $D$ that
  decomposes as follows
  \begin{equation}\label{eq:Gtrans}
    D|_{\cV^{r,s}} : \cV^{r,s} → \cA^{0,1}(\cV^{r+1,s-1}) ⊕ \cA^{1,0}(\cV^{r,s}) ⊕
    \cA^{0,1}(\cV^{r,s}) ⊕ \cA^{1,0}(\cV^{r-1,s+1}),
  \end{equation}
  and a $D$-parallel Hermitian metric on $\cV$ that makes the direct sum
  decomposition orthogonal and that on $\cV^{r,s}$ is positive definite if $r$
  is even and negative definite if $r$ is odd.
\end{defn}

Given a \pCVHS, one constructs an associated Higgs bundle.  In fact, there are
two equivalent constructions that produce isomorphic results.

\begin{construction}[Higgs sheaves induced by a \pCVHS]\label{cons:hs}
  Given a \pCVHS\ as in Definition~\ref{def:pCVHS}, use \eqref{eq:Gtrans} to
  decompose $D$ as $D = \overline{θ} ⊕ ∂ ⊕ \overline{∂} ⊕ θ$.
  \begin{description}
  \item[First construction] The operators $\overline{∂}$ equip the
    $\cC^∞$-bundles $\cV^{r,s}$ with complex structures.  We write $ℰ^{r,s}$ for
    the associated locally free sheaves of $𝒪_X$-modules, and set
    $ℰ := ⊕ ℰ^{r,s}$.  The operators $θ$ then define an $𝒪_X$-linear morphism
    $ℰ → ℰ ⊗ Ω¹_X$.  As $D$ is flat, this is a Higgs field.

  \item[Second construction] The operators $\overline{∂}+\overline{θ}$ equip the
    $\cC^{∞}$-bundle $\cV$ with a complex structure where the $(1,0)$-part of
    the connection $D$ becomes holomorphic.  We call this holomorphically flat
    bundle $ℋ$.  The complex subbundles $ ℱ^p:= \bigoplus_{r≥ p} \cV^{r,s} $
    are holomorphic and hence give a decreasing filtration of $ℋ$ by
    holomorphic subbundles, cf.~\cite[Thm.~10.3]{Voisin-Hodge1}.
    Condition~\eqref{eq:Gtrans} then translates into $D (ℱ^p) ⊂ ℱ^{p-1} ⊗ Ω¹_X$.
    Hence, $D$ induces an $𝒪_X$-linear morphism $ℰ → ℰ ⊗ Ω¹_X$ on the associated
    graded sheaf $ℰ:= \bigoplus ℱ^p/ℱ^{p+1}$.  As $D$ is flat, this is a Higgs
    bundle.
  \end{description}
\end{construction}

While part of Simpson's work refers to the first construction, we will use the
second construction throughout.  The formulation using filtrations is closer to
standard textbooks on Hodge theory and allows us to quote \cite{Voisin-Hodge1}
or \cite{CMSP} without conflict of notation.

\begin{defn}[Higgs bundles induced by a \pCVHS]\label{def:hscvhs}
  Let $X$ be a complex manifold and $(ℰ, θ)$ a Higgs bundle on $X$.  We say that
  \emph{$(ℰ, θ)$ is induced by a \pCVHS} if there exists a \pCVHS\ on $X$ such
  that $(ℰ, θ)$ is isomorphic to the Higgs bundle obtained from it via the
  second construction in \ref{cons:hs}.
\end{defn}

\begin{rem}
  In the setting of Definition~\ref{def:hscvhs}, the \pCVHS\ $\cV$ is in general
  not uniquely determined by $(ℰ, θ)$.
\end{rem}

\subsubsection{Criteria for a Higgs bundle to be induced by a \pCVHS}
\approvals{
  Behrouz & yes \\
  Daniel & yes \\
  Stefan & yes \\
  Thomas & yes
}

Scaling the Higgs field induces an action of $ℂ^*$ on the set of isomorphism
classes of Higgs bundles.  Under suitable assumptions, Simpson shows that Higgs
bundles induced by a \pCVHS\ correspond exactly to $ℂ^*$-fixed points.  The
following theorem summarises his results.

\begin{thm}[{Higgs bundles induced by a \pCVHS, I, \cite[Cor.~4.2]{MR1179076}}]\label{thm:charPCVHS}
  Let $X$ be a complex, projective manifold of dimension $n$ and $H ∈ \Div(X)$
  be an ample divisor.  Let $(ℰ, θ)$ be a Higgs bundle on $X$.  Then, $(ℰ, θ)$
  comes from a variation of Hodge structures in the sense of
  Definition~\ref{def:hscvhs} if and only if the following three conditions
  hold.
  \begin{enumerate}
  \item\label{il:emerson} The Higgs bundle $(ℰ, θ)$ is $H$-polystable.
  \item\label{il:lake} The intersection numbers $ch_1(ℰ)·[H]^{n-1}$ and
    $ch_2(ℰ)·[H]^{n-2}$ both vanish.
  \item\label{il:palmer} For any $t ∈ ℂ^*$, the Higgs bundles $(ℰ, θ)$ and
    $(ℰ, t·θ)$ are isomorphic.  \qed
  \end{enumerate}
\end{thm}

\begin{rem}\label{rem:charPCVHS}
  With $X$ and $H$ as in Theorem~\ref{thm:charPCVHS}, any Higgs bundle $(ℰ, θ)$
  that satisfies \ref{il:emerson} and \ref{il:lake} carries a flat
  $\cC^∞$-connection, \cite[Thm.~1(2) and Cor.~1.3]{MR1179076}.  In particular,
  all its Chern classes vanish.
\end{rem}

As one immediate consequence of Theorem~\ref{thm:charPCVHS}, we obtain the
following strengthening of \cite[Cor.~4.3]{MR1179076}.

\begin{cor}[Higgs bundles induced by a \pCVHS, II]\label{cor:pfccs}
  Let $X$ be a projective manifold, and $H ∈ Div(X)$ an ample divisor.  Let
  $\imath: S \hookrightarrow X$ be a submanifold.  The push-forward map
  $\imath_*: π_1(S) → π_1(X)$ induces a restriction map
  $$
  \begin{array}{rccc}
    r : & \begin{Bmatrix}
      \text{Isomorphism classes of $H$-semi-}\\
      \text{stable Higgs bundles $(ℰ, θ)$ on $X$} \\
      \text{with vanishing Chern classes.}
    \end{Bmatrix}
    & → &
    \begin{Bmatrix}
      \text{Isomorphism classes of $H$-semi-}\\
      \text{stable Higgs bundles $(ℰ, θ)$ on $S$} \\
      \text{with vanishing Chern classes.}
    \end{Bmatrix} \\[0.6cm]
    & (ℰ, θ) & ↦ & (ℰ, θ)|_S.
  \end{array}
  $$
  In particular, if $(ℰ, θ)$ is any $H$-semistable Higgs bundle $(ℰ, θ)$ on $X$
  with vanishing Chern classes, then $(ℰ, θ)|_S$ is again $H$-semistable.  The
  map $r$ has the following properties.
  \begin{enumerate}
  \item\label{il:simon} If $\imath_*$ is surjective, then $r$ is injective.  In
    particular, if $(ℰ, θ)$ is a Higgs bundle on $X$ such that $(ℰ, θ)|_S$ comes
    from a \pCVHS, then $(ℰ, θ)$ comes from a \pCVHS.
  \item\label{il:garfunkel} If in addition the induced push-forward map
    $\what{\imath}_*:\what{π}_1(S) → \what{π}_1(X)$ of algebraic fundamental
    groups is isomorphic, then $r$ is surjective.
  \end{enumerate}
\end{cor}
\begin{proof}
  Simpson's Nonabelian Hodge Correspondence, \cite[Cor.~3.10]{MR1179076} or
  \cite[Thm.~1]{MR1159261}, gives an equivalence between the categories of
  representations of the fundamental group $π_1(X)$ (resp.\ $π_1(S)$) and
  $H$-semistable Higgs bundles on $X$ (resp.\ $S$) with vanishing Chern classes.
  The correspondence is functorial in morphisms between manifolds, and pull-back
  of Higgs bundles corresponds to the push-forward of fundamental groups,
  \cite[Rem.~1 on p.~36]{MR1179076}.  In particular, we see that the restriction
  of an $H$-semistable Higgs bundle with vanishing Chern classes is again
  $H$-semistable.

  In the setting of \ref{il:simon} where the push-forward map $π_1(S) → π_1(X)$
  is surjective, this immediately implies that the restriction $r$ is injective.
  The restriction map $r$ is clearly equivariant with respect to the actions of
  $ℂ^*$ obtained by scaling the Higgs fields.  Injectivity therefore implies
  that the isomorphism class of a Higgs bundle $(ℰ, θ)$ is $ℂ^*$-fixed if and
  only if the same is true for $(ℰ, θ)|_S$.  Theorem~\ref{thm:charPCVHS} thus
  proves the second clause of \ref{il:simon}.

  Now assume that we are in the setting of \ref{il:garfunkel}, where in addition
  the push-forward map $\what{π}_1(S) → \what{π}_1(X)$ is assumed to be
  isomorphic.  Since fundamental groups of algebraic varieties are finitely
  generated, this implies via Malcev's theorem that every representation of
  $π_1(S)$ comes from a representation of $π_1(X)$, \cite[Thm.~1.2b]{MR0262386}
  or see \cite[Sect.~8.1]{GKP13} for a detailed pedestrian proof.  The claim
  thus again follows from Simpson's Nonabelian Hodge Correspondence.
\end{proof}

\subsubsection{The period map}
\approvals{
  Behrouz & yes \\
  Daniel & yes \\
  Stefan & yes \\
  Thomas & yes
}

A \pCVHS\ on a simply connected complex manifold $X$ induces a map to the period
domain.  Here, we will show that Higgs bundles that are induced by a \pCVHS\
come from the period domain.  If $X$ is the desingularisation of a klt variety,
this implies that the relevant bundle comes from the singular space.

\begin{construction}[\protect{Period map, cf.~\cite[Sect.~10.1.2--3]{Voisin-Hodge1} or \cite[Sect.~4.3]{CMSP}}]\label{constr:periodmap}
  Given a \pCVHS\ on a simply connected complex manifold $X$, we obtain a period
  map $ρ: X → \cD$ into the classifying space $\cD$ for Hodge structures of the
  given type, the so-called \emph{period domain}.  Let us quickly recall the
  construction.  Let $F$ be the flag manifold parametrising complex flags of the
  type given by the filtration $ℱ^•$.  The projective manifold $F$ embeds into
  the product $P$ of Grassmannians that parametrise subspaces of those
  dimensions that occur in the filtration $ℱ^•$.  As $X$ is simply connected,
  the holomorphically flat bundle $ℋ$ trivialises, and so the filtration $ℱ^•$
  yields a family of flags in a fixed complex vector space parametrised by $X$.
  Assigning to each point in $X$ the corresponding point in $P$ yields the
  period map $ρ: X → F \hookrightarrow P$, which is actually holomorphic,
  cf.~\cite[Thm.~10.9]{Voisin-Hodge1}.  The image of $ρ$ can be seen to lie in a
  special domain $\cD$ inside the closed complex submanifold $\check{\cD}$ of
  $F$ that is defined by the orthogonality condition required in
  Definition~\ref{def:pCVHS}, the \emph{period domain}.
\end{construction}

\begin{prop}\label{prop:tautologicalpullback}
  Let $X$ be a simply connected manifold and $(ℰ, θ)$ be a Higgs bundle on $X$
  that comes from a \pCVHS.  Let $ρ : X → \cD$ be the associated period map.
  Then, there exists a holomorphic vector bundle $ℰ_{\cD}$ on $\cD$ such that
  $ℰ ≅ ρ^* ℰ_{\cD}$.
\end{prop}
\begin{proof}
  It follows from Construction~\ref{constr:periodmap} that $\cD$ is an open
  subset in a flag manifold (whose type is determined by the filtration $ℱ^•$),
  which in turn can be embedded into a product of Grassmannians.  Each of the
  Grassmannians carries a tautological vector bundle, which can be restricted to
  $\cD$, yielding a holomorphic vector bundle $𝒯^p$ on $\cD$.  By definition and
  holomorphy of the period map, we have $ρ^*(𝒯^p) ≅ ℱ^p$,
  cf.~\cite[p.~250]{Voisin-Hodge1}.  It follows that $ℱ^p/ℱ^{p+1}$ is a pullback
  from the period domain, and hence so is $ℰ = \bigoplus ℱ^p/ℱ^{p+1}$.
\end{proof}

\begin{cor}\label{cor:higgsfromdownst}
  Let $(X,D)$ be a klt pair and $π : \wtilde X → X$ a resolution of
  singularities.  Let $(ℰ, θ)$ be a Higgs bundle on $\wtilde X$ that is induced
  by a \pCVHS.  Then, $ℰ$ comes from $X$.  More precisely, there exists a
  locally free sheaf $ℰ_X$ on $X$ such that $ℰ = π^* ℰ_X$.  Necessarily, we then
  have $ℰ_X ≅ π_* (ℰ)^{**}$.
\end{cor}
\begin{proof}
  It suffices to construct $ℰ_X$ locally in the analytic topology, near any
  given point of $X$.  Now, given any $x ∈ X$, recall from
  \cite[p.~827]{Takayama2003} that there exists a contractible, open
  neighbourhood $U = U(x) ⊆ X^{an}$ whose preimage $\wtilde U := π^{-1}(U)$ is
  simply connected.  By assumption, $(ℰ, θ)$ is induced from a \pCVHS \space
  $\cV$.  Let $ρ : \wtilde U → \cD$ be the corresponding period map.
  
  We claim that $ρ$ factors through the resolution $π: \wtilde U → U$.  Indeed,
  since the fibres of $π$ are rationally chain-connected by, it suffices to show
  that given any morphism $η: ℙ¹ → \wtilde U$, the composed map
  $ρ ◦ η : ℙ¹ → \cD$ is constant.  Pulling back $\cV$ via $η$ yields a
  \pCVHS\space on $ℙ¹$ whose associated period map equals $ρ◦η$.  However, due
  to hyperbolicity properties of the period domain $\cD$, this map has to be
  constant, \cite[Application 13.4.3]{CMSP}.
  
  By Proposition~\ref{prop:tautologicalpullback}, we know that
  $ℰ ≅ ρ^*(ℰ_{\cD})$ for some vector bundle $ℰ_{\cD}$ on the period domain
  $\cD$.  If $ρ_U: U → \cD$ is the holomorphic map whose existence was shown in
  the previous paragraph, the vector bundle $ℰ_U := ρ_U^*(ℰ_{\cD})$ hence
  fulfils $π^*(ℰ_U ) ≅ ℰ$, as desired.
\end{proof}

\begin{rem}
  Corollary~\ref{cor:higgsfromdownst} is actually true in a much more general
  setting.  In fact, the bundle $ℰ$ is trivial on the fibres of $π$.  Then,
  regardless whether $ℰ$ carries a Higgs structure or not, $ℰ$ is the pull-back
  of a bundle on $X$, as $X$ has only klt singularities.  As the proof is much
  more involved than the one presented in the previous paragraphs, with our main
  application in mind we have decided to restrict to the case of Higgs bundles
  coming from \pCVHS s here.  Details for the general case will appear in a
  forthcoming paper.
\end{rem}

\part{Miyaoka-Yau Inequality and Uniformisation}
%
%
\svnid{$Id: 06-Bogomolov.tex 840 2018-11-20 14:25:41Z kebekus $}

\section{The $ℚ$-Bogomolov-Gieseker inequality}
\label{sect:bogo}
\subversionInfo
\approvals{
  Behrouz & yes \\
  Daniel & yes \\
  Stefan & yes \\
  Thomas & yes
}

We establish the $ℚ$-Bogomolov-Gieseker inequality for Higgs sheaves on klt
spaces.  Section~\ref{sect:MY} applies this result to the natural Higgs sheaf of
Example~\ref{ex:BQfield1}, in order to establish the $ℚ$-Miyaoka-Yau inequality
for the tangent sheaf of a klt variety of general type whose canonical divisor
is nef.

\begin{thm}[$ℚ$-Bogomolov-Gieseker inequality]\label{thm:BogIneq}
  Let $(X,D)$ be a projective, klt pair of dimension $n ≥ 2$, and let $P$ be a
  nef $ℚ$-Cartier $ℚ$-divisor on $X$.  If $(ℰ, θ)$ is any reflexive Higgs sheaf
  of $\rank ℰ ≥ 2$ on $X$ that is stable with respect to $P$, then $ℰ$ verifies
  \begin{equation}\label{eq:fasax}
  \what{Δ}(ℰ)·[P]^{n-2} ≥ 0.
  \end{equation}
  We refer to \eqref{eq:fasax} as the \emph{$ℚ$-Bogomolov-Gieseker inequality}.
\end{thm}

We expect that Theorem~\ref{thm:BogIneq} will also hold for semistable sheaves.
Again, with our main application in mind, we restrict ourselves to the stable
case.

\subsection{Preparations for the proof of Theorem~\ref*{thm:BogIneq}}
\approvals{
  Behrouz & yes \\
  Daniel & yes \\
  Stefan & yes \\
  Thomas & yes
}

Cutting by hyperplanes, the proof of the $ℚ$-Bogomolov-Gieseker inequality will
quickly reduce to the surface case, which is handled first.

\begin{prop}[$ℚ$-Bogomolov-Gieseker Inequality on klt surfaces]\label{prop:Chernclassflatnesss}
  Let $(X,D)$ be a projective, klt pair of dimension two, and let $H$ be a nef
  $ℚ$-Cartier $ℚ$-divisor on $X$.  If $(ℰ,θ)$ is any reflexive Higgs sheaf of
  $\rank ℰ ≥ 2$ on $X$ that is stable with respect to $H$, then the sheaf $ℰ$
  satisfies the $ℚ$-Bogomolov-Gieseker inequality $\what{Δ}(ℰ) ≥ 0$.
\end{prop}
\begin{proof}
  Openness of stability, Proposition~\ref{prop:G-opennessx}, allows us to assume
  without loss of generality that $H$ is integral, Cartier, and ample.
  Theorem~\ref{thm:41} gives both a $ℚ$-variety structure $X_ℚ$ on $X$, and a
  global, Cohen-Macaulay and Galois cover $γ: \what{X} → X$ that allows us to
  compute $ℚ$-Chern classes on $X$ in terms of honest Chern classes of pull-back
  sheaves.  Let $G := \Gal(\what{X}/X)$ be the corresponding Galois group and
  set $\what H := γ^*H$.

  Applying Construction~\ref{cons:rpb2} to $(ℰ, θ)$, we obtain a reflexive Higgs
  $ℚ$-sheaf $(ℰ, θ)^{[ℚ]}$ on $X_ℚ$.  Let $\bigl(\what{ℰ}, \what{θ}\bigr)$ be
  the induced Higgs $G$-sheaf on $\what X$, as discussed in
  Fact~\ref{fact:isHgc}.  Since $\what X$ is of dimension two,
  Fact~\ref{fact:isHgc} asserts that $\bigl(\what{ℰ}, \what{θ} \bigr)$ is
  actually a Higgs $G$-bundle.  Finally, let $π: \wtilde{X} → \what{X}$ be a
  strong, $G$-invariant resolution of $\what X$.  The following diagram
  summarises the situation:
  \begin{equation}\label{eq:uiif}
    \xymatrix{ %
      \wtilde{X} \ar[rrr]^{π}_{\text{resolution of sings.}} \ar@/^5mm/[rrrrrr]^{ψ} &&& \what{X}
      \ar[rrr]^(.45){γ}_{\text{Galois, with group $G$}} &&& X.  }
  \end{equation}
  We obtain two locally free Higgs $G$-sheaves on $\wtilde{X}$, namely
  $π^{[*]} \bigl(\what{ℰ}, \what{θ} \bigr)$ and $ψ^{[*]}(ℰ, θ)$---we refer to
  Section~\ref{sec:rpull-back} for the construction of the reflexive pull-back
  $ψ^{[*]}$ and to Lemma~\ref{lem:rpbfgc} and Notation~\ref{not:rpbfgc} for all
  matters concerning $π^{[*]}$.  These Higgs sheaves are not necessarily equal,
  but they do agree over the big open set of $X_{\reg}$ where $ℰ$ is locally
  free.  By reflexivity, the two Higgs sheaves will then coincide outside the
  exceptional set of $π$.

  It follows from Proposition~\ref{prop:stability_upstairs_downstairs} that
  $ψ^{[*]}(ℰ, θ)$ is $G$-stable with respect to $π^*\bigl(\what{H}\bigr)$.
  Since both sheaves agree outside the $π$-exceptional set,
  Proposition~\ref{prop:comparison} implies that the $G$-Higgs bundle
  $π^* \bigl(\what{ℰ}, \what{θ} \bigr)$ is $G$-stable with respect to the nef
  polarisation $π^*(\what{H})$ as well.  Openness of $G$-stability,
  Proposition~\ref{prop:G-opennessx}, allows us to modify
  $π^*\bigl(\what{H}\bigr)$, and find a $G$-stable, ample divisor $\wtilde{A}$
  on $\wtilde{X}$, such that $π^* \bigl(\what{ℰ}, \what{θ} \bigr)$ is $G$-stable
  with respect to $\wtilde{H}$.

  Since $\what{ℰ}$ is locally free, we can discuss the standard Bogomolov
  discriminant $Δ\bigl(\what{ℰ}\bigr)$, as introduced in
  Notation~\ref{not:bogomolov}.  The functorial properties of Chern classes,
  \cite[Thm.~3.2(d)]{Fulton98}, and the choice of $γ$ imply
  \begin{equation}\label{eq:discriminant_upstairs_downstairs}
    Δ\bigl(π^*\what{ℰ} \bigr)= Δ\bigl(\what{ℰ}\bigr) = (\deg γ) · \what{Δ}(ℰ).
  \end{equation}
  Simpson's Bogomolov-Gieseker Inequality for $G$-Higgs bundles that are stable
  with respect to an ample polarisation, \cite[Thm.~1 and Prop.~3.4]{MR944577},
  applies to $π^* \bigl(\what{ℰ}, \what{θ} \bigr)$ and $\wtilde A$, showing that
  $Δ \bigl(π^* \what{ℰ} \bigr) ≥ 0$.  Together with
  \eqref{eq:discriminant_upstairs_downstairs}, this finishes the proof of
  Proposition~\ref{prop:Chernclassflatnesss}.
\end{proof}

\subsection{Proof of Theorem~\ref*{thm:BogIneq}}
\CounterStep
\approvals{
  Behrouz & yes \\
  Daniel & yes \\
  Stefan & yes \\
  Thomas & yes
}

By multilinearity of the form $\what{Δ}$, it suffices to prove the claim in the
case where $P$ is an integral Cartier divisor.  Using that the function
$$
N¹(X)_{ℝ} → ℝ, \qquad α ↦ \what{Δ}(ℰ)·α^{n-2}
$$
is continuous, Proposition~\ref{prop:G-opennessx} allows us to assume without
loss of generality that $P$ is integral, Cartier, and ample.  Choosing $m ≫ 0$
sufficiently large, the Restriction theorem for stable Higgs sheaves,
Theorem~\ref{thm:restriction}, allows us to find a tuple of hyperplanes
$(H_1, …, H_{n-2}) ∈ |m·P|^{⨯ (n-2)}$ with associated complete intersection
surface $S := H_1 ∩ ⋯ ∩ H_{n-2}$ such that the following holds.
\begin{enumerate}
\item The scheme $S$ is a normal and irreducible surface, and not contained in
  the support of $D$.  The pair $(S,D|_S)$ is klt, \cite[Lem.~5.17]{KM98}.
  
\item The restriction $ℰ|_S$ is reflexive, \cite[Thm.~12.2.1]{EGA4-3}.
  
\item Denoting the inclusion by $ι: S → X$, the Higgs sheaf $ι^{*}(ℰ, θ)$ is
  stable with respect to $P|_S$, Theorem~\ref{thm:restriction}.
  
\item We have an equality $\what{Δ}(ℰ)·[P]^{n-2} = m^{n-2} · \what{Δ}(ℰ|_S)$,
  Item~\ref{il:p3} of Theorem~\ref{thm:41}.
\end{enumerate}
The result hence follows from Proposition~\ref{prop:Chernclassflatnesss} above.
\qed

%
%
\svnid{$Id: 07-MiyaokaYau.tex 840 2018-11-20 14:25:41Z kebekus $}

\section{The $ℚ$-Miyaoka-Yau inequality}
\label{sect:MY}
\subversionInfo

\subsection{Proof of Theorem~\ref*{thm:MYinequality}}
\label{ssec:pomyi1}
\approvals{
  Behrouz & yes \\
  Daniel & yes \\
  Stefan & yes \\
  Thomas & yes
}

Theorem~\ref{thm:MYinequality} will follow from the results of
Section~\ref{sect:bogo}, once we can apply them to the natural Higgs sheaf
$(ℰ_X,θ_X)$ of Example~\ref{ex:BQfield1}, where $ℰ_X = Ω_X^{[1]} ⊕ 𝒪_X$.  We
hence establish stability of $(ℰ_X,θ_X)$ first.  This is a consequence of the
following minor generalisation of a recent semistability result of Guenancia,
\cite[Thm.~A]{Guenancia}, which in turn generalises a classical result of Enoki,
\cite[Cor.~1.2]{Eno87}.

\begin{thm}[Semistability of tangent sheaves]\label{thm:Guenancia}
  Let $X$ be a projective, klt variety of general type whose canonical divisor
  $K_X$ is nef.  Then, $𝒯_X$ and $Ω^{[1]}_X$ are semistable with respect to
  $K_X$.
\end{thm}
\begin{proof}
  It suffices to show semistability for $𝒯_X$.  Recall from
  Reminder~\ref{remi:cm} that $K_X$ is semiample and induces a birational
  morphism $φ: X → Z$, where $Z$ is klt, and $K_Z$ is ample.  By
  \cite[Thm.~A]{Guenancia}, the tangent sheaf $𝒯_{Z}$ is semistable with respect
  to $K_{Z}$.  Since $𝒯_X$ coincides with $φ^{[*]}(𝒯_{Z})$ outside of the
  $φ$-exceptional set, $𝒯_X$ is hence semistable with respect to
  $K_X = φ^*(K_{Z})$, cf.\ Lemma~\ref{lem:elemSlp2}.  This concludes the proof.
\end{proof}

\begin{cor}[Higgs-stability for varieties of general type]\label{cor:stable-claim}
  Let $X$ be a projective, klt variety of general type whose canonical divisor
  $K_X$ is nef.  Then, the natural Higgs sheaf $(ℰ_X, θ_X)$ of
  Example~\ref{ex:BQfield1} is stable with respect to $K_X$.
\end{cor}
\begin{proof}
  Write $n := \dim X$ and $d := [K_X]^n ∈ ℚ^+$, which is positive since $K_X$ is
  nef and big.  Aiming for a contradiction, assume that $(ℰ_X,θ_X)$ is not
  stable with respect to $K_X$.  Hence, there exists a subsheaf
  $0 \ne ℱ ⊊ ℰ_X$ that is generically $θ$-invariant and satisfies
  \begin{equation}\label{ineq:slope-F}
    μ_{K_X}(ℱ)≥ μ_{K_X}(ℰ_X) = d/(n+1).
  \end{equation}
  Lemma~\ref{lem:x11} allows us to assume that $ℱ$ is saturated in $ℰ_X$.  In
  particular, $ℱ$ is reflexive.  Write $r := \rank ℱ$ and note that $r < n+1$.
  
  Let $α: ℱ → 𝒪_X$ be the morphism induced by the projection to the
  $𝒪_X$-summand of $ℰ$.  Recalling from Example~\ref{ex:BQfield1} that no
  subsheaf of the direct summand $Ω^{[1]}_X$ is ever generically
  $θ_X$-invariant, it follows that $α$ is not the zero map.  We also notice that
  $α$ is \emph{not} an injection, for otherwise $ℱ$ is the Weil-divisorial sheaf
  of an anti-effective Weil divisor, and $[ℱ]·[K_X]^{n-1} ≤ 0$, contradicting
  Inequality~\eqref{ineq:slope-F}.  It follows that $r > 1$ and
  $\rank( \ker α) = r - 1 > 0$.  More can be said.  Since $\det(\img α)$ is Weil
  divisorial for an anti-effective divisor, we have
  $$
  [\ker α]·[K_X]^{n-1} = [ℱ]·[K_X]^{n-1} - [\img α]·[K_X]^{n-1} ≥ [ℱ]·[K_X]^{n-1}
  $$
  and, dividing by $r-1$,
  \begin{align*}
    μ_{K_X} (\ker α) & ≥ \frac{[ℱ]·[K_X]^{n-1}}{r-1} = μ_{K_X} (ℱ) · \frac{r}{r-1} \\
                     & ≥ \frac{d}{n+1} · \frac{r}{r-1} && \text{by \eqref{ineq:slope-F}} \\
                     & = \frac{d}{n} · \frac{nr}{(n+1)(r-1)} > \frac{d}{n} && \text{since $n+1>r$ and $d > 0$.}
  \end{align*}
  It follows that $μ_{K_X} (\ker α) > μ_{K_X}\bigl( Ω^{[1]}_X \bigr)$, the
  latter one being equal to $d/n$.  Since $\ker α$ injects into $Ω^{[1]}_X$ by
  definition of $α$, we hence obtain a contradiction to the semistability of
  $Ω^{[1]}_X$ proven in Theorem~\ref{thm:Guenancia}.
\end{proof}

\subsection{Proof of Theorem~\ref*{thm:MYinequality}}
\approvals{ Behrouz & yes \\
  Daniel & yes \\
  Stefan & yes \\
  Thomas & yes}

\Preprint{Using the elementary calculus of
  Lemma~\ref{lem:calculus}}\Publication{By elementary calculus of Chern
  classes}, Inequality~\eqref{eq:X2} is equivalent to
$$
\what{Δ} \bigl( 𝒯_X ⊕ 𝒪_X \bigr)·[K_X]^{n-2} = \what{Δ} \bigl( Ω^{[1]}_X ⊕
𝒪_X \bigr)·[K_X]^{n-2} ≥ 0,
$$
which follows from Theorem~\ref{thm:BogIneq} and
Corollary~\ref{cor:stable-claim}.  \qed

%
%
\svnid{$Id: 08-uniformisation.tex 840 2018-11-20 14:25:41Z kebekus $}

\section{Uniformisation}
\label{sect:uniformisation}
\subversionInfo

\subsection{Proof of Theorem~\ref*{thm:BQ}}
\approvals{
  Behrouz & yes \\
  Daniel & yes \\
  Stefan & yes \\
  Thomas & yes
}

Theorem \ref{thm:BQ} follows directly from the subsequent, more general result.

\begin{thm}\label{thm:more_general_uniformisation}
  Let $X$ be an $n$-dimensional, projective, klt variety of general type whose
  canonical divisor $K_X$ is nef.  Assume that $X$ is smooth in codimension two.
  Recall from Reminder~\ref{remi:cm} that $K_X$ is semiample, and induces a
  morphism $\varphi: X → Z$, where $Z$ is klt, and $K_Z$ is ample.  If equality
  holds in the $ℚ$-Miyaoka-Yau inequality \eqref{eq:X2}, then $Z$ is smooth in
  codimension two, there exists a ball quotient $Y$ and a finite, Galois,
  quasi-étale morphism $f: Y → Z$.  In particular, $Z$ has only quotient
  singularities.
\end{thm}

\begin{proof}[Proof of Theorem~\ref{thm:BQ}]
  As varieties with terminal singularities are smooth in codimension two, the
  result follows by applying Theorem~\ref{thm:more_general_uniformisation}.
\end{proof}

\subsection{Preparation for the proof of Theorem~\ref*{thm:more_general_uniformisation}}
\approvals{
  Behrouz & yes \\
  Daniel & yes \\
  Stefan & yes \\
  Thomas & yes
}

The proof of Theorem~\ref{thm:more_general_uniformisation} is based on the
following two propositions.

\begin{prop}\label{prop:c1}
  Let $X$ be a projective, klt variety of general type whose canonical divisor
  is nef.  Suppose that $X$ is smooth in codimension two and that equality holds
  in the $ℚ$-Miyaoka-Yau inequality \eqref{eq:X2}.  Recall from
  Reminder~\ref{remi:cm} that $K_X$ is semiample and induces a morphism
  $\varphi: X → Z$, where $Z$ is klt, and $K_Z$ is ample.  Then, $Z$ is smooth
  in codimension two, and equality holds in the $ℚ$-Miyaoka-Yau inequality for
  $Z$.
\end{prop}
\begin{proof}
  Choose a strong resolution of singularities, say $π: \wtilde X → X$, and
  observe that the composed map $\varphi ◦ π: \wtilde X → Z$ is a resolution of
  $Z$ that is minimal in codimension two.  Let $S_{Z}$ be a surface cut out by
  general sections of $|m·K_{Z}|$, for $m ≫ 0$, and let $S_X$, $S_{\wtilde X}$
  denote the strict transforms in $X$ and $\wtilde X$, respectively.  Since $X$
  is smooth in codimension two, $S_X$ is entirely contained in the smooth locus
  $X_{\reg}$, and $π$ is therefore isomorphic near $S_X$.  We obtain:
  \begin{align}
    \what{c}_2(𝒯_{Z}) · [K_{Z}]^{n-2} & ≤ c_2(𝒯_{\wtilde X}) · [(\varphi ◦ π)^* K_{Z}]^{n-2} && \text{by \cite[Prop.~1.1]{SBW94}} \label{il:es2} \\
                                        & = c_2(𝒯_{\wtilde X}|_{S_{\wtilde X}}) = c_2(𝒯_X|_{S_X}) \nonumber \\
                                        & = c_2(𝒯_X) · [K_X]^{n-2}, \nonumber \\
    \intertext{with strict inequality if and only if $Z$ does have singularities in
    codimension two, \cite[Prop.~1.1]{SBW94}.  In a similar vein,}
    \what{c}_1(𝒯_{Z}) · [K_{Z}]^{n-1} & = c_1(𝒯_X) · [K_X]^{n-1}.  \label{il:es1}
  \end{align}
  The $ℚ$-Miyaoka-Yau inequality for $Z$ thus forces equality in \eqref{il:es2}.
  This shows both that $Z$ is smooth in codimension two, and that equality holds
  in the $ℚ$-Miyaoka-Yau inequality for $Z$.
\end{proof}

\begin{prop}\label{prop:c2}
  Let $X$ be a projective, klt variety of dimension $n$ that is smooth in
  codimension two and such that the étale fundamental group of $X$ and of its
  smooth locus agree, $\what{π}_1(X_{\reg}) ≅ \what{π}_1(X)$.  If $K_X$ is
  ample and if equality holds in the $ℚ$-Miyaoka-Yau Inequality~\eqref{eq:X2},
  then $X$ is smooth.
\end{prop}

\begin{rem}
  The main reason for the assumption on the codimension of the singular set is
  to guarantee smoothness of complete intersection surfaces and hence their
  isomorphic lifting to a strong resolution of singularities, where we are then
  able to use functoriality properties of Simpson's Nonabelian Hodge
  Correspondence; for details, see the subsequent proof.
\end{rem}

\subsubsection*{Proof of Proposition~\ref*{prop:c2}}
\label{subsubsect:proof_of_main_theorem}

For the reader's convenience, the proof is subdivided into a number of
relatively independent steps.

\subsubsection*{Step 1.  Setup}
\approvals{
  Behrouz & yes \\
  Daniel & yes \\
  Stefan & yes \\
  Thomas & yes
}

The main object of study in our proof is the canonical Higgs sheaf $(ℰ_X,θ_X)$
on $X$, introduced in Example~\ref{ex:BQfield1}.  Recall that
$ℰ_X = Ω^{[1]}_X ⊕ 𝒪_X$ and that $(ℰ_X,θ_X)$ is $K_X$-stable owing to
Corollary~\ref{cor:stable-claim}.  Choose a strong log resolution of
singularities, $π : \wtilde X → X$, such that there exists a $π$-ample Cartier
divisor supported on the exceptional locus of $π$.

\begin{claim}\label{claim:flatbounded}
  Write $r := (n+1)²$.  Let ${\sf B}_r$ denote the set of locally free sheaves
  $ℱ$ on $X$ that have rank $r$, satisfy
  $μ^{\max}_{K_X}(ℱ) = μ^{\max}_{K_X}(\sEnd ℰ_X)$, and have Chern classes
  $c_i\bigl(π^* ℱ\bigr) = 0$ for all $0 < i ≤ r$.  Then, ${\sf B}_r$ is bounded.
\end{claim}
\begin{proof}[Proof of Claim~\ref{claim:flatbounded}]
  Since $X$ has rational singularities, the Euler characteristics $χ_X(𝒢)$ and
  $χ_{\wtilde X}(π^* 𝒢)$ agree for all locally free sheaves $𝒢$ on $X$.  The
  assumption on Chern classes thus guarantees that the Hilbert polynomials of
  the members $ℱ ∈ {\sf B}_r$ are constant, cf.\ \cite[Cor.~15.2.1]{Fulton98}.
  Boundedness thus follows from \cite[Thm.~3.3.7]{MR2665168}.  This ends the
  proof of Claim~\ref{claim:flatbounded}.
\end{proof}

Next, take general divisors in the linear system $|m·K_X|$, for $m$ sufficiently
large, and cut down to a surface.  To be precise, observe the following.

Choosing a sufficiently increasing and divisible sequence of numbers
$0 ≪ m_1 ≪ ⋯ ≪ m_{n-2}$ and a general tuple of elements
$(H_1, …, H_{n-2}) ∈ \prod_i |m_i·K_X|$ the following will hold when we set
$S := H_1 ∩ ⋯ ∩ H_{n-2}$.

\begin{enumerate}
\item\label{il:whitehorse} The intersection $S$ is a smooth surface, and
  entirely contained in $X_{\reg}$; this is because $X$ is smooth in codimension
  two by assumption.

\item\label{il:redhorse} The restriction $(ℰ_X,θ_X)|_S$ is stable with respect
  to $K_X|_S$, cf.\ the Restriction Theorem~\ref{thm:restriction}.

\item\label{il:blackhorse} The natural morphism $ι_*: π_1(S) → π_1(X_{\reg})$,
  induced by the inclusion $ι: S \into X_{\reg}$, is isomorphic, cf.\
  Goresky-MacPherson's Lefschetz-theorem \cite[Thm.~in
  Sect.~II.1.2]{GoreskyMacPherson}.

\item\label{il:dappledhorse} Let $ℱ ∈ {\sf B}_r$.  Then, $ℱ$ is isomorphic to
  $\sEnd ℰ_X$ if and only if the restrictions $ℱ|_S$ and $(\sEnd ℰ_X)|_S$ are
  isomorphic, cf.\ the Bertini-type theorem for isomorphism classes in bounded
  families \cite[Cor.~5.3]{GKP13}.
\end{enumerate}

\begin{rem}\label{rem:blackhorse}
  The natural morphism $π_1(X_{\reg}) → π_1(X)$ is surjective, \cite[0.7.B on
  p.~33]{FL81}, and induces an isomorphism of profinite completions by
  assumption.  Composed with the inclusion $S \hookrightarrow X_{\reg}$, it
  follows from \ref{il:blackhorse} that the morphism $π_1(S) → π_1(X)$ is
  surjective and induces an isomorphism of profinite completions.
\end{rem}

\subsubsection*{Step 2.  The endomorphism bundle}
\approvals{
  Behrouz & yes \\
  Daniel & yes \\
  Stefan & yes \\
  Thomas & yes
}

Since $S$ is entirely contained in the smooth locus of $X$, the restricted Higgs
sheaf $(ℰ_X,θ_X)|_S$ is actually a Higgs bundle, and
Construction~\ref{cons:nHGSdual} allows us to equip the corresponding
endomorphism bundle with a Higgs field.  For brevity of notation, write
$(ℱ_S, Θ_S) := \sEnd \bigl( (ℰ_X, θ_X)|_S \bigr)$.  The rank of $ℱ_S$ equals
$r = (n+1)²$.
  
\begin{claim}\label{claim:HVHS}
  The Higgs bundle $(ℱ_S, Θ_S)$ is induced by a \pCVHS, in the sense of
  Definition~\ref{def:hscvhs}.
\end{claim}
\begin{proof}[Proof of Claim~\ref{claim:HVHS}]
  We need to check the properties listed in Theorem~\ref{thm:charPCVHS}.

  \smallskip

  Item~\ref{il:emerson}: polystability with respect to $K_X|_S$.  By
  Theorem~\ref{thm:restriction}, we know that both $(ℰ_X,θ_X)|_S$ and its dual
  are $K_X|_S$-stable Higgs bundles on the smooth surface $S$.  In particular,
  it follows from~\cite[Thm.~1(2)]{MR1179076} that both bundles carry a
  Hermitian-Yang-Mills metric with respect to $K_X|_S$, and thus so does
  $(ℱ_S, Θ_S)$.  Hence it follows from~\cite[Thm.~1]{MR1179076} that
  $(ℱ_S, Θ_S)$ is polystable with respect to $K_X|_S$.

  \smallskip

  Item~\ref{il:lake}: vanishing of Chern classes.  As the endomorphism bundle of
  the locally free sheaf $ℰ_X|_S$, the first Chern class of $ℱ_S$ clearly
  vanishes.  Vanishing of $c_2(ℱ_S)$ is then an immediate consequence of the
  assumed equality in \eqref{eq:X2}.  Together with polystability, this implies
  that $ℱ_S$ is flat, \cite[Thm.~1]{MR1179076}, and hence all its Chern classes
  vanish.

  \smallskip

  Item~\ref{il:palmer}: the Higgs bundle $(ℰ_X, θ_X)|_S$ has the structure of a
  system of Hodge bundles, \cite[Sect.~4]{MR1179076}.  Its isomorphism class is
  therefore fixed under the action of $ℂ^*$, \cite[p.~45]{MR1179076}.  Observing
  that the same holds for its dual and its endomorphism bundle, this ends the
  proof of Claim~\ref{claim:HVHS}.
\end{proof}

\subsubsection*{Step 3.  End of proof}
\approvals{
  Behrouz & yes \\
  Daniel & yes \\
  Stefan & yes \\
  Thomas & yes
}

Since $S$ is entirely contained in the smooth locus of $X$, it is canonically
isomorphic to its preimage $\wtilde{S} := π^{-1}(S)$ in the resolution
$\wtilde X$.  Let $(ℱ_{\wtilde{S}}, Θ_{\wtilde{S}})$ be the Higgs bundle on
$\wtilde{S}$ that corresponds to $(ℱ_S, Θ_S)$ under this isomorphism.

There exists a $ℚ$-divisor $E ∈ ℚ\Div(\wtilde X)$, supported entirely on the
$π$-exceptional locus, such that $\wtilde H := π^*(K_X)+E$ is ample.  Since
$\wtilde S$ and $\supp E$ are disjoint, the Higgs bundle
$(ℱ_{\wtilde{S}}, Θ_{\wtilde{S}})$ is clearly semistable with respect to
$\wtilde H$.

Recall from \cite[Thm.~1.1]{Takayama2003} that the natural map of fundamental
groups, $π_1(\wtilde X) → π_1(X)$ is isomorphic.  Together with
Remark~\ref{rem:blackhorse}, this implies that $π_1(\wtilde S) → π_1(\wtilde X)$
is surjective, and induces an isomorphism of profinite completions.
Item~\ref{il:garfunkel} of Corollary~\ref{cor:pfccs} therefore allows us to find
a Higgs bundle $(ℱ_{\wtilde X}, Θ_{\wtilde X})$ on $\wtilde X$ that restricts to
$(ℱ_{\wtilde S}, Θ_{\wtilde{S}})$, and is hence induced by a \pCVHS\ owing to
Corollary~\ref{cor:pfccs}, Item~\ref{il:simon}.  We have seen in
Remark~\ref{rem:charPCVHS} that all Chern classes of $ℱ_{\wtilde X}$ vanish.

Corollary~\ref{cor:higgsfromdownst} implies that $ℱ_{\wtilde X}$ comes from $X$.
More precisely, there exists a locally free sheaf $ℱ_X$ on $X$ such that
$ℱ_{\wtilde X} = π^* ℱ_X$.  The restriction $ℱ_X|_S$ agrees with
$ℱ_S = \sEnd ℰ_X|_S$, which together with the observation on the Chern classes
of $ℱ_{\wtilde X}$ made above implies that $ℱ_X$ is a member of the family
${\sf B}_r$ that was introduced in Claim~\vref{claim:flatbounded}.
Item~\ref{il:dappledhorse} thus gives an isomorphism $\sEnd ℰ_X ≅ ℱ_X$,
showing that $\sEnd ℰ_X$ is locally free.  But $\sEnd ℰ_X$ contains $𝒯_X$ as a
direct summand.  It follows that $𝒯_X$ is locally free and thus $X$ is smooth by
the solution of the Zariski-Lipman problem for klt spaces,
\cite[Thm.~6.1]{GKKP11}.  \qed

\subsection{Proof of Theorem~\ref*{thm:more_general_uniformisation}}
\label{subsect:proof_of_general_uniformisation}
\approvals{
  Behrouz & yes \\
  Daniel & yes \\
  Stefan & yes \\
  Thomas & yes
}

By Proposition~\ref{prop:c1} we know that the variety $Z$ is smooth in
codimension two.  Now, let $γ: Y → Z$ be a quasi-étale, Galois cover such that
$\what{π}_1(Y_{\reg}) ≅ \what{π}_1(Y)$.  By \cite[Thm.~1.14]{GKP13}, such a
cover exists.  Since $γ$ branches only over the singular set of $Z$, it follows
from \cite[Prop.~5.20]{KM98} that $Y$ is still klt and smooth in codimension
two.  Since $γ$ is finite, the $ℚ$-Cartier divisor $K_Y = γ^* K_Z$ is still
ample.  Moreover, as both $f^{[*]}𝒯_X$ and $𝒯_Y$ are reflexive and agree on the
big open set of $Y$ where $γ$ is étale, we conclude that $f^{[*]}𝒯_X = 𝒯_Y$.
Consequently, Lemma~\ref{lem:buqec} guarantees that equality holds in the
$ℚ$-Miyaoka-Yau Inequality for $Y$.  Proposition~\ref{prop:c2} hence applies and
$Y$ is smooth.  We may thus use the classical uniformisation theorem of
Yau~\cite[Rem.~(iii) on p.~1799]{MR0451180} to conclude that $Y$ is a ball
quotient, as claimed.  \qed

\approvals{
  Behrouz & yes \\
  Daniel & yes \\
  Stefan & yes \\
  Thomas & yes
}
\begin{rem}[Comparison with the torus-quotient case]\label{rem:CompareFlat}
  Let us now briefly explain the difference between the above strategy and those
  that appear in the proof of the uniformisation theorem in the case of
  vanishing Chern classes, \cite{GKP13} and \cite{LT14}.  For simplicity, we
  assume that $X$ is smooth in codimension two with only klt singularities and
  that $\what{π}_1(X_{\reg})≅ \what{π}_1(X)$.

  In the setting where $c_1(X)· H^{n-1}=0$ and $c_2(X)· H^{n-2}=0$ for some
  ample divisor $H$, one uses the (slope) semistability of $𝒯_X|_S$, where $S$
  is a sufficiently general, complete intersection, smooth, projective surface
  determined by $H$, to construct a (holomorphic) flat connection
  $$
  ∇ : 𝒯_X|_S → Ω¹_S ⊗ 𝒯_X|_S.
  $$
  Here $∇$ is compatible with the holomorphic structure of $𝒯_X|_S$, that is
  $∇^{0,1}$ is defined by the holomorphic structure of $𝒯_X|_S$.  After
  extending the linear representation $π_1(S)$ corresponding to $∇$ to a
  representation of $π_1(X)$, one can construct a flat locally free analytic
  sheaf $ℱ$ on $X$ verifying the isomorphism $ℱ|_S≅ 𝒯_X|_S$, as analytic
  sheaves.
  
  On the other hand, when $K_X$ is ample and the equality in the Miyaoka-Yau
  inequality is attained, the holomorphic structure of the harmonic bundle
  $\sEnd(ℰ_X|_S, θ_X|_S)$ defined above is different from the one given by the
  representation of $π_1(S)$ associated to the underlying flat connection.  This
  is simply because the $(0,1)$ part of the HYM connection on
  $\sEnd(ℰ_X|_S, θ_X|_S)$ is of the form
  $$
  \bar{∂} + (θ_X|_S)^h,
  $$
  where $\bar{∂}$ is the holomorphic structure of $\sEnd(ℰ_X|_S)$, $h$ is the
  harmonic metric and $(θ_X|_S)^h$ is the conjugate of the Higgs field $θ_X|_S$
  with respect to $h$.  As a result, the argument in the torus-quotient case
  breaks down: If one naively extends the representation of $π_1(S)$ defined by
  $\sEnd(ℰ_X|_S)$ to a representation $ρ$ of $π_1(X)$, the flat analytic sheaf
  $ℱ$ constructed from $ρ$ does not satisfy the isomorphism
  $\sEnd(ℰ)|_S≅ ℱ|_S$, as analytic sheaves; the holomorphic structures are
  simply not compatible.
\end{rem}

%
%
\svnid{$Id: 09-BallQuotients.tex 821 2017-10-26 11:37:11Z peternell $}

\section{Characterisation of singular ball quotients}
\label{sect:ball-quotient}
\subversionInfo
\approvals{
  Behrouz & yes \\
  Daniel & yes \\
  Stefan & yes \\
  Thomas & yes
}

In this section, we prove Theorem~\ref{thm:ball-quotient-II} and
Corollary~\ref{cor:smoothparthyperbolic}, and concerning optimality of our
results discuss an example of a singular ball quotient in
Section~\ref{ssec:keum}.  First, we recall a few standard definitions and
elementary properties.  Throughout the present section, all complex spaces will
be reduced and are assumed to have a countable basis of topology.

\begin{defn}[Properly discontinuous action]\label{def:pda}
  Let $X$ be a complex space, and $Γ$ a group of holomorphic automorphisms of
  $X$.  We say that $Γ$ acts \emph{properly discontinuously} on $X$, if for any
  points $x,y ∈ X$, there exist neighbourhoods $U = U(x)$ and $V=V(y)$ such that
  the set $\{g ∈ Γ \mid g· U ∩ V ≠ ∅\} ⊂ Γ$ is finite.
\end{defn}

\begin{rem}
  Note that there exist several, not necessarily equivalent definitions of
  ``properly discontinuous'' in the literature, especially in a purely
  topological context.  We follow \cite[Sect.~2.1]{MR1756407}, where the
  terminology ``discrete group of transformations'' is used for the same
  concept.  A further general reference is \cite[Chap.~12]{MR2766102}.
\end{rem}

\begin{lem}[Criteria for actions to be properly discontinous]\label{lem:discrete}
  Let $Γ$ be a subgroup of $\Aut_{𝒪}(𝔹^n) = \PSU(1, n)$.  Then, the following
  statements are equivalent.
  \begin{enumerate}
  \item The group $Γ$ acts properly discontinuously on $𝔹^n$.
  \item The group $Γ$ is discrete in $\PSU(1, n)$.
  \item Every $Γ$-orbit in $𝔹^n$ is a discrete subset of $𝔹^n$, and for every
    $z ∈ 𝔹^n$ the isotropy group $Γ_z = \{γ ∈ Γ \mid γ· z = z\}$ is finite.
  \end{enumerate}
\end{lem}
\begin{proof}
  This is classical, see for example \cite[Sect.~2.1]{MR1756407}, or
  \cite[Sect.~2.2]{MR1177168} for the prototypical case $n=1$.
\end{proof}

\subsection{Proof of Theorem~\ref*{thm:ball-quotient-II}}
\label{ssec:pfthmbq2}
\CounterStep
\approvals{
  Behrouz & yes \\
  Daniel & yes \\
  Stefan & yes \\
  Thomas & yes
}

We will prove the implications \ref{il:z2} $⇒$ \ref{il:z3} $⇒$ \ref{il:z1} $⇒$
\ref{il:z2} separately.

\subsubsection*{\ref{il:z2} $⇒$ \ref{il:z3}}
\approvals{
  Behrouz & yes \\
  Daniel & yes \\
  Stefan & yes \\
  Thomas & yes
}

As $G$ is a finite group, and as $Y$ is projective and smooth, $X$ is
projective.  Moreover, it follows from the assumptions on the $G$-action that
$f: Y → X$ is quasi-étale.  This implies that $K_X$ is $ℚ$-Cartier, that $X$ is
klt, and that $K_Y = f^* K_X$.  Moreover, by the same argument as in the first
paragraph of Section~\ref{subsect:proof_of_general_uniformisation} we have
$𝒯_Y = f^{[*]}𝒯_X$.

Now, recall that $K_Y$ is ample, and that the Chern classes of $𝒯_Y$ satisfy the
Miyaoka-Yau equality, see e.g.~\cite[(8.8.3)]{Kollar95s}.  It follows that $K_X$
is ample.  The $ℚ$-Miyaoka-Yau equality for $𝒯_X$ then follows from
Lemma~\ref{lem:buqec}.

\subsubsection*{\ref{il:z3} $⇒$ \ref{il:z1}}
\approvals{
  Behrouz & yes \\
  Daniel & yes \\
  Stefan & yes \\
  Thomas & yes
}

Let $f:Y→ X$ be the finite, Galois, quasi-étale morphism from a ball quotient
$Y$ to $X$ guaranteed by Theorem~\ref{thm:BQ}.  Let $G$ be the Galois group of
$f: Y → X$ and define $\wtilde{π}: 𝔹^n → X$ as $\wtilde{π} = f ◦ π$, where
$π: 𝔹^n → Y$ is the universal cover of $Y^{an}$.  Let $Γ := π_1(Y^{an})$ be the
deck transformation group of $π$.  Then, the restriction of $\wtilde{π}$ to
$\wtilde U := \wtilde{π}^{-1}(X_{\reg}^{an})$ is a topological covering map,
which we call $\wtilde{π}_{\reg}$.  Additionally, as the codimension of
$\wtilde U$ in the manifold $𝔹^n$ is more than two, $\wtilde U$ is simply
connected.  Consequently,
$π_{\reg}:= π|_{\wtilde U}: \wtilde U → f^{-1}(X_{\reg}^{an})$ and
$\wtilde{π}_{\reg}$ are universal covering maps.  It follows that
$\what{Γ} = π_1(X_{\reg}^{an})$ acts on $\wtilde U$ by holomorphic
automorphisms, and the action is properly discontinuous and fixed-point free.
As $Γ = π_1(Y^{an}) = π_1(f^{-1}(X_{\reg})^{an})$, and since
$f^{-1}(X_{\reg})/ G = X_{\reg}$, we have an exact sequence of groups
\begin{equation}\label{eq:fundamentalgroupsequence}
  1 → Γ → \what{Γ} → G → 1,
\end{equation}
and the action of $\what{Γ}$ on $\wtilde U$ extends the action of $Γ$ on
$\wtilde U$.  Our situation can hence be summarised in the following commutative
diagram,
\begin{equation}
  \begin{gathered}\label{eq:bigdiagram}
    \xymatrix{ %
        \wtilde U \ar@{^(->}[d] \ar[rrr]_(.4){π_{\reg} \text{, quot.\ by }Γ} \ar@/^5mm/[rrrrr]^{\wtilde{π}_{\reg}\text{, quot.\ by }\what{Γ}} &&& f^{-1}(X_{\reg})^{an} \ar@{^(->}[d] \ar[rr]_{\text{quot.\ by }G} && X_{\reg}^{an} \ar@{^(->}[d]\\
        𝔹^n \ar[rrr]^{π\text{, quot.\ by }Γ} &&& Y^{an} \ar[rr]^{f^{an}\text{, quot.\ by }G} && X^{an}.  }
  \end{gathered}
\end{equation}
As the inclusion $\wtilde U \hookrightarrow 𝔹^n$ realises $𝔹^n$ as the envelope
of holomorphy of $\wtilde U$, the action of $\what{Γ}$ on $\wtilde U$ uniquely
extends to a holomorphic action of $\what{Γ}$ on $𝔹^n$, see
\cite[Lem.~4.1]{MR3170714}.  This extended action is fixed-point free in
codimension two by construction.  It now follows from the exact
Sequence~\eqref{eq:fundamentalgroupsequence} and from
Diagram~\eqref{eq:bigdiagram} that the topological quotient
$𝔹^n/\what{Γ} \simeq (𝔹^n/Γ)/G$ is homeomorphic to $X^{an}$, and therefore
Hausdorff.  As $𝔹^n$ and $X^{an}$ are both normal complex spaces, and as we
already know that $X_{\reg}^{an}$ is biholomorphic to $\wtilde U /\what{Γ}$,
\cite[Satz on p.~328]{MR0150789} hence implies that $𝔹^n/\what{Γ}$ is in a
natural way a normal complex space, which is in fact biholomorphic to $X^{an}$.
In particular, $\wtilde{π}: 𝔹^n → X^{an}$ is the quotient map for the
$\what{Γ}$-action.  To conclude the proof, we will show that this action is
properly discontinuous.

As $\wtilde{π}$ is holomorphic, for every $z ∈ 𝔹^n$ the fibre
$\wtilde{π}^{-1} \bigl(\wtilde{π} (z)\bigr) = \what{Γ}·z$ is a zero-dimensional
analytic, and hence discrete, subset of $𝔹^n$.  Moreover, we claim that all
isotropy groups $\what{Γ}_z$ of points $z ∈ 𝔹^n$ are finite.  From this, it will
follow that the $\what{Γ}$-action is properly discontinuous, see
Lemma~\ref{lem:discrete}.  So, suppose that there is a point $z_0 ∈ 𝔹^n$ such
that $Γ_{z_0}$ is infinite.  As the isotropy of ${z_0}$ in the full automorphism
group $\PSU(1, n)$ is compact, $\what{Γ}_{z_0}$ is not a discrete subgroup of
$\PSU(1, n)$, i.e., there exists a sequence of elements $γ_n ∈ \what{Γ}_{z_0}$
converging to the identity element, cf.\ \cite[p.~7]{MR1756407}.  Now, if $z_1$
is any point in $\wtilde{U}$, where the $\what{Γ}$-action is free, it follows
that $\what{Γ}· z_1 = \wtilde{π}_{\reg}^{-1}\bigl(\wtilde{π}_{\reg}(z_1)\bigr)$
is not discrete, a contradiction.

\subsubsection*{\ref{il:z1} $⇒$ \ref{il:z2}}
\approvals{
  Behrouz & yes \\
  Daniel & yes \\
  Stefan & yes \\
  Thomas & yes
}

Recall that compact quotients of $𝔹^n$ by discrete subgroups of $\PSU(1, n)$ are
projective algebraic, see e.g.~\cite{MR0068872}.  Let
$\what{π}: 𝔹^n → X= 𝔹^n/\what{Γ}$ be the quotient map.  As the action of
$\what{Γ}$ is fixed-point free in codimension two, the restriction
$\what{π}|_{\what{π}^{-1}(X_{\sing})}$ is unramified and hence a topological
covering map.  Moreover, the preimage $\what{π}^{-1}(X_{\sing})$ has complement
of complex codimension at least three in the smooth manifold $𝔹^n$, and is
therefore simply connected.  As $X_{\reg}$ is (the complex space associated
with) a quasi-projective algebraic variety, its fundamental group, which is
isomorphic to $\what{Γ}$, is finitely generated.  It therefore follows from
Selberg's Lemma, e.g.~see \cite{MR925989}, that $\what{Γ}$ has a normal subgroup
$Γ$ of finite index that acts without fixed points on $𝔹^n$.  From this, we
obtain the following factorisation of the $\what{Γ}$-quotient map:
$$
𝔹^n \longrightarrow 𝔹^n/Γ \overset{f}{\longrightarrow} 𝔹^n/\what{Γ}= X.
$$
Here, $f$ is the quotient for the action of the finite group $G:= \what{Γ} / Γ$
on the projective manifold $Y:= 𝔹^n/Γ$, which by the assumption on the
$\what{Γ}$-action is fixed-point free in codimension two.  It follows that $f$
is quasi-étale.  \qed

\subsection{Proof of Corollary~\ref*{cor:smoothparthyperbolic}}
\approvals{
  Behrouz & yes \\
  Daniel & yes \\
  Stefan & yes \\
  Thomas & yes
}

If $X$ is a singular ball quotient, let $π: 𝔹^n → X$ be the quotient map for the
corresponding discrete group action.  Then,
$$
π|_{π^{-1}(X_{\reg})}: π^{-1}(X_{\reg}) → X_{\reg}
$$
is an unramified covering map.  By \cite[Prop.~3.2.2(1)]{KobayashiGrundlehren}
the manifold $π^{-1}(X_{\reg}) ⊂ 𝔹^n$ is Kobayashi-hyperbolic, as it is
contained in the $n$-dimensional polydisk $\bD⨯ ⋯ ⨯ \bD$, which is
Kobayashi-hyperbolic by \cite[Prop.~3.2.3]{KobayashiGrundlehren}.  Hence, the
statement follows from \cite[Thm.~3.2.8(2)]{KobayashiGrundlehren}.  \qed

\subsection{Further comments on Corollary~\ref*{cor:smoothparthyperbolic}}
\label{subsect:hyperbolicitycomments}
\approvals{
  Behrouz & yes \\
  Daniel & yes \\
  Stefan & yes \\
  Thomas & yes
}

Let $π: 𝔹^n → X$ be a singular ball quotient.  Then, $X$ has slightly more
general hyperbolicity properties than those stated in
Corollary~\ref*{cor:smoothparthyperbolic}, as we will explain now.  Let
$d_{𝔹^n}: 𝔹^n ⨯ 𝔹^n → ℝ^{≥ 0}$ be the Kobayashi distance on the ball.  Then, we
can define a natural distance on $X$ as follows: if $p, q ∈ X$, and if
$\wtilde{p} ∈ 𝔹^n$ satisfies $π(\wtilde{p}) = p$, we set
$$
d'_X:= \inf_{\wtilde{q}} d_{𝔹^n}(\wtilde{p}, \wtilde{q}),
$$
where the infimum runs over all points $\wtilde{q} ∈ 𝔹^n$ such that
$π(\wtilde{q}) = q$.  In fact, analogous to the Kobayashi pseudodistance, $d'_X$
can be defined using chains of locally liftable holomorphic maps from the unit
disc $\bD ⊂ ℂ$ to $X$, see \cite[p.~101]{MR2194466}.  Here, a holomorphic map
$f$ from a complex space $Z$ into $X$ is called \emph{locally liftable} if every
point $z∈ Z$ has an analytically open neighbourhood $U$ such that $f|_{U}$
factors via $π$.  As $𝔹^n$ is Kobayashi-hyperbolic, $d'_X$ is indeed a distance,
see \cite[Chap.~VII, Prop.~6.3]{MR2194466}.  It follows that every locally
liftable holomorphic map from $ℂ$ to $X$ is constant.  This property does not
imply that $X$ is Kobayashi-hyperbolic, see the subsequent subsection for an
example.  However, many of the properties known for holomorphic maps into
Kobayashi-hyperbolic manifolds hold for locally liftable holomorphic maps into
$X$.  At this time, we are not aware of any singular ball quotient with
\emph{canonical} singularities that fails to be Kobayashi-hyperbolic.

\subsection{Keum's singular ball quotient}
\label{ssec:keum}
\CounterStep
\approvals{
  Behrouz & yes \\
  Daniel & yes \\
  Stefan & yes \\
  Thomas & yes
}

The following example illustrates three points:
\begin{enumerate}
\item The fundamental group of singular ball quotients might be trivial.
\item Kobayashi-hyperbolicity in general will not extend over klt singularities.
\item The resolution of a klt singular ball quotient $X$ might have a geometry
  that is very different from $X$.
\end{enumerate}

Keum found a two-dimensional ball quotient $Y$ together with an order $7$
automorphism $g$ that acts with isolated fixed points on $Y$ such that the
minimal resolution $π: \wtilde X → X$ of the quotient $X = Y/\langle g \rangle$
is simply connected, of Kodaira dimension one, and admits an elliptic fibration
$η: \wtilde X → C$, see \cite[Thm.~1.1(2) and Prop.~2.4]{MR2443971} and
\cite{MR2239523}.  The general fibre $F$ of $η$ is an elliptic curve.  Composing
the universal covering map $f: ℂ → F$ with $π$ yields a non-constant (not
locally liftable) holomorphic map from $ℂ$ to $X^{an}$, which is therefore not
Kobayashi-hyperbolic.  On the other hand, note that the smooth locus
$X_{\reg}^{an}$ is hyperbolic by the proof of
Corollary~\ref{cor:smoothparthyperbolic}.  As $\wtilde X$ is simply connected
and as the singularities of $X$ are klt, $X$ itself is also simply connected,
while its smooth locus has infinite fundamental group.  Note that the
singularities of $X$ are worse than canonical, as the resolution $\wtilde X$ is
not of general type.  Hence, $X$ is certainly not the minimal model of any
smooth projective variety of general type.

%
%
\svnid{$Id: 10-directions.tex 835 2018-01-30 10:12:11Z greb $}

\section{Further directions}\label{subsect:furtherdirections}
\subversionInfo

\subsection{The general klt case}
\approvals{
  Behrouz & yes \\
  Daniel & yes \\
  Stefan & yes \\
  Thomas & yes
}

While we established the Miyaoka-Yau inequality, Theorem~\ref{thm:MYinequality},
for a general klt variety with big and nef canonical divisor, the uniformisation
theorem, Theorem~\ref{thm:BQ}, assumes the variety to be smooth in codimension
two.  This is used crucially in its proof when we construct a complex variation
of Hodge structures on the restriction of the natural Higgs sheaf to the
\emph{smooth} complete intersection surface and subsequently extend this
variation of Hodge structure to a strong resolution of singularities, into which
the complete intersection surface embeds.

The authors are currently working towards removing this additional assumption by
comparing the Nonabelian Hodge correspondence on a resolution of a given klt
variety $X$ to Mochizuki's Nonabelian Hodge correspondence, \cite{MR2310103}, on
the smooth part of $X$.  An alternative approach consists in proving that a
generalisation of the Nonabelian Hodge correspondence to the case of smooth
projective DM-stacks, as discussed in \cite{MR2918179}, applies to the
restriction of the canonical Higgs sheaf to a complete intersection surface, and
then again in comparing with a resolution.

\subsection{The case of klt pairs}
\approvals{
  Behrouz & yes \\
  Daniel & yes \\
  Stefan & yes \\
  Thomas & yes
}

In \cite[Thm.~B]{GT16}, partially building upon the foundational work done in
the present paper, the following generalisation of
Theorem~\ref{thm:MYinequality} has been proved.

\begin{thm}\label{thm:OrbiMY}
  Let $(X, D)$ be a projective, klt pair, where $D$ is of the form
  $D := \sum (1 -\frac{1}{a_i})D_i$ for prime divisors $D_{•}$ and
  positive integers $a_{•}$.  Assume that $K_X + D$ is ample.  Writing
  $\check{c}_1$, $\check{c}_2$ for the relevant orbifold Chern classes for
  pairs, \cite[Sect.~2]{GT16}, the following ``Miyaoka-Yau inequality'' is
  satisfied,
  $$
  \bigl(2(n+1)· \check{c}_2 (X,D) - n · \check{c}_1²(X,D)\bigr)·[K_X+D]^{n-2}≥
  0.  \eqno\qed
  $$
\end{thm}

Based on this result, one expects the following generalisation of our result on
uniformisation, Theorem~\ref{thm:BQ}.

\begin{expectation}\label{conj:orbi}
  In the setting of Theorem~\ref{thm:OrbiMY}, if the equality is achieved, then
  $X$ has only quotient singularities and $(X,D)$ admits a (global) smooth
  Deligne-Mumford stack structure $\mathcal X$ whose universal cover (in the
  sense of Deligne-Mumford stacks) is the ball.
\end{expectation}

We note that this in particular implies that the local isotropy groups of
$\mathcal X$ act trivially in codimension one with the exception of the divisors
$D_i$ along which the isotropy groups are isomorphic to $ℤ/(a_i ℤ)$.  We
conclude by pointing out that the Expectation was confirmed for surface pairs
$(X, D)$ by Kobayashi, Nakamura, and Sakai in~\cite[Thm.~12]{MR1030189}.

\Preprint{
\part{Appendices}
\appendix

%
%
\svnid{$Id: 0A-restriction.tex 838 2018-04-25 11:39:59Z kebekus $}

\section{The restriction theorem for sheaves with operators}
\label{sect:app-oper}
\subversionInfo

\subsection{Generalised Bogomolov-Gieseker inequalities}
\label{ssec:langer1}
\approvals{
  Behrouz & yes \\
  Daniel & yes \\
  Stefan & yes \\
  Thomas & yes
}

As a preparation for the proof of the restriction theorem in
Section~\ref{ssec:restrSWO}, we establish a technical Bogomolov-Gieseker type
inequality for sheaves with operators.  We are grateful to Adrian Langer who
explained much of the content of Sections~\ref{ssec:langer1} and
\ref{ssec:restrSWO} to us, and allowed us to reproduce his ideas here.  Some
parts of the proofs are variations of arguments found in Langer's papers.

As before, we have not tried to formulate the strongest result possible.  In
contrast to the setting of Section~\ref{sect:Higgs}, we can restrict ourselves
to the traditional setting of torsion free sheaves $(ℰ,θ)$ with $𝒲$-valued
operators, where $𝒲$ is locally free.  This will allow us to quote Langer's
paper \cite{MR1954067}, to discuss the Harder-Narasimhan filtration of
$(ℰ, θ)$, and to write $μ^{\max}_H(ℰ,θ)$ and $μ^{\min}_H(ℰ,θ)$ in
Corollary~\ref{cor:ncl1}.

\begin{prop}[Generalised Bogomolov-Giesecker inequality I]\label{prop:nli}
  Let $X$ be a smooth, projective variety of dimension $n ≥ 2$, let $H$ be a big
  and nef divisor on $X$ and $𝒲$ be a locally free sheaf.  Let $(ℰ, θ)$ be a
  torsion free sheaf with a $𝒲$-valued operator, where $(ℰ, θ)$ is
  semistable with respect to $H$.  Then,
  $$
  Δ(ℰ)·[H]^{n-2} ≥ -\frac{r⁴}{4d}·μ^{\max}_H(𝒲)², \quad \text{where $d := [H]^n$ and $r := \rank ℰ$.}
  $$
\end{prop}
\begin{proof}
  This is an immediate consequence of \cite[Prop.~7.2]{MR1954067} where a
  slightly stronger result is shown.  Observe that Definition~\ref{def:nshfop1}
  agrees with \cite[Def.~1.1]{MR1954067} because $𝒲$ is assumed to be locally
  free.
\end{proof}

\begin{cor}[Generalised Bogomolov-Gieseker inequality II]\label{cor:ncl1}
  Let $X$ be a smooth, projective variety of dimension $n ≥ 2$, let $H$ be a big
  and nef divisor on $X$, and $𝒲$ be a locally free sheaf.  Let $(ℰ, θ)$ be
  a torsion free sheaf with a $𝒲$-valued operator.  Then,
  $$
  Δ(ℰ)·[H]^{n-2} ≥ -\frac{r⁴}{4d}· μ^{\max}_H(𝒲)² - \frac{r²}{d}·δ(ℰ),
  $$
  where
  $δ(ℰ) := \bigl( μ^{\max}_H(ℰ,θ)-μ_H(ℰ) \bigr)·\bigl(
  μ_H(ℰ)-μ^{\min}_H(ℰ,θ) \bigr)$, where $d = [H]^n$ and $r := \rank(ℰ)$.
\end{cor}
\begin{proof}
  Consider the Harder-Narasimhan filtration of $(ℰ, θ)$ in the category of
  sheaves with a $𝒲$-valued operator,
  $$
  0 = ℱ_0 ⊊ ℱ_1 ⊊ ⋯ ⊊ ℱ_{ℓ} = ℰ.
  $$
  To keep the notation reasonably short, set
  $$
  ℱⁱ := \factor{ℱ_i}{ℱ_{i-1}}, \quad r_i := \rank ℱⁱ, \quad μ_i := μ_H(ℱⁱ).
  $$
  By \cite[eq.~(7.3)]{MR2665168}, we can express the Bogomolov discriminant of
  $ℰ$ in terms of the discriminants of the $ℱⁱ$ as follows,
  \begin{equation}\label{eq:n3a}
    Δ(ℰ) = r·\sum_{i=1}^{ℓ} \frac{1}{r_i}·Δ(ℱⁱ) - \sum_{i<j} r_ir_j \left( \frac{1}{r_i}·c_1(ℱⁱ) - \frac{1}{r_j}·c_1(ℱ^j) \right)²
  \end{equation}

  The quotients $ℱⁱ$ inherit $𝒲$-valued operators $θⁱ$ that make
  $(ℱⁱ, θⁱ)$ semistable with respect to $H$.  In particular,
  Proposition~\ref{prop:nli} gives an estimate for the first summand in the
  right-hand side of \eqref{eq:n3a}, after taking the product with $[H]^{n-2}$
  \begin{align}
    r·\sum_{i=1}^{ℓ} \frac{1}{r_i}·Δ(ℱⁱ)·[H]^{n-2} & ≥ -r·\sum_{i=i}^{ℓ} \frac{r_i³}{4d} · μ^{\max}_H(𝒲)² && \text{Prop.~\ref{prop:nli}} \label{eq:nAx} \\
                                                                     & ≥ - \frac{r⁴}{4d}·μ^{\max}_H(𝒲)².  \notag
  \end{align}
  To discuss the second summand in \eqref{eq:n3a}, set
  $M_i := \frac{1}{r_i}·c_1(ℱ_i)$.  Observe that $μ_i = M_i·[H]^{n-1}$ and
  recall from the Hodge index theorem that for two divisor classes $α$ and $β$
  on $X$ with $α$ nef, we have
  $\bigl( β·α^{n-1} \bigr)² ≥ α^n· \bigl(β²·α^{n-2} \bigr)$.  For $α = [H]$ and
  $β = M_i - M_j$, we hence obtain
  $$
  \sum_{i<j} r_ir_j \left( M_i-M_j\right)²·[H]^{n-2} ≤ \frac{1}{d}·\sum_{i<j}
  r_ir_j·(μ_i-μ_j)².
  $$
  An elementary calculation, carried out in
  \cite[Lem.~1.4]{Langer04a}\footnote{Note that $μ_i > μ_j$ for $i < j$ by
    definition of the Harder-Narasimhan filtration.}, gives
  $$
  \sum_{i < j} r_i r_j (μ_i - μ_j)² ≤ r² \bigl(μ_1 - μ_H(ℰ)\bigr)·\bigl(μ_H(ℰ) - μ_m\bigr),
  $$
  hence
  \begin{equation}\label{eq:nBx}
    \sum_{i<j} r_ir_j \left( M_i-M_j\right)²·[H]^{n-2} ≤ \frac{r²}{d}·δ(ℰ).
  \end{equation}
  Combining Inequalities~\eqref{eq:nAx} and \eqref{eq:nBx} with the description
  of the Bogomolov discriminant found in \eqref{eq:n3a} ends the proof of
  Corollary~\ref{cor:ncl1}.
\end{proof}

\subsection{Restriction theorem for sheaves with operators}
\label{ssec:restrSWO}
\approvals{
  Behrouz & yes \\
  Daniel & yes \\
  Stefan & yes \\
  Thomas & yes
}

The following restriction theorem of Mehta-Ramanathan type is a main technical
tool for the proof of the restriction theorem for (singular) Higgs sheaves in
Section~\ref{ssect:restrict}, and hence of the $ℚ$-Miyaoka-Yau type inequality
that we will establish in Section~\ref{sect:MY}.  Again, we can restrict
ourselves to the traditional setting of sheaves with operators that take values
in a locally free sheaf.

\begin{thm}[Restriction theorem for sheaves with operators]\label{thm:nrestrSWO}
  Let $X$ be a smooth, projective variety of dimension $n ≥ 2$, let $H$ be a big
  and nef divisor on $X$, and $𝒲$ be a locally free sheaf on $X$.  Let
  $(ℰ, θ)$ be a torsion free sheaf with a $𝒲$-valued operator, where
  $(ℰ, θ)$ is stable with respect to $H$.  Set $r := \rank(ℰ)$ and
  $d := [H]^n$.  Assume further that we are given a number $m ∈ ℕ^+$ with
  \begin{equation}\label{eq:prop:li}
    m > r·Δ(ℰ)·[H]^{n-2} + \frac{r⁵}{4d} · μ^{\max}_H(𝒲)²
  \end{equation}
  and a hypersurface $D ∈ |m·H|$ such that the following holds.
  \begin{enumerate}
  \item The hypersurface $D$ is irreducible, normal, and not contained in
    $\supp H$.
  \item The restriction $ℰ|_D$ is torsion free.
  \end{enumerate}
  Then, $(ℰ|_D, θ|_D)$ is stable with respect to $H|_D$.
\end{thm}
\begin{proof}
  Argue by contradiction and assume that $(ℰ|_D, θ|_D)$ is \emph{not} stable
  with respect to $H|_D$.  Then, there exists a maximal, $θ|_D$-invariant,
  saturated, destabilising subsheaf $ℱ_D ⊆ ℰ|_D$.  Set
  $$
  𝒬 := \factor{ℰ|_D}{ℱ_D}, \quad ρ := \rank_D 𝒬 \quad \text{and} \quad
  ℰ' := \ker \bigl( ℰ → ℰ|_D → 𝒬 \bigr).
  $$
  Observe that $μ_{H|_D}(𝒬) ≤ μ_{H|_D}(ℰ|_D)$.  Also, observe that $ℰ'$ is a
  $θ$-invariant subsheaf of $ℰ$.  Since $𝒲$ is locally free, the operator $θ$
  induces a $𝒲$-valued operator $θ'$ on $ℰ'$, cf.\ Warning~\ref{war:noois}, and
  we can consider the Harder-Narasimhan filtration of $(ℰ', θ')$,
  $$
  0 = ℱ'_0 ⊊ ℱ'_1 ⊊ ⋯ ⊊ ℱ'_{ℓ'} = ℰ'
  $$
  In the following, we aim to compute the main invariants of $(ℰ', θ')$.

  To begin, Chern class computations analogous to \cite[Prop.~5.2.2 and proof of
  Thm.~7.3.5]{MR2665168}, yields the following
  \begin{align}
    \label{il:nCC-c1} [ℰ'] & = [ℰ] - ρm·[H] \\
    \label{il:nCC-c2} Δ(ℰ')·[H]^{n-2} & = Δ(ℰ)·[H]^{n-2}-m²ρ(r-ρ)·[H]^n \\
                            & \qquad\qquad\qquad\qquad\qquad + 2rρ·\bigl(μ_{H|_D}(𝒬)-μ_{H|_D}(ℰ|_D)\bigr) \notag\\
                            & ≤ Δ(ℰ)·[H]^{n-2} - m²dρ(r-ρ).\notag
  \end{align}

  Secondly, observing that $ℰ'$ is a proper subsheaf of $ℰ$ whose slope is
  strictly smaller than that of $ℰ$ by \eqref{il:nCC-c1}, it follows from
  stability of $(ℰ,θ)$ that $μ^{\max}_H(ℰ', θ')$ is strictly smaller than
  $μ_H(ℰ)$, the difference being at least $1/r²$.  In particular,
  \begin{equation}\label{eq:ntz-A}
    μ^{\max}_H(ℰ', θ') - μ_H(ℰ') = \frac{ρ}{r}md + μ^{\max}_H(ℰ', θ') - μ_H(ℰ) ≤ \frac{ρ}{r}md - \frac{1}{r²}.
  \end{equation}

  Thirdly, consider $(ℰ'', θ'') := (ℰ,θ) ⊗ 𝒪_X(-D)$, which is stable with
  respect to $H$ by Lemma~\ref{lem:stabtp}.  On the other hand, $ℰ''$ is a
  $θ$-invariant subsheaf of $ℰ$, and admits a generically surjective morphism to
  $\factor{ℱ'_{ℓ'}}{ℱ'_{ℓ'-1}}$.  It follows that $μ_H(ℰ'') ≤ μ^{\min}_H(ℰ',θ')$
  and
  \begin{equation}\label{eq:ntz-B}
    μ_H(ℰ') - μ^{\min}_H(ℰ', θ') ≤ \frac{r-ρ}{r}md + μ_H(ℰ'') - μ^{\min}_H(ℰ',θ') ≤ \frac{r-ρ}{r}md.
  \end{equation}
  Combining \eqref{eq:ntz-A} and \eqref{eq:ntz-B}, we obtain in the notation of
  Corollary~\ref{cor:ncl1},
  \begin{equation}\label{eq:tz-C}
    δ(ℰ') ≤ \Bigl( \frac{ρ}{r}md - \frac{1}{r²} \Bigr)·\Bigl( \frac{r-ρ}{r}md \Bigr) ≤ \frac{1}{r²} \Bigl(m²d²ρ(r-ρ)-\frac{md}{r} \Bigr).
  \end{equation}
  Summing up, we have:
  \begin{align*}
    0 & ≤ d·Δ(ℰ')·[H]^{n-2} + \frac{r⁴}{4}μ^{\max}_H(𝒲)² + r²·δ(ℰ') && \text{Cor.~\ref{cor:ncl1}} \\
      & ≤ d·Δ(ℰ)·[H]^{n-2} - m²d²ρ(r-ρ) + \frac{r⁴}{4}μ^{\max}_H(𝒲)² + r²·δ(ℰ') &&\text{Ineq.~\eqref{il:nCC-c2}} \\
      & ≤ d·Δ(ℰ)·[H]^{n-2} + \frac{r⁴}{4}μ^{\max}_H(𝒲)² - \frac{md}{r} &&\text{Ineq.~\eqref{eq:tz-C}.}
  \end{align*}
  This contradicts the choice of $m$ in \eqref{eq:prop:li} and therefore ends
  the proof of Theorem~\ref{thm:nrestrSWO}.
\end{proof}

}

\vspace{0.3cm}

\end{document}